\documentclass[a4j,10pt,showkeys]{article}

\usepackage{amsmath,amsthm,amssymb}
\usepackage{geometry}
\usepackage{hyperref}
\usepackage[capitalize]{cleveref}
\usepackage{cancel}
\usepackage{mathtools}
\usepackage{bm}
\usepackage{color}
\usepackage{mathrsfs}
\usepackage[dvipsnames]{xcolor}
\hypersetup{
    colorlinks=true,
    citecolor=MidnightBlue,
    linkcolor=Red,
    urlcolor=OliveGreen,
}
\usepackage{tikz}
\usetikzlibrary{calc}
\tikzset{set label/.style={fill=white}}

\newtheorem{theo}{Theorem}[section]
\newtheorem{lemm}[theo]{Lemma}
\newtheorem{prop}[theo]{Proposition}
\newtheorem{cor}[theo]{Corollary}
\newtheorem{claim}{Claim}
\newtheorem*{result}{Main result}
\theoremstyle{definition}
\newtheorem{defi}[theo]{Definition}
\newtheorem{rem}[theo]{Remark}
\newtheorem{assum}{Assumption}
\newtheorem{exam}[theo]{Example}

\newcommand{\bE}{\mathbb{E}}
\newcommand{\bF}{\mathbb{F}}

\newcommand{\bN}{\mathbb{N}}

\newcommand{\bP}{\mathbb{P}}

\newcommand{\bR}{\mathbb{R}}

\newcommand{\cA}{\mathcal{A}}
\newcommand{\cB}{\mathcal{B}}
\newcommand{\cC}{\mathcal{C}}

\newcommand{\cF}{\mathcal{F}}

\newcommand{\cH}{\mathcal{H}}

\newcommand{\cK}{\mathcal{K}}

\newcommand{\cP}{\mathcal{P}}

\newcommand{\cV}{\mathcal{V}}

\newcommand{\sC}{\mathscr{C}}

\newcommand{\ep}{\varepsilon}
\newcommand{\diff}{\mathrm{d}}
\newcommand{\dmu}{\,\mu(\diff\theta)}

\newcommand{\supp}{{\mathrm{supp}\,\mu}}
\newcommand{\Law}{\mathrm{Law}}
\newcommand{\LP}{\mathrm{LP}}
\newcommand{\TV}{\mathrm{TV}}

\newcommand{\UE}{\mathrm{UE}}
\newcommand{\LG}{\mathrm{LG}}

\newcommand{\Hol}{\mathrm{Hol}}

\newcommand{\loc}{\mathrm{loc}}
\newcommand{\Bb}{\cB_\mathrm{b}(\cH_\mu)}
\newcommand{\Cb}{\cC_\mathrm{b}(\cH_\mu)}

\newcommand{\dual}[2]{{}_{\cV^*_\mu}\langle{#1},{#2}\rangle_{\cV_\mu}}
\newcommand{\op}{\mathrm{op}}

\newcommand{\relmiddle}[1]{\mathrel{}\middle#1\mathrel{}}
\newcommand{\1}{\mbox{\rm{1}}\hspace{-0.25em}\mbox{\rm{l}}}

\providecommand{\keywords}[1]{\textbf{Keywords:} #1}
\makeatletter

\@addtoreset{equation}{section}
\makeatother
\def\widebar{\accentset{{\cc@style\underline{\mskip10mu}}}}
\numberwithin{equation}{section}
\allowdisplaybreaks

\title{Weak well-posedness of stochastic Volterra equations with completely monotone kernels and non-degenerate noise}

\author{
Yushi Hamaguchi\footnote{Graduate School of Engineering Science, Department of Systems Innovation, Osaka University. 1-3, Machikaneyama, Toyonaka, Osaka, Japan. Email: \href{mailto:hmgch2950@gmail.com}{hmgch2950@gmail.com}}\ \footnote{The author was supported by JSPS KAKENHI Grant Number 22K13958.}
}

\begin{document}
\maketitle


\begin{abstract}
We establish weak existence and uniqueness in law for stochastic Volterra equations (SVEs for short) with completely monotone kernels and non-degenerate noise under mild regularity assumptions. In particular, our results reveal the regularization-by-noise effect for SVEs with singular kernels, allowing for multiplicative noise with H\"{o}lder diffusion coefficients. In order to prove our results, we reformulate the SVE into an equivalent stochastic evolution equation (SEE for short) defined on a Gelfand triplet of Hilbert spaces. We prove weak well-posedness of the SEE using stochastic control arguments, and then translate it into the original SVE.
\end{abstract}


\keywords
Stochastic Volterra equation; stochastic evolution equation; uniqueness in law; regularization by noise.






\section{Introduction}\label{intro}

We consider the following stochastic Volterra equation (SVE for short):
\begin{equation}\label{intro_eq_SVE}
	X_t=x(t)+\int^t_0K(t-s)b(X_s)\,\diff s+\int^t_0K(t-s)\sigma(X_s)\,\diff W_s,\ \ t>0,
\end{equation}
associated with a $d$-dimensional Brownian motion $W$. Here, the drift coefficient $b:\bR^n\to\bR^n$, the diffusion coefficient $\sigma:\bR^n\to\bR^{n\times d}$, the kernel $K:(0,\infty)\to[0,\infty)$ and the forcing term $x:(0,\infty)\to\bR^n$ are given. If the kernel $K$ and the forcing term $x$ are constants, say $K(t)=1$ and $x(t)=x_0$ for some $x_0\in\bR^n$, then the SVE \eqref{intro_eq_SVE} becomes a standard stochastic differential equation (SDE for short):
\begin{equation}\label{intro_eq_SDE}
	\begin{dcases}
	\diff X_t=b(X_t)\,\diff t+\sigma(X_t)\,\diff W_t,\ \ t>0,\\
	X_0=x_0.
	\end{dcases}
\end{equation}
More generally, if $K$ is the fractional kernel, i.e.\ $K(t)=\frac{1}{\Gamma(\alpha)}t^{\alpha-1}$ with exponent $\alpha\in(\frac{1}{2},1]$, and if the forcing term $x$ is of the form $x(t)=\frac{1}{\Gamma(\alpha)}t^{a-1}x_0$ for some $x_0\in\bR^n$, then the SVE \eqref{intro_eq_SVE} corresponds to a kind of time-fractional SDE (cf.\ \cite{SaKiMa87}). Here and after, $\Gamma(\alpha)$ denotes the Gamma function. SVEs provide suitable models of dynamics with hereditary properties, memory effects and roughness of the path (cf.\ \cite{BaBeVe11,GrLoSt90,JaRo16}), which cannot be described by standard SDEs. In particular, a rough volatility model in mathematical finance can be modelled by using an SVE of the form \eqref{intro_eq_SVE} with the fractional kernel $K(t)=\frac{1}{\Gamma(\alpha)}t^{\alpha-1}$ with exponent $\alpha$ strictly less than $1$ and a square-root type diffusion coefficient $\sigma$ (cf.\ \cite{EuFuRo18,JaRo16}). This motivates us to study SVEs with singular kernels and non-Lipschitz coefficients; we say that the kernel $K$ is singular if $\lim_{t\downarrow0}K(t)=\infty$, and regular if $\lim_{t\downarrow0}K(t)<\infty$. The analysis of SVEs with singular kernels is much more difficult than that of the regular kernels case and the standard SDEs case since the solutions are no longer Markovian or semimartingales in general. For these reasons, we cannot apply fundamental tools established in the literature of Markov or semimartingale theory to SVEs directly. There are still many important open problems on well-posedness in the setting of singular kernels and non-Lipschitz coefficients.

In this paper, we are interested in \emph{weak well-posedness}, that is, \emph{existence of a weak solution} and \emph{uniqueness in law} for SVEs with completely monotone kernels and non-Lipschitz coefficients. By a weak solution of the SVE \eqref{intro_eq_SVE}, we mean a tuple $(X,W,\Omega,\cF,\bF,\bP)$ consisting of a filtered probability space $(\Omega,\cF,\bF=(\cF_t)_{t\geq0},\bP)$, a $d$-dimensional $\bF$-Brownian motion $W$ and an $\bF$-predictable process $X$ which satisfies \eqref{intro_eq_SVE} in a suitable sense. Uniqueness in law means that the first components $X$ of any weak solutions have the same probability law.

In the context of SDEs, general theory of weak well-posedness has been extensively studied. For example, it is well-known that if the drift coefficient $b$ is measurable and bounded, and if the diffusion coefficient $\sigma$ is continuous, bounded and non-degenerate (in the sense that $\sigma(x)\sigma(x)^\top$ is uniformly elliptic), then the SDE \eqref{intro_eq_SDE} is weakly well-posed, and the weak solution is a time-homogeneous Markov process satisfying the (strong) Feller property; see \cite[Chapter 7]{StVa79}. This means that, although the deterministic ordinary differential equation which corresponds to \eqref{intro_eq_SDE} with $\sigma=0$ may have multiple solutions, adding a non-degenerate noise by a Brownian motion makes the equation well-posed in the sense of uniqueness in law. This phenomenon is called the \emph{regularization-by-noise} effect for SDEs in terms of the Brownian noise; see \cite{Fl11}. From this well-known fact about SDEs, it is natural to ask whether such a regularization-by-noise effect appears for the SVE \eqref{intro_eq_SVE} as well. This question is quite non-trivial since, unlike the SDEs case, many standard tools in stochastic calculus cannot be applied directly to SVEs due to the lack of Markov and semimartingale properties. Nevertheless, in this paper, we provide a positive answer in this direction. Our main result is \cref{main_theo_main} which can be roughly stated as follows:


\begin{result}[Rough statement of \cref{main_theo_main}]
Assume that the kernel $K$ is completely monotone and satisfies $\int^1_0t^{-1/2}K(t)\,\diff t<\infty$. Suppose that the drift coefficient $b$ is uniformly continuous and that the diffusion coefficient $\sigma$ is non-degenerate and uniformly continuous. Then, under some \emph{``balance condition''} between the singularity of $K$ and the modulus of continuity of $\sigma$, the SVE \eqref{intro_eq_SVE} is weakly well-posed.
\end{result}

The ``balance condition'' in the above statement means that the more singular the kernel $K$ is, the more regular the diffusion coefficient $\sigma$ must be, and vice versa. On the one hand, if the diffusion coefficient $\sigma$ is Lipschitz continuous, then the ``balance condition'' in the above statement automatically holds; see \cref{main_cor_Lip}. On the other hand, if $K$ is the fractional kernel $K(t)=\frac{1}{\Gamma(\alpha)}t^{\alpha-1}$ with exponent $\alpha\in(\frac{1}{2},1]$ and if $\sigma$ is $\gamma$-H\"{o}lder continuous with exponent $\gamma\in(\frac{1}{2},1]$, then the ``balance condition'' becomes $\alpha\gamma>\frac{1}{2}$; see \cref{main_cor_fractional}. We note that no additional conditions are assumed on the drift coefficient $b$ other than the uniform continuity. For example, although the (one-dimensional) deterministic fractional Volterra equation $X_t=\frac{1}{\Gamma(\alpha)}\int^t_0(t-s)^{\alpha-1}|X_s|^\beta\mathrm{sign}(X_s)\,\diff s$ with $\alpha\in(\frac{1}{2},1]$ and $\beta\in(0,1)$ has infinitely many solutions, the corresponding SVE with non-degenerate ``Volterra noise'' is weakly well-posed; see \cref{main_exam_regularization}. This seems to be the first explicit example of the regularization-by-noise effect for Volterra equations.

Well-posedness of SVEs with (locally) Lipschitz continuous coefficients has been investigated in many papers; see for example \cite{BeMi80a,BeMi80b} for the classical results of SVEs with sufficiently regular kernels and \cite{AbiJaLaPu19,Ha23,Zh10} for the case of singular kernels. Here, let us briefly review some existing results on existence and/or uniqueness of solutins of SVEs with non-Lipschitz coefficients investigated in the literature so far.
\begin{itemize}
\item
\underline{A slight extension beyond the Lipschitz case.}
Wang \cite{Wa08} considered SVEs with (non-convolution type) singular kernels and coefficients which slightly beyond the Lipschitz case. This is the first paper for well-posedness of SVEs with non-Lipschitz coefficients. However, the assumptions in \cite{Wa08} exclude the case where the coefficient $b$ or $\sigma$ is only H\"{o}lder continuous with the exponent strictly less than $1$.
\item
\underline{Yamada--Watanabe type results for one-dimensional SVEs with the fractional kernel.}
Mytnik and Salisbury \cite{MySa15} considered a one-dimensional (drift-less) SVE with the fractional kernel $K(t)=\frac{1}{\Gamma(\alpha)}t^{\alpha-1}$ with exponent $\alpha\in(\frac{1}{2},1]$ and a $\gamma$-H\"{o}lder continuous diffusion coefficient $\sigma$ with exponent $\gamma\in(\frac{1}{2},1]$. Relying on the special structure of the fractional kernel, they first reformulated the SVE into a stochastic partial differential equation (SPDE for short), and then applied the method of the famous Yamada--Watanabe pathwise uniqueness result for one-dimensional SDEs \cite{YaWa71} to the SPDE. They showed that, under the balance condition $\alpha\gamma>\frac{1}{2}$, pathwise uniqueness holds for the SVE (with the forcing term $x$ of a special form). Pr\"{o}mel and Scheffels \cite{PrSc22b} generalized this result to the time-inhomogeneous setting, allowing for a non-zero (but Lipschitz continuous) drift coefficient $b$. The pathwise uniqueness issue without the balance condition (i.e.\ the case of $\alpha\gamma\leq\frac{1}{2}$) is an open problem at the moment.
\item
\underline{Yamada--Watanabe type results for one-dimensional SVEs with regular kernels.} In the case of sufficiently smooth kernels (which excludes any singular kernels such as the fractional kernel), it has been shown that pathwise uniqueness holds for one-dimensional SVEs with Lipschitz drift coefficients and $\gamma$-H\"{o}lder diffusion coefficients with exponent $\gamma\in[\frac{1}{2},1]$; see \cite[Section 8]{MySa15}, \cite[Proposition B.3]{AbiJaEu19} and \cite{PrSc23}.
\item
\underline{Weak existence for SVEs with continuous coefficients.}
Abi Jaber et al.\ \cite{AbiJaCuLaPu21} showed general weak existence and stability results for SVEs (with jumps) of convolution type under mild regularity assumptions, allowing for non-Lipschitz coefficients and singular kernels. See \cite{PrSc22a} for the study in this direction with time-inhomogeneous coefficients and non-convolution type kernels.
\item
\underline{Uniqueness in law for SVEs with affine coefficients.} Abi Jaber, Larsson and Pulido \cite{AbiJaLaPu19} introduced the so-called affine Volterra process which is defined as the solution of the SVE \eqref{intro_eq_SVE} where both $b(x)$ and $a(x):=\sigma(x)\sigma(x)^\top$ are affine with respect to $x$. By means of the explicit exponential-affine representation formula for the Fourier--Laplace transform, they proved uniqueness in law (in certain state spaces) for some specific examples. A unified theory for SVEs with affine coefficients was introduced by Cuchiero and Teichmann \cite{CuTe19,CuTe20} in terms of lifts of affine Volterra processes to the so-called infinite dimensional generalized Feller processes. See also \cite{AbiJa21} for the study of weak well-posedness of affine SVEs with $L^1$ kernels and jumps.
\end{itemize}
Compared with the aforementioned works, this paper is the first step showing uniqueness in law for SVEs under much mild regularity assumptions, beyond the affine case. It is interesting that, although the derivations are completely different, the ``balance condition'' $\alpha\gamma>\frac{1}{2}$ between the exponent $\alpha$ of the fractional kernel $K$ and the H\"{o}lder exponent $\gamma$ of the diffusion coefficient $\sigma$ appearing in our special case (\cref{main_cor_fractional}) coincides with that of \cite{MySa15} where pathwise uniqueness for one-dimensional SVEs was investigated. In view of general theory on weak well-posedness of SDEs (cf.\ \cite{StVa79}), one may expect that some of assumptions in our results, especially the ``balance condition'', can be somewhat weakened or dropped. However, the ``balance condition'' turns out to be necessary at least in our strategy for proving weak well-posedness. We will discuss about this condition after completing the proof; see \cref{proof-remark-balance}.

Now let us explain our idea to show weak well-posedness of SVEs. As we mentioned before, the analysis of SVEs with singular kernels is very difficult due to the lack of Markov and semimartingale properties. Then, instead of analysing SVEs directly, we consider the associated \emph{Hilbert-valued Markovian lifts} introduced in our previous work \cite{Ha23}. Namely, noting that every completely monotone kernel $K$ can be written as $K(t)=\int_{[0,\infty)}e^{-\theta t}\dmu$ by a unique Radon measure $\mu$ on $[0,\infty)$ (cf.\ \cite[Theorem 1.4]{ScSoVo12}), we reformulate the SVE \eqref{intro_eq_SVE} into the following equation:
\begin{equation}\label{intro_eq_SEE-lift}
	\begin{dcases}
	\diff Y_t(\theta)=-\theta Y_t(\theta)\,\diff t+b(\mu[Y_t])\,\diff t+\sigma(\mu[Y_t])\,\diff W_t,\ \ \theta\in\supp,\ t>0,\\
	Y_0(\theta)=y(\theta),\ \ \theta\in\supp.
	\end{dcases}
\end{equation}
Here, $\supp\subset[0,\infty)$ is the support of the measure $\mu$, and $\mu[y]:=\int_\supp y(\theta)\dmu$ for suitable functions $y:\supp\to\bR^n$. The above equation can be seen as a stochastic evolution equation (SEE for short) on a Gelfand triplet $\cV_\mu\hookrightarrow\cH_\mu\hookrightarrow\cV^*_\mu$ of Hilbert spaces consisting of functions $y:\supp\to\bR^n$; see \cite[Section 2]{Ha23} and \cref{sol} below. Roughly speaking, there is an equivalence between the original SVE \eqref{intro_eq_SVE} and the lifted SEE \eqref{intro_eq_SEE-lift}, and we can apply well-known tools of stochastic calculus in Hilbert spaces to the latter. Our idea is to show weak well-posedness of the lifted SEE \eqref{intro_eq_SEE-lift}, and then to translate it into the original SVE \eqref{intro_eq_SVE}.

Note that, however, there are at least three difficulties for proving weak well-posedness of the lifted SEE \eqref{intro_eq_SEE-lift}. First, the lifted SEE is \emph{highly degenerate} in the sense that the state space $\cH_\mu$ is (typically) infinite dimensional, but the Brownian noise is finite dimensional; even if the diffusion coefficient $\sigma$ is non-degenerate in $\bR^n$, it would be degenerate in the infinite dimensional state space $\cH_\mu$. This indicates that the regularization-by-noise effect for the lifted SEE is highly non-trivial. Second, in the case of singular kernels, the integral map $y\mapsto\mu[y]$ is continuous only in the smaller space $\cV_\mu$ (with stronger topology) and not continuous in the state space $\cH_\mu$, hence the coefficients $y\mapsto b(\mu[y])$ or $y\mapsto\sigma(\mu[y])$ are not continuous in $\cH_\mu$ even if the maps $x\mapsto b(x)$ and $x\mapsto\sigma(x)$ are continuous in $\bR^n$. Third, at this moment, there are no useful criteria of compactness in the space of paths on the Gelfand triplet $\cV_\mu\hookrightarrow\cH_\mu\hookrightarrow\cV^*_\mu$. For these reasons, it is not realistic to use the analytic approach relying on the results of the associated Kolmogorov equations which has been commonly adopted in the study of finite dimensional SDEs (cf.\ \cite{StVa79}). Instead, inspired by the probabilistic approach of Kulik and Scheutzow \cite{KlSc20} for the study of SDEs with delay, we demonstrate the so-called \emph{Control-and-Reimburse (C-n-R)} strategy to prove weak well-posedness of the lifted SEE \eqref{intro_eq_SEE-lift}. The idea of the C-n-R strategy is ``\emph{to apply a stochastic control in order to improve the system, and then to take into account the impact of the control}'' \cite{KlSc20}. We will explain more detailed ideas concerning the C-n-R strategy in the introductory part of \cref{proof}. Using such stochastic control arguments, we overcome the difficulties caused by the non-degeneracy and low regularity of the coefficients of the lifted SEE. Furthermore, besides the weak well-posedness, we show that the weak solution of the lifted SEE becomes a Markov process on the Hilbert space $\cH_\mu$ and satisfies the Feller property. This result indicates that the solution $Y$ of the SEE \eqref{intro_eq_SEE-lift} is a ``Markovian lift'' of the solution $X$ of the original SVE \eqref{intro_eq_SVE}. The same general ideas in the spirit of the C-n-R strategy were also used in the studies of ergodicity of various stochastic systems; see \cite{BuKuSc20,Kl18} and references cited therein. It is worth to remark that a related idea was used to establish the asymptotic log-Harnack inequality for the lifted SEE in our previous work \cite{Ha23}.

The rest of this paper is organized as follows: In \cref{sol}, we briefly review the framework of Hilbert-valued Markovian lifts of SVEs introduced in \cite{Ha23} and then describe the associated solution concepts. In \cref{main}, we state our main result (\cref{main_theo_main}) and show some applications. \cref{proof} is devoted to the proof of \cref{main_theo_main}. In \hyperref[appendix]{Appendix}, we show an auxiliary result needed in the proof of the main result.

\subsection*{Notations}

We fix two natural numbers $n$ and $d$, which represent the dimension of the solutions of SVEs and Brownian motion, respectively, considered in this paper. $\bR^n$ is the Euclidean space of $n$-dimensional vectors, and we denote by $|v|$ and $\langle v_1,v_2\rangle$ the usual Euclidean norm and the inner product for vectors $v,v_1,v_2\in\bR^n$. $\bR^{n\times d}$ is the space of $(n\times d)$-matrices endowed with the Frobenius norm. For each matrix $M\in\bR^{n\times d}$, $M^\top\in\bR^{d\times n}$ denotes the transpose. For two symmetric matrices $M_1,M_2\in\bR^{n\times n}$, $M_1\geq M_2$ means that $M_1-M_2$ is nonnegative definite. $I_{n\times n}\in\bR^{n\times n}$ denotes the identity matrix.

In this paper, we call $(\Omega,\cF,\bF,\bP)$ a filtered probability space if $(\Omega,\cF,\bP)$ is a complete probability space endowed with the filtration $\bF=(\cF_t)_{t\geq0}$ satisfying the usual conditions. For a given $d$-dimensional Brownian motion $W$ on $(\Omega,\cF,\bP)$, we denote by $\bF^W$ the $\bP$-completion of the filtration generated by $W$. $\bE$ denotes the expectation under $\bP$. For each set $A$, $\1_A$ denotes the indicator function.

For a topological space $E$, we denote by $\cB(E)$ the Borel $\sigma$-algebra on $E$. We denote by $\cP(E)$ the set of probability measures on $(E,\cB(E))$. For an $E$-valued random variable $\xi$ on a probability space $(\Omega,\cF,\bP)$, we denote by $\Law_\bP(\xi)\in\cP(E)$ the law of $\xi$ under the probability measure $\bP$. We denote by $\cB_\mathrm{b}(E)$ the space of $\bR$-valued bounded Borel measurable functions on $E$, and by $\cC_\mathrm{b}(E)$ the space of $\bR$-valued bounded continuous functions on $E$.

$C(\bR_+;\bR^d)$ denotes the set of continuous functions from $\bR_+=[0,\infty)$ to $\bR^d$, which is a Polish space with the topology of the uniform convergence on every compact interval. For each $T>0$, $L^2(0,T;\bR^n)$ denotes the separable Hilbert space of equivalent classes of measurable and square-integrable functions $\varphi:[0,T]\to\bR^n$ endowed with the norm $\|\varphi\|_{L^2(0,T)}:=\big(\int^T_0|\varphi(t)|^2\,\diff t\big)^{1/2}$. Also, $L^2_\loc(\bR_+;\bR^n)$ denotes the Polish space of locally square-integrable functions $\varphi:\bR_+\to\bR^n$ endowed with the metric
\begin{equation*}
	d_{L^2_\loc}(\varphi,\bar{\varphi}):=\sum^\infty_{T=1}\frac{1}{2^T}\frac{\|\varphi-\bar{\varphi}\|_{L^2(0,T)}}{1+\|\varphi-\bar{\varphi}\|_{L^2(0,T)}},\ \ \varphi,\bar{\varphi}\in L^2_\loc(\bR_+;\bR^n).
\end{equation*}
For a Borel measure $\mu$ on $[0,\infty)$ and $p\geq1$, $L^p(\mu)$ denotes the separable Banach space of equivalent classes of $\bR^n$-valued measurable functions $\varphi$ on $[0,\infty)$ such that $\int_\supp|\varphi(\theta)|^p\dmu<\infty$ endowed with the norm $\|\varphi\|_{L^p(\mu)}:=\big(\int_\supp|\varphi(\theta)|^p\dmu\big)^{1/p}$. Here, $\supp\subset[0,\infty)$ denotes the support of $\mu$.
 

\section{Solution concepts for SVEs and SEEs}\label{sol}

First, we describe an assumption on the kernel $K$ appearing in the SVE \eqref{intro_eq_SVE}. Recall that a kernel $K:(0,\infty)\to[0,\infty)$ is said to be completely monotone if $K$ is infinitely differentiable and satisfies $(-1)^k\frac{\diff^k}{\diff t^k}K(t)\geq0$ for any $t\in(0,\infty)$ and any nonnegative integers $k$. By Bernstein's theorem (cf.\ \cite[Theorem 1.4]{ScSoVo12}), there exists a unique Radon measure $\mu$ on $[0,\infty)$ such that
\begin{equation*}
	K(t)=\int_{[0,\infty)}e^{-\theta t}\dmu=\int_\supp e^{-\theta t}\dmu,\ \ t>0,
\end{equation*}
where $\supp\subset[0,\infty)$ denotes the support of $\mu$. We say that a completely monotone kernel $K$ is \emph{regular} if $\lim_{t\downarrow0}K(t)<\infty$, and \emph{singular} if $\lim_{t\downarrow0}K(t)=\infty$. It is easy to see that $K$ is regular if and only if $\mu([0,\infty))<\infty$. Concerning the singularity, for each fixed $\alpha\in(0,1]$, $\int^1_0t^{\alpha-1}K(t)\,\diff t<\infty$ if and only if $\int_\supp1\wedge(\theta^{-\alpha})\dmu<\infty$, and these equivalent conditions imply $\int^1_0K(t)^{1/\alpha}\,\diff t<\infty$; see \cite[Lemma 2.1]{Ha23}. Namely, the regularity/singularity of the kernel $K$ (at $t=0$) is characterized by the lightness/heaviness of the tail of the measure $\mu$.


\begin{assum}\label{sol_assum_kernel}
$K:(0,\infty)\to[0,\infty)$ is a completely monotone kernel with the corresponding Radon measure $\mu$ on $[0,\infty)$. Furthermore, there exists a non-increasing Borel measurable function $r$ on $[0,\infty)$ satisfying $1\wedge(\theta^{-1/2})\leq r(\theta)\leq1$ for any $\theta\in[0,\infty)$ such that
\begin{equation*}
	\int_\supp r(\theta)\dmu<\infty.
\end{equation*}
\end{assum}

For a completely monotone kernel $K$, the above assumption is equivalent to $\int^1_0t^{-1/2}K(t)\,\diff t<\infty$. The function $r$ characterizes the regularity/singularity of the kernel $K$. In this paper, if the kernel $K$ is regular, we will take $r\equiv1$. The following are examples of kernels which satisfy \cref{sol_assum_kernel}:
\begin{itemize}
\item
A finite combination of exponential kernel: $K(t)=\sum^k_{i=1}c_ie^{-\theta_it}$ with $k\in\bN$, $c_i>0$ and $\theta_i\in[0,\infty)$ for $i=1,\dots,k$. This kernel is regular, and the corresponding Radon measure is $\mu(\diff\theta)=\sum^k_{i=1}c_i\delta_{\theta_i}(\diff\theta)$ with $\supp=\{\theta_1,\dots,\theta_k\}$. Here, $\delta_{\theta_i}$ denotes the Dirac measure at point $\theta_i$.
\item
The fractional kernel: $K(t)=\frac{1}{\Gamma(\alpha)}t^{\alpha-1}$ with exponent $\alpha\in(\frac{1}{2},1)$. This kernel is singular. The corresponding Radon measure is as follows:
\begin{equation*}
	\mu(\diff\theta)=\frac{1}{\Gamma(\alpha)\Gamma(1-\alpha)}\theta^{-\alpha}\1_{(0,\infty)}(\theta)\,\diff\theta,\ \ \supp=[0,\infty).
\end{equation*}
We can take $r(\theta)=1\wedge(\theta^{-\eta})$ for any $\eta\in(1-\alpha,\frac{1}{2}]$.
\item
The Gamma kernel: $K(t)=\frac{1}{\Gamma(\alpha)}e^{-\beta t}t^{\alpha-1}$ with exponents $\alpha\in(\frac{1}{2},1)$ and $\beta>0$. This kernel is singular. The corresponding Radon measure is as follows:
\begin{equation*}
	\mu(\diff\theta)=\frac{1}{\Gamma(\alpha)\Gamma(1-\alpha)}(\theta-\beta)^{-\alpha}\1_{(\beta,\infty)}(\theta)\,\diff\theta,\ \ \supp=[\beta,\infty).
\end{equation*}
we can take $r(\theta)=1\wedge(\theta^{-\eta})$ for any $\eta\in(1-\alpha,\frac{1}{2}]$.
\end{itemize}

Here, we describe the solution concepts for the SVE \eqref{intro_eq_SVE}.


\begin{defi}
Suppose that \cref{sol_assum_kernel} holds. Let $b:\bR^n\to\bR^n$ and $\sigma:\bR^n\to\bR^{n\times d}$ be measurable maps. Let $x\in L^2_\loc(\bR_+;\bR^n)$ be a given forcing term.
\begin{itemize}
\item
Fix a filtered probability space $(\Omega,\cF,\bF,\bP)$ and a $d$-dimensional $\bF$-Brownian motion $W$. We say that an $\bR^n$-valued process $X$ defined on $(\Omega,\cF,\bF,\bP)$ is a \emph{solution} of the SVE \eqref{intro_eq_SVE} (associated with the Brownian motion $W$ and the forcing term $x$) if $X$ is $\bF$-predictable, the integrability conditions
\begin{equation*}
	\int^T_0|X_t|^2\,\diff t<\infty,\ \int^T_0|b(X_t)|\,\diff t<\infty\ \text{and}\ \int^T_0|\sigma(X_t)|^2\,\diff t<\infty
\end{equation*}
hold a.s.\ for any $T>0$, and the equality in \eqref{intro_eq_SVE} holds a.s.\ for a.e.\ $t>0$.
\item
We say that a tuple $(X,W,\Omega,\cF,\bF,\bP)$ is a \emph{weak solution} of the SVE \eqref{intro_eq_SVE} if $X$ is a solution of \eqref{intro_eq_SVE} associated with a $d$-dimensional $\bF$-Brownian motion $W$ on a filtered probability space $(\Omega,\cF,\bF,\bP)$. We say that \emph{weak existence} holds for the SVE \eqref{intro_eq_SVE} with the forcing term $x$ if there exists a weak solution.
\item
We say that a weak solution $(X,W,\Omega,\cF,\bF,\bP)$ of the SVE \eqref{intro_eq_SVE} is a \emph{strong solution} if $X$ is $\bF^W$-predictable. We say that \emph{strong existence} holds for the SVE \eqref{intro_eq_SVE} with the forcing term $x$ if there exists a strong solution.
\item
We say that \emph{uniqueness in law} holds for the SVE \eqref{intro_eq_SVE} with the forcing term $x$ if, for any two weak solutions $(X^1,W^1,\Omega^1,\cF^1,\bF^1,\bP^1)$ and $(X^2,W^2,\Omega^2,\cF^2,\bF^2,\bP^2)$, it holds that $\Law_{\bP^1}(X^1)=\Law_{\bP^2}(X^2)$, where $X^1$ and $X^2$ are understood as random variables on $L^2_\loc(\bR_+;\bR^n)$.
\item
We say that \emph{pathwise uniqueness} holds for the SVE \eqref{intro_eq_SVE} with the forcing term $x$ if, for any two weak solutions $(X^1,W,\Omega,\cF,\bF^1,\bP)$ and $(X^2,W,\Omega,\cF,\bF^2,\bP)$ (with common $d$-dimensional Brownian motion $W$ relative to possibly different filtrations $\bF^1$ and $\bF^2$ on a common probability space $(\Omega,\cF,\bP)$), it holds that $X^1_t=X^2_t$ for a.e.\ $t>0$ a.s.
\end{itemize}
\end{defi}

%
%
%

As we explained in \hyperref[intro]{Introduction}, instead of studying SVEs directly, we will consider the lifted SEE \eqref{intro_eq_SEE-lift}. Here, we describe the framework of SEEs introduced in \cite{Ha23}. Under \cref{sol_assum_kernel}, for each Borel measurable function $y:\supp\to\bR^n$, define
\begin{equation*}
	\|y\|_{\cH_\mu}:=\Big(\int_\supp r(\theta)|y(\theta)|^2\dmu\Big)^{1/2}
\end{equation*}
and
\begin{equation*}
	\|y\|_{\cV_\mu}:=\Big(\int_\supp (\theta+1)r(\theta)|y(\theta)|^2\dmu\Big)^{1/2}.
\end{equation*}
We denote by $\cH_\mu$ the separable Banach space of equivalent classes of Borel measurable functions $y:\supp\to\bR^n$ such that $\|y\|_{\cH_\mu}<\infty$. Similarly, we denote by $\cV_\mu$ the separable Banach space of equivalent classes of Borel measurable functions $y:\supp\to\bR^n$ such that $\|y\|_{\cV_\mu}<\infty$. Furthermore, we equip $\cH_\mu$ with the natural inner product
\begin{equation*}
	\langle y_1, y_2\rangle_{\cH_\mu}:=\int_\supp r(\theta)\langle y_1(\theta),y_2(\theta)\rangle\dmu,\ \ y_1,y_2\in\cH_\mu,
\end{equation*}
hence $(\cH_\mu,\|\cdot\|_{\cH_\mu},\langle\cdot,\cdot\rangle_{\cH_\mu})$ is a separable Hilbert space. It is easy to see that $\cV_\mu\subset\cH_\mu$ continuously and densely. Denote the continuous dual spaces of $\cH_\mu$ and $\cV_\mu$ by $\cH^*_\mu$ and $\cV^*_\mu$, respectively. By means of the Riesz isomorphism, we identify $\cH_\mu$ with $\cH^*_\mu$. Then we see that $\cH_\mu\subset\cV^*_\mu$ continuously and densely. The duality pairing $\dual{\cdot}{\cdot}$ is then compatible with the inner product on $\cH_\mu$ in the sense that $\dual{y'}{y}=\langle y',y\rangle_{\cH_\mu}$ whenever $y\in\cV_\mu\subset\cH_\mu$ and $y'\in\cH_\mu\subset\cV^*_\mu$. By this constructions, we obtain a Gelfand triplet of Hilbert spaces $\cV_\mu\hookrightarrow\cH_\mu\hookrightarrow\cV^*_\mu$. We will adopt the Hilbert space $\cH_\mu$ (or the Gelfand triplet $\cV_\mu\hookrightarrow\cH_\mu\hookrightarrow\cV^*_\mu$) as the state space of the SEE \eqref{intro_eq_SEE-lift}. By \cite[Lemma 2.5]{Ha23}, we see that the family of maps $e^{-\cdot t}:y\mapsto e^{-\cdot t}y=(\theta\mapsto e^{-\theta t}y(\theta))$, $t\geq0$, is a contraction $C_0$-semigroup on the Hilbert space $\cH_\mu$. Also, we have the following continuous embeddings:
\begin{equation*}
	\bR^n\hookrightarrow\cH_\mu,\ \cV_\mu\hookrightarrow L^1(\mu)\ \text{and}\ \cV_\mu\hookrightarrow L^2(\mu)\hookrightarrow\cH_\mu.
\end{equation*}
Note that if the kernel $K$ is singular, or equivalently $\mu([0,\infty))=\infty$, then any nonzero constant vectors do not belong to $L^1(\mu)$ or $L^2(\mu)$, hence $\cH_\mu\setminus (L^1(\mu)\cup L^2(\mu))\neq\emptyset$. Besides, if $\mu$ is diffusive, then $L^1(\mu)\setminus\cH_\mu\neq\emptyset$. Therefore, the spaces $L^1(\mu)$ and $\cH_\mu$ have no inclusion-relation in general.

In the above framework, the integral map $\mu[\cdot]$ appearing in \eqref{intro_eq_SEE-lift} is defined as follows: for any $y\in\cH_\mu$,
\begin{equation}\label{sol_eq_mu}
	\mu[y]:=
	\begin{dcases}
	\int_\supp y(\theta)\dmu\ &\text{if $y\in L^1(\mu)$},\\
	0\ &\text{if $y\notin L^1(\mu)$}.
	\end{dcases}
\end{equation}
Then, the map $\mu[\cdot]:\cH_\mu\to\bR^n$ is $\cB(\cH_\mu)/\cB(\bR^n)$-measurable, and its restriction $\mu[\cdot]|_{\cV_\mu}$ on $\cV_\mu$ is a bounded linear operator from $\cV_\mu$ to $\bR^n$. The following lemma shows that the above framework is suitable for studying Volterra processes. See \cite[Lemma 2.7]{Ha23} for more detailed assertions and their proofs.


\begin{lemm}[Lemma 2.7 of \cite{Ha23}]
Let \cref{sol_assum_kernel} hold. Fix $T>0$ and a filtered probability space $(\Omega,\cF,\bF,\bP)$ endowed with a $d$-dimensional $\bF$-Brownian motion $W$.
\begin{itemize}
\item[(i)]
Let $b:\Omega\times[0,T]\to\bR^n$ be a $\bF$-predictable process such that $\int^T_0|b_t|\,\diff t<\infty$ a.s. For each $t\in[0,T]$, consider the convolution
\begin{equation*}
	I^1_t:=\int^t_0e^{-\cdot(t-s)}b_s\,\diff s,
\end{equation*}
where the integral is defined as a Bochner integral on $\cH_\mu$. Then $I^1=(I^1_t)_{t\in[0,T]}$ is an $\cH_\mu$-valued continuous $\bF$-adapted process and satisfies $\int^T_0\|I^1_t\|^2_{\cV_\mu}\,\diff t<\infty$ a.s. Furthermore, it holds that
\begin{equation*}
	I^1_t\in\cV_\mu\ \ \text{and}\ \ \mu[I^1_t]=\int^t_0K(t-s)b_s\,\diff s\ \ \text{a.s.\ for a.e.\ $t\in[0,T]$}.
\end{equation*}
\item[(ii)]
Let $\sigma:\Omega\times[0,T]\to\bR^{n\times d}$ be a $\bF$-predictable process such that $\int^T_0|\sigma_t|^2\,\diff t<\infty$ a.s. For each $t\in[0,T]$, consider the stochastic convolution
\begin{equation*}
	I^2_t:=\int^t_0e^{-\cdot(t-s)}\sigma_s\,\diff W_s,
\end{equation*}
where the integral is defined as a stochastic integral on the Hilbert space $\cH_\mu$. Then $I^2=(I^2_t)_{t\in[0,T]}$ is an $\cH_\mu$-valued $\bF$-adapted process, has an $\cH_\mu$-continuous version (again denoted by $I^2$), and satisfies $\int^T_0\|I^2_t\|^2_{\cV_\mu}\,\diff t<\infty$ a.s. Furthermore, it holds that
\begin{equation*}
	I^2_t\in\cV_\mu\ \ \text{and}\ \ \mu[I^2_t]=\int^t_0K(t-s)\sigma_s\,\diff  W_s\ \ \text{a.s.\ for a.e.\ $t\in[0,T]$.}
\end{equation*}
\end{itemize}
\end{lemm}

From the above lemma, it is natural to introduce the following path space: for each $T\in(0,\infty)$,
\begin{equation*}
	\Lambda_T:=\left\{y:[0,T]\to\cH_\mu\relmiddle|\text{$y$ is $\cH_\mu$-continuous and satisfies}\ \int^T_0\|y_t\|^2_{\cV_\mu}\,\diff t<\infty\right\}.
\end{equation*}
$\Lambda_T$ is a separable Banach space with the norm $\|\cdot\|_{\Lambda_T}$ defined by
\begin{equation}\label{sol_eq_norm-Lambda}
	\|y\|_{\Lambda_T}:=\Big(\sup_{t\in[0,T]}\|y_t\|^2_{\cH_\mu}+\int^T_0\|y_t\|^2_{\cV_\mu}\,\diff t\Big)^{1/2},\ \ y\in\Lambda_T.
\end{equation}
Also, define
\begin{equation*}
	\Lambda:=\{y:[0,\infty)\to\cH_\mu\,|\,\text{for any $T\in(0,\infty)$, the restriction of $y$ on $[0,T]$ belongs to $\Lambda_T$}\},
\end{equation*}
which is a Polish space endowed with the metric
\begin{equation*}
	d_\Lambda(y,\bar{y}):=\sum_{T\in\bN}\frac{1}{2^T}\frac{\|y-\bar{y}\|_{\Lambda_T}}{1+\|y-\bar{y}\|_{\Lambda_T}},\ \ y,\bar{y}\in\Lambda.
\end{equation*}
Note that, for any $y,y_k\in\Lambda$, $k\in\bN$, $y_k\to y$ in $(\Lambda,d_\Lambda)$ if and only if $y_k\to y$ in $(\Lambda_T,\|\cdot\|_{\Lambda_T})$ for all $T>0$.

Now we describe the solution concepts for the SEE \eqref{intro_eq_SEE-lift}. Here, let us consider a slightly more general SEE of the following form:
\begin{equation}\label{sol_eq_SEE-general}
	\begin{dcases}
	\diff Y_t(\theta)=-\theta Y_t(\theta)\,\diff t+\hat{b}(Y_t)\,\diff t+\hat{\sigma}(Y_t)\,\diff W_t,\ \ \theta\in\supp,\ t>0,\\
	Y_0(\theta)=y(\theta),\ \ \theta\in\supp.
	\end{dcases}
\end{equation}


\begin{defi}
Suppose that \cref{sol_assum_kernel} holds. Let $\hat{b}:\cH_\mu\to\bR^n$ and $\hat{\sigma}:\cH_\mu\to\bR^{n\times d}$ be measurable maps. Let $y\in\cH_\mu$ be a given initial condition.
\begin{itemize}
\item
Fix a filtered probability space $(\Omega,\cF,\bF,\bP)$ and a $d$-dimensional $\bF$-Brownian motion $W$. We say that an $\cH_\mu$-valued process $Y$ defined on $(\Omega,\cF,\bF,\bP)$ is a \emph{mild solution} of the SEE \eqref{sol_eq_SEE-general} (associated with the Brownian motion $W$ and the initial condition $y$) if $Y$ is $\cH_\mu$-continuous and $\bF$-adapted, the integrability conditions
\begin{equation*}
	\int^T_0\|Y_t\|^2_{\cV_\mu}\,\diff t<\infty,\ \int^T_0|\hat{b}(Y_t)|\,\diff t<\infty\ \text{and}\ \int^T_0|\hat{\sigma}(Y_t)|^2\,\diff t<\infty
\end{equation*}
hold a.s.\ for any $T>0$, and the equality
\begin{equation*}
	Y_t(\theta)=e^{-\theta t}y(\theta)+\int^t_0e^{-\theta(t-s)}\hat{b}(Y_s)\,\diff s+\int^t_0e^{-\theta(t-s)}\hat{\sigma}(Y_s)\,\diff W_s
\end{equation*}
holds for $\mu$-a.e.\ $\theta\in\supp$ a.s., for any $t\geq0$.
\item
We say that a tuple $(Y,W,\Omega,\cF,\bF,\bP)$ is a \emph{weak solution} (in the probabilistic sense) of the SEE \eqref{sol_eq_SEE-general} if $Y$ is a mild solution of the SEE \eqref{sol_eq_SEE-general} associated with a $d$-dimensional $\bF$-Brownian motion $W$ on a filtered probability space $(\Omega,\cF,\bF,\bP)$. We say that \emph{weak existence} holds for the SEE \eqref{sol_eq_SEE-general} with the initial condition $y$ if there exists a weak solution.
\item
We say that a weak solution $(Y,W,\Omega,\cF,\bF,\bP)$ of the SEE \eqref{sol_eq_SEE-general} is a \emph{strong solution} (in the probabilistic sense) if $Y$ is $\bF^W$-adapted. We say that \emph{strong existence} holds for the SEE \eqref{sol_eq_SEE-general} with the initial condition $y$ if there exists a strong solution.
\item
We say that \emph{uniqueness in law} holds for the SEE \eqref{sol_eq_SEE-general} with the initial condition $y$ if, for any two weak solutions $(Y^1,W^1,\Omega^1,\cF^1,\bF^1,\bP^1)$ and $(Y^2,W^2,\Omega^2,\cF^2,\bF^2,\bP^2)$, it holds that $\Law_{\bP^1}(Y^1)=\Law_{\bP^2}(Y^2)$, where $Y^1$ and $Y^2$ are understood as random variables on $\Lambda$.
\item
We say that \emph{pathwise uniqueness} holds for the SEE \eqref{sol_eq_SEE-general} with the initial condition $y$ if, for any two weak solutions $(Y^1,W,\Omega,\cF,\bF^1,\bP)$ and $(Y^2,W,\Omega,\cF,\bF^2,\bP)$ (with common $d$-dimensional Brownian motion $W$ relative to possibly different filtrations $\bF^1$ and $\bF^2$ on a common probability space $(\Omega,\cF,\bP)$), it holds that $Y^1_t=Y^2_t$ in $\cH_\mu$ for any $t\geq0$ a.s.
\end{itemize}
\end{defi}


\begin{rem}
The SEE \eqref{sol_eq_SEE-general} can be seen as a parametrized family of finite dimensional SDEs. Indeed, for each weak solution $(Y,W,\Omega,\cF,\bF,\bP)$ of the SEE \eqref{sol_eq_SEE-general}, there exists a jointly measurable map $\Omega\times[0,\infty)\times\supp\ni(\omega,t,\theta)\mapsto\tilde{Y}_t(\theta)(\omega)\in\bR^n$ (with respect to the predictable $\sigma$-algebra on $\Omega\times[0,\infty)$ and the Borel $\sigma$-algebra on $\supp$) such that $Y_t=\tilde{Y}_t$ in $\cH_\mu$ a.s.\ for any $t\geq0$ and that, for each $\theta\in\supp$, the process $\tilde{Y}(\theta)=(\tilde{Y}_t(\theta))_{t\geq0}$ is an $\bR^n$-valued It\^{o} process satisfying
\begin{equation*}
	\begin{dcases}
	\diff \tilde{Y}_t(\theta)=-\theta \tilde{Y}_t(\theta)\,\diff t+\hat{b}(\tilde{Y_t})\,\diff t+\hat{\sigma}(\tilde{Y_t})\,\diff W_t,\ \ t>0,\\
	\tilde{Y}_0(\theta)=y(\theta),
	\end{dcases}
\end{equation*}
in the usual It\^{o}'s sense; see \cite[Remarks 2.8 and 2.10]{Ha23}.
\end{rem}

Define $\cK:\cH_\mu\to L^2_\loc(\bR_+;\bR^n)$ by
\begin{equation*}
	(\cK y)(t):=\mu[e^{-\cdot t}y]=\int_\supp e^{-\theta t}y(\theta)\dmu,\ \ t>0,\ y\in\cH_\mu.
\end{equation*}
$\cK$ is a continuous linear map from $\cH_\mu$ to $L^2_\loc(\bR_+;\bR^n)$; see \cite[Section 2.1]{Ha23}. The following proposition shows that the ``lifted SEE'' \eqref{intro_eq_SEE-lift} with the initial condition $y\in\cH_\mu$ and the SVE \eqref{intro_eq_SVE} with the free term $x=\cK y$ are equivalent in some sense. Recall that the SVE \eqref{intro_eq_SVE} and the lifted SEE \eqref{intro_eq_SEE-lift} are of the following forms:
\begin{equation*}
	X_t=x(t)+\int^t_0K(t-s)b(X_s)\,\diff s+\int^t_0K(t-s)\sigma(X_s)\,\diff W_s,\ \ t>0,
\end{equation*}
and
\begin{equation*}
	\begin{dcases}
	\diff Y_t(\theta)=-\theta Y_t(\theta)\,\diff t+b(\mu[Y_t])\,\diff t+\sigma(\mu[Y_t])\,\diff W_t,\ \ \theta\in\supp,\ t>0,\\
	Y_0(\theta)=y(\theta),\ \ \theta\in\supp,
	\end{dcases}
\end{equation*}
respectively.


\begin{prop}[Theorem 2.11 of \cite{Ha23}]\label{sol_prop_equiv}
Suppose that \cref{sol_assum_kernel} holds. Let $b:\bR^n\to\bR^n$ and $\sigma:\bR^n\to\bR^{n\times d}$ be measurable maps. Fix a filtered probability space $(\Omega,\cF,\bF,\bP)$ and a $d$-dimensional $\bF$-Brownian motion $W$. For each $y\in\cH_\mu$, the following assertions hold:
\begin{itemize}
\item
If $X$ is a solution of the SVE \eqref{intro_eq_SVE} with the forcing term $x=\cK y$, then the $\cH_\mu$-valued process $Y$ defined by
\begin{equation*}
	Y_t:=e^{-\cdot t}y(\cdot)+\int^t_0e^{-\cdot(t-s)}b(X_s)\,\diff s+\int^t_0e^{-\cdot(t-s)}\sigma(X_s)\,\diff W_s,\ \ t\geq0,
\end{equation*}
is a mild solution of the lifted SEE \eqref{intro_eq_SEE-lift} with the initial condition $y$. Furthermore, it holds that
\begin{equation*}
	X_t=\mu[Y_t]
\end{equation*}
a.s.\ for a.e.\ $t>0$.
\item
If $Y$ is a mild solution of the lifted SEE \eqref{intro_eq_SEE-lift} with the initial condition $y$, then the $\bR^n$-valued process $X$ defined by
\begin{equation*}
	X_t:=\mu[Y_t],\ \ t>0,
\end{equation*}
is a solution of the SVE \eqref{intro_eq_SVE} with the forcing term $x=\cK y$. Furthermore, it holds that
\begin{equation*}
	Y_t=e^{-\cdot t}y(\cdot)+\int^t_0e^{-\cdot(t-s)}b(X_s)\,\diff s+\int^t_0e^{-\cdot(t-s)}\sigma(X_s)\,\diff W_s
\end{equation*}
a.s.\ for any $t\geq0$.
\end{itemize}
\end{prop}


\begin{cor}\label{sol_cor_equivalence}
Suppose that \cref{sol_assum_kernel} holds. Let $b:\bR^n\to\bR^n$ and $\sigma:\bR^n\to\bR^{n\times d}$ be measurable maps. Then, for any $y\in\cH_\mu$, the following assertions hold:
\begin{itemize}
\item[(i)]
Weak existence holds for the SVE \eqref{intro_eq_SVE} with the forcing term $x=\cK y$ if and only if weak existence holds for the lifted SEE \eqref{intro_eq_SEE-lift} with the initial condition $y$.
\item[(ii)]
Strong existence holds for the SVE \eqref{intro_eq_SVE} with the forcing term $x=\cK y$ if and only if strong existence holds for the lifted SEE \eqref{intro_eq_SEE-lift} with the initial condition $y$.
\item[(iii)]
Pathwise uniqueness holds for the SVE \eqref{intro_eq_SVE} with the forcing term $x=\cK y$ if and only if pathwise uniqueness holds for the lifted SEE \eqref{intro_eq_SEE-lift} with the initial condition $y$.
\item[(iv)]
If uniqueness in law holds for the lifted SEE \eqref{intro_eq_SEE-lift} with the initial condition $y\in\cH_\mu$, then uniqueness in law holds for the SVE \eqref{intro_eq_SVE} with the forcing term $x=\cK y$.
\end{itemize}
\end{cor}


\begin{proof}
The assertions (i), (ii) and (iii) are trivial from \cref{sol_prop_equiv}. To show the assertion (iv), assume that uniqueness in law holds for the lifted SEE \eqref{intro_eq_SEE-lift} with the initial condition $y\in\cH_\mu$. Let $(X^i,W^i,\Omega^i,\cF^i,\bF^i,\bP^i)$, $i=1,2$, be two weak solutions of the SVE \eqref{intro_eq_SVE} with the forcing term $x=\cK y$. For $i=1,2$, define the $\cH_\mu$-valued process $Y^i$ by
\begin{equation*}
	Y^i_t:=e^{-\cdot t}y(\cdot)+\int^t_0e^{-\cdot(t-s)}b(X^i_s)\,\diff s+\int^t_0e^{-\cdot(t-s)}\sigma(X^i_s)\,\diff W^i_s,\ \ t\geq0.
\end{equation*}
By \cref{sol_prop_equiv}, the tuple $(Y^i,W^i,\Omega^i,\cF^i,\bF^i,\bP^i)$ is a weak solution of the lifted SEE \eqref{intro_eq_SEE-lift} with the initial condition $y\in\cH_\mu$. Furthermore, it holds that $X^i=\mu[Y^i]$ in $L^2_\loc(\bR_+;\bR^n)$ $\bP^i$-a.s. By the assumption, we have $\Law_{\bP^1}(Y^1)=\Law_{\bP^2}(Y^2)$. Noting that the map $(y_t)_{t\geq0}\mapsto(\mu[y_t])_{t\geq0}$ is a continuous linear map from $\Lambda$ to $L^2_\loc(\bR_+;\bR^n)$, we have $\Law_{\bP^1}(X^1)=\Law_{\bP^2}(X^2)$. Thus, uniqueness in law holds for the SVE \eqref{intro_eq_SVE} with the forcing term $x=\cK y$.
\end{proof}


\begin{rem}
The opposite direction of the assertion (iv) of \cref{sol_cor_equivalence} is not trivial. Indeed, even if two weak solutions $(X^i,W^i,\Omega^i,\cF^i,\bF^i,\bP^i)$, $i=1,2$, of the SVE \eqref{intro_eq_SVE} satisfy $\Law_{\bP^1}(X^1)=\Law_{\bP^2}(X^2)$, the stochastic convolutions
\begin{equation*}
	\int^t_0e^{-\cdot(t-s)}b(X^i_s)\,\diff s+\int^t_0e^{-\cdot(t-s)}\sigma(X^i_s)\,\diff W^i_s,\ \ t\geq0,\ i=1,2,
\end{equation*}
might have different laws in the space $\Lambda$. We note that the equivalence in (iv) of \cref{sol_cor_equivalence} would hold true if we replaced the notion of uniqueness in law by \emph{joint} uniqueness in law. Here, joint uniqueness in law means that $\Law_{\bP^1}(X^1,W^1)=\Law_{\bP^2}(X^2,W^2)$ holds for any two weak solutions $(X^i,W^i,\Omega^i,\cF^i,\bF^i,\bP^i)$, $i=1,2$, of the SVE \eqref{intro_eq_SVE} with the same forcing term $x$; similarly we can define it for the lifted SEE \eqref{intro_eq_SEE-lift}. Cherny \cite{Ch03} showed that, in the SDEs case, the two notions of uniqueness in law and joint uniqueness in law are equivalent. Whether such an equivalence of the two notions holds in the SVEs case is an important problem. However, since the assertion (iv) of \cref{sol_cor_equivalence} is sufficient for our purpose in this paper, we leave this problem to the future research.
\end{rem}

By means of the above results, we try to show weak well-posedness of the lifted SEE \eqref{intro_eq_SEE-lift}, and then to translate it into the original SVE \eqref{intro_eq_SVE}.


\section{Main results}\label{main}

Before stating our main results, let us introduce some additional notations.

Let $r$ be the function appearing in \cref{sol_assum_kernel}. For each $m\in[1,\infty)$, define
\begin{equation}\label{main_eq_r_m}
	r_m(\theta):=\frac{r(m\vee\theta)}{r(m)}=
	\begin{dcases}
	1\ &\text{for $\theta\in[0,m]$},\\
	\frac{r(\theta)}{r(m)}\ &\text{for $\theta\in(m,\infty)$},
	\end{dcases}
\end{equation}
and
\begin{equation}\label{main_eq_R-ep}
	R_m:=\int_\supp r_m(\theta)\dmu,\ \ep_m:=\int_\supp(1-r_m(\theta))^2\theta^{-1}r_m(\theta)^{-1}\dmu.
\end{equation}
Since $r$ is non-increasing and $1\wedge(\theta^{-1/2})\leq r(\theta)\leq1$, we have, for $1\leq m\leq M<\infty$,
\begin{equation*}
	1\wedge(\theta^{-1/2})\leq r(\theta)=r_1(\theta)\leq r_m(\theta)\leq r_M(\theta)\leq1
\end{equation*}
for any $\theta\in[0,\infty)$, and
\begin{equation*}
	R_m\leq R_M\leq\frac{1}{r(M)}\int_\supp r(\theta)\dmu<\infty,\ \ep_M\leq\ep_m\leq r(m)\int_{\supp\cap(m,\infty)}r(\theta)\dmu.
\end{equation*}
In particular, by the dominated convergence theorem, we have
\begin{equation}\label{main_eq_R-ep-lim}
	\lim_{m\to\infty}\frac{\ep_m}{r(m)}=0\ \text{and}\ \lim_{m\to\infty}R_m\ep_m=0.
\end{equation}
On the one hand, if the kernel $K$ is regular (or equivalently if $\mu([0,\infty))<\infty$), then we take $r\equiv1$, and we see that $R_m=\mu([0,\infty))<\infty$ and $\ep_m=0$ for any $m\in[1,\infty)$. On the other hand, if the kernel $K$ is singular (or equivalently if $\mu([0,\infty))=\infty$), then $\lim_{m\to\infty}R_m=\infty$ and $\ep_m>0$ for any $m\in[1,\infty)$.

Next, for each uniformly continuous map $\varphi$ defined on $\bR^n$ which takes values in either $\bR^n$ or $\bR^{n\times d}$, we define
\begin{equation*}
	\rho_\varphi(t):=\inf_{(a,b)\in\cA_\varphi}(at+b),\ t\geq0,
\end{equation*}
where
\begin{equation*}
	\cA_\varphi:=\left\{(a,b)\in[0,\infty)\times[0,\infty)\relmiddle||\varphi(x)-\varphi(\bar{x})|^2\leq a|x-\bar{x}|^2+b\ \text{for any $x,\bar{x}\in\bR^n$}\right\}.
\end{equation*}
By the uniform continuity of $\varphi$, it can be easily shown that, for any $\ep>0$, there exists $a_\ep>0$ such that $(a_\ep,\ep)\in\cA_\varphi$. This implies that $\rho_\varphi(0)=0$. Also, we see that $\rho_\varphi:[0,\infty)\to[0,\infty)$ is non-decreasing, concave and satisfies
\begin{equation*}
	|\varphi(x)-\varphi(\bar{x})|^2\leq\rho_\varphi(|x-\bar{x}|^2)
\end{equation*}
for any $x,\bar{x}\in\bR^n$. The function $\rho_\varphi$ can be seen as a modulus of continuity of $\varphi$, which is the minimal choice among all functions satisfying the same properties. That is, if $\rho:[0,\infty)\to[0,\infty)$ is a non-decreasing concave function with $\rho(0)=0$ such that $|\varphi(x)-\varphi(\bar{x})|^2\leq\rho(|x-\bar{x}|^2)$ for any $x,\bar{x}\in\bR^n$, then it holds that $\rho_\varphi(t)\leq\rho(t)$ for any $t\geq0$. The concavity and $\rho_\varphi(0)=0$ imply that $\rho_\varphi$ is sub-additive in the sense that $\rho_\varphi(s+t)\leq\rho_\varphi(s)+\rho_\varphi(t)$ for any $s,t\geq0$, which also implies that $\rho_\varphi(st)\leq(1+s)\rho_\varphi(t)$ for any $s,t\geq0$.

Let $(E,d)$ be a Polish space. We denote the space of probability measures on $E$ (endowed with the Borel $\sigma$-algebra $\cB(E)$) by $\cP(E)$. The \emph{L\'{e}vy--Prokhorov metric} between two probability measures $\nu_1,\nu_2\in\cP(E)$ is defined by
\begin{equation*}
	d_\LP(\nu_1,\nu_2):=\inf\left\{\ep>0\relmiddle|\nu_1(A)\leq\nu_2(A^\ep)+\ep\ \text{and}\ \nu_2(A)\leq\nu_1(A^\ep)+\ep\ \text{for all}\ A\in\cB(E)\right\},
\end{equation*}
where $A^\ep:=\left\{x\in E\relmiddle|\text{there exists $y\in A$ such that}\ d(x,y)<\ep\right\}$. It is well-known that $(\cP(E),d_\LP)$ is a Polish space and that the convergence with respect to $d_\LP$ is equivalent to the weak convergence of probability measures.

Now we are ready to state our main result.


\begin{theo}\label{main_theo_main}
Suppose that \cref{sol_assum_kernel} holds and that $b:\bR^n\to\bR^n$ and $\sigma:\bR^n\to\bR^{n\times d}$ are uniformly continuous. Assume that $\sigma$ satisfies the uniform ellipticity condition; there exists a constant $c_\UE>0$ such that $\sigma(x)\sigma(x)^\top\geq c_\UE I_{n\times n}$ for any $x\in\bR^n$. Furthermore, assume either the following (i) or (ii) holds:
\begin{itemize}
\item[(i)]
The kernel $K$ is singular, and it holds that $\displaystyle\liminf_{m\to\infty}R_m\frac{\rho_\sigma(\ep^2_m)}{\ep_m}=0$.
\item[(ii)]
The kernel $K$ is regular, and it holds that $\displaystyle\liminf_{\delta\downarrow0}\frac{\rho_\sigma(\delta^2)}{\delta}=0$.
\end{itemize}
Then, for the SEE \eqref{intro_eq_SEE-lift}:
\begin{equation*}
	\begin{dcases}
	\diff Y_t(\theta)=-\theta Y_t(\theta)\,\diff t+b(\mu[Y_t])\,\diff t+\sigma(\mu[Y_t])\,\diff W_t,\ \ \theta\in\supp,\ t>0,\\
	Y_0(\theta)=y(\theta),\ \ \theta\in\supp,
	\end{dcases}
\end{equation*}
the following hold:
\begin{itemize}
\item
Weak existence and uniqueness in law hold for any initial condition $y\in\cH_\mu$.
\item
The map $y\mapsto \Law_{\bP^y}(Y^y)$ from $(\cH_\mu,\|\cdot\|_{\cH_\mu})$ to $(\cP(\Lambda),d_\LP)$ is continuous, where $(Y^y,W^y,\Omega^y,\cF^y,\bF^y,\bP^y)$ is a weak solution with the initial condition $y\in\cH_\mu$.
\item
For any weak solution $(Y,W,\Omega,\cF,\bF,\bP)$, the process $Y$ is a time-homogeneous Markov process on $\cH_\mu$, which satisfies the Feller property.
\end{itemize}
In particular, for the SVE \eqref{intro_eq_SVE}:
\begin{equation*}
	X_t=x(t)+\int^t_0K(t-s)b(X_s)\,\diff s+\int^t_0K(t-s)\sigma(X_s)\,\diff W_s,\ \ t>0,
\end{equation*}
the following hold:
\begin{itemize}
\item
Weak existence and uniqueness in law hold for any forcing term $x$ of the form $x=\cK y$ for some $y\in\cH_\mu$.
\item
The map $y\mapsto\Law_{\bP^y}(X^y)$ from $(\cH_\mu,\|\cdot\|_{\cH_\mu})$ to $(\cP(L^2_\loc(\bR_+;\bR^n)),d_\LP)$ is continuous, where, for each $y\in\cH_\mu$, $(X^y,W^y,\Omega^y,\cF^y,\bF^y,\bP^y)$ denotes a weak solution of the SVE with the forcing term $x=\cK y$.
\end{itemize}
\end{theo}


\begin{rem}
The pair of conditions (i) and (ii) in the above theorem stands for the ``balance condition'' between the singularity of the kernel $K$ and the modulus of continuity of the diffusion coefficient $\sigma$. This means that the more singular the kernel $K$ is, the more regular the diffusion coefficient $\sigma$ must be, and vice versa; see \cref{main_cor_Holder} below. In particular, $\sigma$ is at least $1/2$-H\"{o}lder continuous in the sense that $\displaystyle\liminf_{\delta\downarrow0}\frac{\rho_\sigma(\delta^2)}{\delta}=0$. In contrast, no additional conditions are assumed on the drift coefficient $b$ other than the uniform continuity.
\end{rem}

As a special case of \cref{main_theo_main}, we shall consider the case of H\"{o}lder continuous diffusion coefficients. We say that a map $\sigma:\bR^n\to\bR^{n\times d}$ satisfies the \emph{finite range H\"{o}lder condition} with exponent $\gamma\in(0,1]$ if there exists a constant $c_\Hol\in(0,\infty)$ such that
\begin{equation*}
	|\sigma(x)-\sigma(\bar{x})|\leq c_\Hol|x-\bar{x}|^\gamma
\end{equation*}
for any $x,\bar{x}\in\bR^n$ satisfying $|x-\bar{x}|\leq1$. This condition implies that
\begin{equation*}
	|\sigma(x)-\sigma(\bar{x})|\leq c_\Hol\big(|x-\bar{x}|^\gamma+|x-\bar{x}|\big)
\end{equation*}
for any $x,\bar{x}\in\bR^n$ satisfying $|x-\bar{x}|>1$. Indeed, taking $p\in\bN$ such that $|x-\bar{x}|\in(p,p+1]$, we have
\begin{align*}
	|\sigma(x)-\sigma(\bar{x})|&\leq\sum^{p+1}_{q=1}\Big|\sigma\Big(\bar{x}+\frac{q}{p+1}(x-\bar{x})\Big)-\sigma\Big(\bar{x}+\frac{q-1}{p+1}(x-\bar{x})\Big)\Big|\\
	&\leq c_\Hol(p+1)\Big(\frac{|x-\bar{x}|}{p+1}\Big)^\gamma\\
	&\leq c_\Hol\big(|x-\bar{x}|^\gamma+|x-\bar{x}|\big).
\end{align*}
Therefore, $|\sigma(x)-\sigma(\bar{x})|^2\leq\rho\big(|x-\bar{x}|^2\big)$ for any $x,\bar{x}\in\bR^n$, where $\rho(t):=2c^2_\Hol\big(t^\gamma+t\big)$, $t\geq0$, which is a concave function on $[0,\infty)$. This implies that $\rho_\sigma\leq\rho$, and hence $\limsup_{t\downarrow0}\frac{\rho_\sigma(t)}{t^\gamma}<\infty$.


\begin{cor}\label{main_cor_Holder}
Suppose that the kernel $K$ is completely monotone and satisfies $\int^1_0t^{\eta-1}K(t)\,\diff t<\infty$ for some $\eta\in(0,\frac{1}{2}]$. Let $b:\bR^n\to\bR^n$ be a uniformly continuous map. Suppose that $\sigma:\bR^n\to\bR^{n\times d}$ satisfies the uniform ellipticity condition and the finite range H\"{o}lder condition with exponent $\gamma=\frac{1}{2(1-\eta)}\in(\frac{1}{2},1]$. Then, for any free term $x$ of the form $x(t)=K(t)x_0$ for any $t>0$ for some $x_0\in\bR^n$, weak existence and uniqueness in law hold for the SVE \eqref{intro_eq_SVE}.
\end{cor}


\begin{proof}
We check the conditions (i) and (ii) in \cref{main_theo_main}. If the kernel $K$ is regular, then we have $\lim_{\delta\downarrow0}\frac{\rho_\sigma(\delta^2)}{\delta}=0$ since $\limsup_{t\downarrow0}\frac{\rho_\sigma(t)}{t^\gamma}<\infty$ and $\gamma>\frac{1}{2}$, and thus the condition (ii) in \cref{main_theo_main} holds. Assume that $K$ is singular. By the assumption and \cite[Lemma 2.1]{Ha23}, we can and will take $r(\theta)=1\wedge(\theta^{-\eta})$ as the function $r$ appearing in \cref{sol_assum_kernel}. Let $m\in[1,\infty)$ be fixed. Observe that $\theta^{-1}r_m(\theta)^{-1}=m^{-\eta}\theta^{\eta-1}\leq r(m)^{\frac{1}{\eta}-1}r(\theta)$ for any $\theta\in(m,\infty)$. Thus, we have
\begin{equation*}
	\ep_m=\int_\supp(1-r_m(\theta))^2\theta^{-1}r_m(\theta)^{-1}\dmu\leq r(m)^{\frac{1}{\eta}-1}\int_{\supp\cap(m,\infty)}r(\theta)\dmu
\end{equation*}
and
\begin{align*}
	R_m\ep^{2\gamma-1}_m&\leq r(m)^{(\frac{1}{\eta}-1)(2\gamma-1)-1}\Big(\int_{\supp\cap(m,\infty)}r(\theta)\dmu\Big)^{2\gamma-1}\Big(\int_\supp r(\theta)\dmu\Big)\\
	&=\Big(\int_{\supp\cap(m,\infty)}r(\theta)\dmu\Big)^{2\gamma-1}\Big(\int_\supp r(\theta)\dmu\Big).
\end{align*}
Noting that $\int_\supp r(\theta)\dmu<\infty$, by the dominated convergence theorem, we get $\lim_{m\to\infty}R_m\ep^{2\gamma-1}_m=0$. Since $\limsup_{t\downarrow0}\frac{\rho_\sigma(t)}{t^\gamma}<\infty$, the condition (i) in \cref{main_theo_main} holds.
\end{proof}


\begin{cor}\label{main_cor_Lip}
Suppose that the kernel $K$ is completely monotone and satisfies $\int^1_0t^{-1/2}K(t)\,\diff t<\infty$. Let $b:\bR^n\to\bR^n$ be a uniformly continuous map. Assume that $\sigma:\bR^n\to\bR^{n\times d}$ is globally Lipschitz continuous and satisfies the uniform ellipticity condition. Then, for any free term $x$ of the form $x(t)=K(t)x_0$ for any $t>0$ for some $x_0\in\bR^n$, weak existence and uniqueness in law hold for the SVE \eqref{intro_eq_SVE}.
\end{cor}


\begin{cor}\label{main_cor_fractional}
Consider the fractional kernel $K(t)=\frac{1}{\Gamma(\alpha)}t^{\alpha-1}$ with exponent $\alpha\in(\frac{1}{2},1]$. Let $b:\bR^n\to\bR^n$ be a uniformly continuous map. Assume that $\sigma:\bR^n\to\bR^{n\times d}$ satisfies the uniform ellipticity condition and the finite range H\"{o}lder condition with exponent $\gamma\in(\frac{1}{2},1]$. Furthermore, assume that the \emph{balance condition} $\alpha\gamma>\frac{1}{2}$ holds. Then, for any free term $x$ of the form $x(t)=\frac{1}{\Gamma(\alpha)}t^{\alpha-1}x_0$ for any $t>0$ for some $x_0\in\bR^n$, weak existence and uniqueness in law hold for the SVE \eqref{intro_eq_SVE}.
\end{cor}


\begin{rem}
Consider the fractional kernel $K(t)=\frac{1}{\Gamma(\alpha)}t^{\alpha-1}$ with exponent $\alpha\in(\frac{1}{2},1)$. Note that the corresponding Radon measure is given by $\mu(\diff\theta)=\frac{1}{\Gamma(\alpha)\Gamma(1-\alpha)}\theta^{-\alpha}\1_{(0,\infty)}\,\diff\theta$. As the function $r$ in \cref{sol_assum_kernel} we can take $r(\theta)=1\wedge(\theta^{-\eta})$ for any $\eta\in(1-\alpha,\frac{1}{2}]$. In this case, for each $m\in[1,\infty)$, the numbers $R_m$ and $\ep_m$ can be calculated explicitly as follows:
\begin{align*}
	R_m&=\int_\supp r_m(\theta)\dmu\\
	&=\frac{1}{\Gamma(\alpha)\Gamma(1-\alpha)}\int^m_0\theta^{-\alpha}\,\diff\theta+\frac{1}{\Gamma(\alpha)\Gamma(1-\alpha)}m^\eta\int^\infty_m\theta^{-\eta-\alpha}\,\diff\theta\\
	&=\frac{1}{\Gamma(\alpha)\Gamma(1-\alpha)}\cdot\frac{\eta}{(1-\alpha)(\alpha+\eta-1)}m^{1-\alpha}
\end{align*}
and
\begin{align*}
	\ep_m&=\int_\supp(1-r_m(\theta))^2\theta^{-1}r_m(\theta)^{-1}\dmu\\
	&=\frac{1}{\Gamma(\alpha)\Gamma(1-\alpha)}m^{-\eta}\int^\infty_m(1-m^\eta\theta^{-\eta})^2\theta^{-1+\eta-\alpha}\,\diff\theta\\
	&=\frac{1}{\Gamma(\alpha)\Gamma(1-\alpha)}\cdot\frac{2\eta^2}{\alpha(\alpha-\eta)(\alpha+\eta)}m^{-\alpha}.
\end{align*}
\end{rem}


\begin{exam}\label{main_exam_regularization}
Consider the following (one-dimensional) deterministic Volterra equation with the fractional kernel:
\begin{equation}\label{main_eq_deterministic}
	X_t=\frac{1}{\Gamma(\alpha)}\int^t_0(t-s)^{\alpha-1}|X_s|^\beta\mathrm{sign}(X_s)\,\diff s,\ t>0,
\end{equation}
with $\alpha\in(0,1]$ and $\beta\in(0,1)$. Obviously, the function $X_t=0$, $t>0$, is a solution of the above Volterra equation. Besides, for any $t_0\geq0$, the functions
\begin{equation*}
	X_t=\pm C_{\alpha,\beta}(t-t_0)^{\frac{\alpha}{1-\beta}}\1_{(t_0,\infty)}(t),\ t>0,
\end{equation*}
with the constant $C_{\alpha,\beta}=\Big(\beta\frac{\Gamma(\frac{\alpha\beta}{1-\beta})}{\Gamma(\frac{\alpha}{1-\beta})}\Big)^{\frac{1}{1-\beta}}$, are also solutions. Indeed, for $t>t_0$,
\begin{align*}
	\frac{1}{\Gamma(\alpha)}\int^t_0(t-s)^{\alpha-1}|X_s|^\beta\mathrm{sign}(X_s)\,\diff s&=\pm\frac{C^\beta_{\alpha,\beta}}{\Gamma(\alpha)}\int^t_{t_0}(t-s)^{\alpha-1}(s-t_0)^{\frac{\alpha\beta}{1-\beta}}\,\diff s\\
	&=\pm\frac{C^\beta_{\alpha,\beta}}{\Gamma(\alpha)}\int^{t-t_0}_0(t-t_0-s)^{\alpha-1}s^{\frac{\alpha\beta}{1-\beta}}\,\diff s\\
	&=\pm\frac{C^\beta_{\alpha,\beta}}{\Gamma(\alpha)}B\Big(\alpha,\frac{\alpha\beta}{1-\beta}+1\Big)(t-t_0)^{\alpha+\frac{\alpha\beta}{1-\beta}}\\
	&=\pm C_{\alpha,\beta}(t-t_0)^{\frac{\alpha}{1-\beta}}\\
	&=X_t,
\end{align*}
where $B(x,y)=\int^1_0(1-s)^{x-1}s^{y-1}\,\diff s$, $x,y>0$, denotes the Beta function; we used the relations $B(x,y)=\frac{\Gamma(x)\Gamma(y)}{\Gamma(x+y)}$ and $\Gamma(x+1)=x\Gamma(x)$ for $x,y>0$. Thus, the deterministic Volterra equation \eqref{main_eq_deterministic} has infinitely many solutions. In contrast, if $\alpha\in(\frac{1}{2},1]$, \cref{main_cor_fractional} shows that weak existence and uniqueness in law hold for the corresponding SVE
\begin{equation*}
	X_t=\frac{1}{\Gamma(\alpha)}\int^t_0(t-s)^{\alpha-1}|X_s|^\beta\mathrm{sign}(X_s)\,\diff s+\frac{1}{\Gamma(\alpha)}\int^t_0(t-s)^{\alpha-1}\sigma(X_s)\,\diff W_s,\ \ t>0,
\end{equation*}
whenever the diffusion coefficient $\sigma:\bR\to\bR$ is non-degenerate and satisfies the finite range H\"{o}lder condition with exponent $\gamma\in(\frac{1}{2\alpha},1]$. This reveals the regularization-by-noise effect for SVEs, allowing for multiplicative noise with H\"{o}lder coefficients.
\end{exam}


\begin{rem}
Abi Jaber et al.\ \cite{AbiJaCuLaPu21} proved weak existence of SVEs (with jumps) under some mild assumptions. In their main result \cite[Theorem 1.2]{AbiJaCuLaPu21}, they assumed that the kernel $K$ satisfies
\begin{equation}\label{main_eq_Abi-Jaber}
	\int^T_0\frac{K(t)^p}{t^{\eta p}}\,\diff t+\int^T_0\int^T_0\frac{|K(t)-K(s)|^p}{|t-s|^{1+\eta p}}\,\diff s\,\diff t<\infty
\end{equation}
for any $T>0$ for some $p\in[2,\infty)$ and $\eta\in(0,1)$. This assumption was used to show a priori estimates on Sobolev--Slobodeckij norms of weak solutions of SVEs, which was crucial for the tightness argument; see \cite[Theorem 1.4 and Corollary 1.5]{AbiJaCuLaPu21}. Observe that if $\int^1_0\frac{K(t)^p}{t^{\eta p}}\,\diff t<\infty$, by the H\"{o}lder inequality one has
\begin{equation*}
	\int^1_0t^{\alpha-1}K(t)\,\diff t\leq\Big(\int^1_0t^{\frac{p}{p-1}(\alpha+\eta-1)}\,\diff t\Big)^{\frac{p-1}{p}}\Big(\int^1_0\frac{K(t)^p}{t^{\eta p}}\,\diff t\Big)^{\frac{1}{p}}<\infty
\end{equation*}
for any $\alpha\in(\frac{1}{p}-\eta,\frac{1}{2})$. If the kernel $K$ is completely monotone with the corresponding Radon measure $\mu$, then $\int^1_0t^{\alpha-1}K(t)\,\diff t<\infty$ is equivalent to $\int_\supp1\wedge(\theta^{-\alpha})\dmu<\infty$; see \cite[Lemma 2.1]{Ha23}. Thus, we cannot apply \cite[Theorem 1.2]{AbiJaCuLaPu21} to show weak existence for an SVE under \cref{sol_assum_kernel}. Indeed, for example, the completely monotone kernel $K(t)=\int^\infty_2\frac{e^{-\theta t}}{\theta^{1/2}(\log\theta)^2}\,\diff\theta$ with the corresponding Radon measure $\dmu=\frac{1}{\theta^{1/2}(\log\theta)^2}\1_{[2,\infty)}(\theta)\,\diff\theta$ satisfies \cref{sol_assum_kernel} with $r(\theta)=1\wedge(\theta^{-1/2})$:
\begin{equation*}
	\int_\supp r(\theta)\dmu=\int^\infty_2\frac{\diff\theta}{\theta(\log\theta)^2}=\frac{1}{\log2}<\infty,
\end{equation*}
but does not satisfy \eqref{main_eq_Abi-Jaber} since, for any $\alpha<\frac{1}{2}$,
\begin{equation*}
	\int_\supp1\wedge(\theta^{-\alpha})\dmu=\int^\infty_2\frac{\diff\theta}{\theta^{\frac{1}{2}+\alpha}(\log\theta)^2}=\infty.
\end{equation*}
\end{rem}


\section{Proof of the main theorem}\label{proof}

In this section, we prove \cref{main_theo_main}. Our aim is to construct a sequence $\{Y^k\}_{k\in\bN}$ of mild solutions of some (strongly) well-posed SEEs and show that:
\begin{itemize}
\item
the sequence of the laws of $\{Y^k\}_{k\in\bN}$ is a Cauchy sequence in $\cP(\Lambda)$ with respect to the weak convergence topology, which implies weak existence of the original SEE \eqref{intro_eq_SEE-lift}, and that
\item
for an arbitrarily given weak solution $(Y,W,\Omega,\cF,\bF,\bP)$ of the original SEE \eqref{intro_eq_SEE-lift}, the law of $Y$ is specified as the weak limit of the laws of $\{Y^k\}_{k\in\bN}$, which implies uniqueness in law.
\end{itemize}
We will construct such a sequence in \cref{proof-main}. In order to show the above assertions, we need to estimate the L\'{e}vy--Prokhorov metric between $\Law_\bP(Y)$ and $\Law_\bP(\bar{Y})$, where $(Y,W,\Omega,\cF,\bF,\bP)$ is a given weak solution of an SEE, and $\bar{Y}$ is a mild solution of a (strongly) well-posed SEE approximating the original one. In the Lipschitz coefficients case, we can estimate the difference $Y-\bar{Y}$ in the strong sense by a standard method; see \cite[Theorem 2.17]{Ha23}. However, this kind of estimate is very difficult in the general non-Lipschitz (and even non-H\"{o}lder) setting, and we need to compare their probability laws in a more sophisticated way. Inspired by the work \cite{KlSc20} about SDEs with delay, we will perform the so-called \emph{Control-and-Reimburse (C-n-R)} strategy. Namely, instead of estimating $d_\LP(\Law_\bP(Y),\Law_\bP(\bar{Y}))$ directly, we divide the estimate into two steps:
\begin{equation}\label{proof_eq_CnR}
	d_\LP\big(\Law_\bP(Y),\Law_\bP(\bar{Y})\big)\leq\underbrace{d_\LP\big(\Law_\bP(Y),\Law_\bP(\hat{Y})\big)}_{\text{control}}+\underbrace{d_\LP\big(\Law_\bP(\hat{Y}),\Law_\bP(\bar{Y})\big)}_{\text{reimbursement}},
\end{equation}
where $\hat{Y}$ is a ``controlled version'' of $\bar{Y}$. The first term of the right-hand side of \eqref{proof_eq_CnR} is about the ``control-step'' where we aim to improve $\bar{Y}$ by an appropriately controlled version $\hat{Y}$ so that the difference $Y-\hat{Y}$ is easier to estimate in the strong sense; we will estimate its modified (but equivalent) norm in terms of the topology of the convergence in probability on $(\Omega,\cF,\bP)$; see \eqref{proof_eq_estimate-control} below. The second term of the right-hand side of \eqref{proof_eq_CnR} is about the ``reimbursement-step'' where we have to reimburse the impact of the control; we will estimate the total variation distance between the law of the controlled version $\hat{Y}$ and the law of $\bar{Y}$; see \eqref{proof_eq_estimate-reimburse} below.

The rest of this section is organized as follows: In \cref{proof-pre}, we summarize some preliminary results. In \cref{proof-CnR}, we provide key estimates for the C-n-R strategy. Then, in \cref{proof-main}, we prove \cref{main_theo_main}. After that, in \cref{proof-remark}, we make some remarks on the proof of our main theorem.

\subsection{Preliminaries}\label{proof-pre}

Suppose that \cref{sol_assum_kernel} holds. Recall the definitions \eqref{main_eq_r_m} and \eqref{main_eq_R-ep} of $r_m$, $R_m$ and $\ep_m$. For each $m\in[1,\infty)$, define
\begin{equation*}
	\mu_m[y]:=\int_\supp r_m(\theta)y(\theta)\dmu,\ \|y\|_m:=\Big(\int_\supp r_m(\theta)|y(\theta)|^2\dmu\Big)^{1/2},\ \ y\in\cH_\mu.
\end{equation*}
For $m=\infty$, we set $r_\infty(\theta):=1$ for any $\theta\in[0,\infty)$ and $\mu_\infty[\cdot]:=\mu[\cdot]$; see \eqref{sol_eq_mu}.


\begin{lemm}\label{proof_lemm_mu}
Suppose that \cref{sol_assum_kernel} holds.
\begin{itemize}
\item[(i)]
Let $m\in[1,\infty)$. Then, it holds that
\begin{equation*}
	\|y\|_{\cH_\mu}\leq\|y\|_m\leq r(m)^{-1/2}\|y\|_{\cH_\mu}
\end{equation*}
for any $y\in\cH_\mu$. In particular, the norm $\|\cdot\|_m$ is equivalent to the original norm $\|\cdot\|_{\cH_\mu}$ in $\cH_\mu$.
\item[(ii)]
Let $m\in[1,\infty)$. Then, it holds that
\begin{equation*}
	\big|\mu_m[y]\big|^2\leq R_m\|y\|^2_m
\end{equation*}
for any $y\in\cH_\mu$. In particular, the map $\mu_m[\cdot]$ is a bounded linear operator from $\cH_\mu$ to $\bR^n$.
\item[(iii)]
Let $m\in[1,\infty)$ and $M\in[m,\infty]$. Then, it holds that
\begin{equation*}
	\big|\mu_M[y]-\mu_m[y]\big|^2\leq\ep_m\int_\supp\theta r_m(\theta)|y(\theta)|^2\dmu\leq\frac{\ep_m}{r(m)}\int_\supp\theta r(\theta)|y(\theta)|^2\dmu
\end{equation*}
for any $y\in\cV_\mu$. In particular, $\lim_{m\to\infty}\mu_m[\cdot]|_{\cV_\mu}=\mu[\cdot]|_{\cV_\mu}$ with respect to the operator norm of bounded linear operators from $\cV_\mu$ to $\bR^n$.
\end{itemize}
\end{lemm}


\begin{proof}
The assertion (i) follows from the estimates $r(\theta)\leq r_m(\theta)\leq\frac{r(\theta)}{r(m)}$. Noting the definitions of $R_m$, $\mu_m[\cdot]$ and $\|\cdot\|_m$, by the Cauchy--Schwarz inequality,
\begin{equation*}
	\big|\mu_m[y]\big|^2=\Big|\int_\supp r_m(\theta)y(\theta)\dmu\Big|^2\leq\int_\supp r_m(\theta)\dmu\,\int_\supp r_m(\theta)|y(\theta)|^2\dmu=R_m\|y\|^2_m
\end{equation*}
for any $y\in\cH_\mu$. Thus, the assertion (ii) holds. Concerning the assertion (iii), note that $0\leq r_m(\theta)\leq r_M(\theta)\leq1$. Then, by the Cauchy--Schwarz inequality, we have
\begin{align*}
	\big|\mu_M[y]-\mu_m[y]\big|^2&=\Big|\int_\supp(r_M(\theta)-r_m(\theta))y(\theta)\dmu\Big|^2\\
	&\leq\int_\supp(1-r_m(\theta))^2\theta^{-1}r_m(\theta)^{-1}\dmu\,\int_\supp\theta r_m(\theta)|y(\theta)|^2\dmu\\
	&=\ep_m\int_\supp\theta r_m(\theta)|y(\theta)|^2\dmu
\end{align*}
for any $y\in\cV_\mu$. Thus, the first inequality holds. The second inequality follows from the inequality $r_m(\theta)\leq\frac{r(\theta)}{r(m)}$. By \eqref{main_eq_R-ep-lim}, we see that $\mu_m[\cdot]|_{\cV_\mu}$ converges to $\mu_\infty[\cdot]|_{\cV_\mu}=\mu[\cdot]|_{\cV_\mu}$ as $m\to\infty$ in the operator norm of bounded linear operators from $\cV_\mu$ to $\bR^n$. This completes the proof.
\end{proof}

For each $M\in[1,\infty]$, consider the following SEE:
\begin{equation}\label{proof_eq_SEE-M}
	\begin{dcases}
	\diff Y_t(\theta)=-\theta Y_t(\theta)\,\diff t+b(\mu_M[Y_t])\,\diff t+\sigma(\mu_M[Y_t])\,\diff W_t,\ \theta\in\supp,\ t>0,\\
	Y_0(\theta)=y(\theta),\ \theta\in\supp,
	\end{dcases}
\end{equation}
where $b:\bR^n\to\bR^n$ and $\sigma:\bR^n\to\bR^{n\times d}$ are measurable maps. Note that the above SEE with $M=\infty$ coincides with the original SEE \eqref{intro_eq_SEE-lift}. Also, by \cref{sol_prop_equiv}, the SEE \eqref{proof_eq_SEE-M} with $M<\infty$ corresponds to the SVE
\begin{equation*}
	X_t=x_M(t)+\int^t_0K_M(t-s)b(X_s)\,\diff s+\int^t_0K_M(t-s)\sigma(X_s)\,\diff W_s,\ \ t>0,
\end{equation*}
where $x_M(t)=\int^t_0e^{-\theta t}y(\theta)r_M(\theta)\dmu$ and $K_M(t):=\int_\supp e^{-\theta t}r_M(\theta)\dmu$ for $t>0$. We note that the kernel $K_M$ is regular if $M<\infty$.


\begin{lemm}\label{poof_lemm_estimate}
Suppose that \cref{sol_assum_kernel} holds. Assume that the maps $b$ and $\sigma$ satisfy the linear growth condition; there exists a constant $c_\LG>0$ such that, for $\varphi=b,\sigma$,
\begin{equation}\label{pre_eq_LG}
	|\varphi(x)|\leq c_\LG(1+|x|)
\end{equation}
for any $x\in\bR^n$. Then, there exists a constant $C_0=C_0(\mu,c_\LG)>0$ such that, for any $M\in[1,\infty]$, any initial condition $y\in\cH_\mu$, any weak solution $(Y,W,\Omega,\cF,\bF,\bP)$ to the SEE \eqref{proof_eq_SEE-M} and any $T>0$, the following holds:
\begin{equation*}
	\bE\Big[\sup_{t\in[0,T]}\|Y_t\|^2_{\cH_\mu}+\int^T_0\|Y_t\|^2_{\cV_\mu}\,\diff t\Big]\leq C_0e^{C_0T}(1+\|y\|^2_{\cH_\mu}).
\end{equation*}
\end{lemm}


\begin{proof}
This lemma can be proved by the same arguments as in the proof of \cite[Theorem 2.16]{Ha23} where the case $M=\infty$ was considered; though we cannot use this theorem since we need a uniform estimate with respect to $M\in[1,\infty]$. Nevertheless, we can apply its proof to our lemma by noting that $|\mu_M[y]|\leq\|y\|_{L^1(\mu)}$ for any $y\in\cV_\mu$ and any $M\in[1,\infty]$; the key is that the right-hand side does not depend on $M$. This kind of estimate is essential in the proof of \cite[Theorem 2.16]{Ha23}, and then completely the same arguments show the desired a priori estimate with the constant $C_0=C_0(\mu,c_\LG)$ which does not depend on $M\in[1,\infty]$, $y\in\cH_\mu$, or the choice of the weak solution $(Y,W,\Omega,\cF,\bF,\bP)$.
\end{proof}


\begin{lemm}\label{pre_lemm_SEE-M}
Suppose that \cref{sol_assum_kernel} holds. Assume that $b$ and $\sigma$ satisfy the linear growth condition and the local Lipschitz condition; for any $k\in\bN$, there exists a constant $L_k>0$ such that, for $\varphi=b,\sigma$,
\begin{equation*}
	|\varphi(x)-\varphi(\bar{x})|\leq L_k|x-\bar{x}|
\end{equation*}
for any $x,\bar{x}\in\bR^n$ satisfying $|x|\vee|\bar{x}|\leq k$. Let $M\in[1,\infty)$, that is, $M$ is assumed to be finite. Then, for any initial condition $y\in\cH_\mu$, weak existence and pathwise uniqueness hold for the SEE \eqref{proof_eq_SEE-M}. In particular, uniqueness in law and strong existence hold.

Under the above assumptions, let a filtered probability space $(\Omega,\cF,\bF,\bP)$ and a $d$-dimensional $\bF$-Brownian motion $W$ be given. For each $y\in\cH_\mu$, denote by $Y^y$ the mild solution of the SEE \eqref{proof_eq_SEE-M} with the initial condition $y$. Then, the following hold:
\begin{itemize}
\item
The map $y\mapsto Y^y$ is continuous in probability as a map from $(\cH_\mu,\|\cdot\|_{\cH_\mu})$ to the space of $\Lambda$-valued random variables on $(\Omega,\cF,\bP)$. In particular, the map $y\mapsto\Law_\bP(Y^y)$ from $(\cH_\mu,\|\cdot\|_{\cH_\mu})$ to $(\cP(\Lambda),d_\LP)$ is continuous.
\item
The process $Y^y$ is a time-homogeneous Markov process on $\cH_\mu$, which has the Feller property.
\end{itemize}
\end{lemm}


\begin{proof}
Since $M<\infty$, by \cref{proof_lemm_mu} (ii), the map $\cH_\mu\ni y\mapsto\mu_M[y]\in\bR^n$ is a bounded linear operator. Thus, by the assumption, the maps $\cH_\mu\ni y\mapsto b(\mu_M[y])\in\bR^n$ and $\cH_\mu\ni y\mapsto\sigma(\mu_M[y])\in\bR^{n\times d}$ are locally Lipschitz continuous. Thus, by \cite[Theorem 2.18]{Ha23}, we see that weak existence and pathwise uniqueness hold for the SEE \eqref{proof_eq_SEE-M}. In particular, by the Yamada--Watanabe theorem (see \cite{Ku14} for a general version of the Yamada--Watanabe theorem), uniqueness in law and strong existence hold.

Let a filtered probability space $(\Omega,\cF,\bF,\bP)$ and a $d$-dimensional $\bF$-Brownian motion $W$ be fixed. For each $k\in\bN$ and $y\in\cH_\mu$, let $Y=Y^{k,y}$ be the mild solution of the SEE
\begin{equation}\label{proof_eq_SEE-M-k}
	\begin{dcases}
	\diff Y_t(\theta)=-\theta Y_t(\theta)\,\diff t+b_k(\mu_M[Y_t])\,\diff t+\sigma_k(\mu_M[Y_t])\,\diff W_t,\ \ \theta\in\supp,\ t>0,\\
	Y_0(\theta)=y(\theta),\ \ \theta\in\supp,
	\end{dcases}
\end{equation}
where $b_k:\bR^n\to\bR^n$ and $\sigma_k:\bR^n\to\bR^{n\times d}$ are defined by
\begin{equation*}
	b_k(x):=
	\begin{dcases}
	b(x)\ &\text{if $|x|\leq k$},\\
	b\Big(\frac{k}{|x|}x\Big)\ &\text{if $|x|>k$}
	\end{dcases}
	\ \text{and}\ 
	\sigma_k(x):=
	\begin{dcases}
	\sigma(x)\ &\text{if $|x|\leq k$},\\
	\sigma\Big(\frac{k}{|x|}x\Big)\ &\text{if $|x|>k$}.
	\end{dcases}
\end{equation*}
Note that the maps $\cH_\mu\ni y\mapsto b_k(\mu_M[y])\in\bR^n$ and $\cH_\mu\ni y\mapsto\sigma_k(\mu_M[y])\in\bR^{n\times d}$ are globally Lipschitz continuous. Thus, by \cite[Theorem 2.17]{Ha23}, there exists a constant $C_k>0$ such that, for any $T>0$ and any $y_1,y_2\in\cH_\mu$,
\begin{equation*}
	\bE\Big[\sup_{t\in[0,T]}\|Y^{k,y_1}_t-Y^{k,y_2}_t\|^2_{\cH_\mu}+\int^T_0\|Y^{k,y_1}_t-Y^{k,y_2}_t\|^2_{\cV_\mu}\,\diff t\Big]\leq C_ke^{C_kT}\|y_1-y_2\|^2_{\cH_\mu}.
\end{equation*}
In particular, $y\mapsto Y^{k,y}$ is continuous in probability as a map from $(\cH_\mu,\|\cdot\|_{\cH_\mu})$ to the space of $\Lambda$-valued random variables on $(\Omega,\cF,\bP)$. Furthermore, by \cite[Theorem 2.19]{Ha23}, we see that $Y^{k,y}$ is a Markov process on $\cH_\mu$ for any $y\in\cH_\mu$.

Let $y\in\cH_\mu$ be fixed, and denote by $Y^y$ the strong solution of the SEE \eqref{proof_eq_SEE-M} with the initial condition $y$. Define
\begin{equation*}
	\zeta^{k,y}:=\inf\{t\geq0\,|\,|\mu_M[Y^y_t]|\geq k\}.
\end{equation*}
Again by \cite[Theorem 2.17]{Ha23}, there exists a constant $C_k>0$ such that, for any $T>0$,
\begin{align*}
	&\bE\Big[\sup_{t\in[0,T]}\|Y^{k,y}_{t\wedge\zeta^{k,y}}-Y^y_{t\wedge\zeta^{k,y}}\|^2_{\cH_\mu}+\int^{T\wedge\zeta^{k,y}}_0\|Y^{k,y}_t-Y^y_t\|^2_{\cV_\mu}\,\diff t\Big]\\
	&\leq C_ke^{C_kT}\bE\Big[\int^{T\wedge\zeta^{k,y}}_0\Big\{|b_k(\mu_M[Y^y_t])-b(\mu_M[Y^y_t])|^2+|\sigma_k(\mu_M[Y^y_t])-\sigma(\mu_M[Y^y_t])|^2\Big\}\,\diff t\Big].
\end{align*}
By the definition of the maps $b_k$, $\sigma_k$ and the stopping time $\zeta^{k,y}$, we see that the right-hand side is zero. Thus, by using \cref{proof_lemm_mu} (i), (ii) and \cref{poof_lemm_estimate}, we have
\begin{align*}
	\bP\Big(Y^y|_{[0,T]}\neq Y^{k,y}|_{[0,T]}\ \text{in $\Lambda_T$}\Big)&\leq\bP(\zeta^{k,y}<T)=\bP\Big(\sup_{t\in[0,T]}|\mu_M[Y^y_t]|\geq k\Big)\\
	&\leq\frac{1}{k^2}\bE\Big[\sup_{t\in[0,T]}|\mu_M[Y^y_t]|^2\Big]\\
	&\leq\frac{R_M}{r(M)k^2}\bE\Big[\sup_{t\in[0,T]}\|Y^y_t\|^2_{\cH_\mu}\Big]\\
	&\leq\frac{R_M}{r(M)k^2}C_0e^{C_0T}(1+\|y\|^2_{\cH_\mu}),
\end{align*}
where $C_0=C_0(c_\LG,\mu)>0$ is the constant appearing in \cref{poof_lemm_estimate}. In particular, $Y^{k,y}\to Y^y$ in $\Lambda$ as $k\to\infty$ in probability on $(\Omega,\cF,\bP)$ uniformly in $y$ on each bounded subset of $\cH_\mu$. By the continuity of the map $y\mapsto Y^{k,y}$, we see that $y\mapsto Y^y$ is continuous in probability as a map from $(\cH_\mu,\|\cdot\|_{\cH_\mu})$ to the space of $\Lambda$-valued random variables on $(\Omega,\cF,\bP)$, and in particular weakly continuous.

The Markov property of $Y=Y^y$ follows from that of $Y^k=Y^{k,y}$ and a standard approximation argument. Denote by $\{P^k_t\}_{t\geq0}$ the Markov semigroup (defined on $\Bb$) associated with the approximating SEE \eqref{proof_eq_SEE-M-k}, and define
\begin{equation*}
	P_tf(y):=\bE^y\big[f(Y^y_t)\big],\ \ t\geq0,\ y\in\cH_\mu,\ f\in\Bb,
\end{equation*}
where $(Y^y,W^y,\Omega^y,\cF^y,\bF^y,\bP^y)$ is the weak solution (which is unique in law) of the SEE \eqref{proof_eq_SEE-M} with the initial condition $y$. By the above discussions, for any $f\in\Cb$, we have $P_tf\in\Cb$ and $\lim_{k\to\infty}P^k_tf(y)=P_tf(y)$ uniformly in $y$ on each bounded subset of $\cH_\mu$. By the Markov property of $Y^k$, for any $k\in\bN$, $0\leq s_1<\cdots<s_p=s<t$, $f\in\cC_\mathrm{b}(\cH_\mu^p)$ and $g\in \cC_\mathrm{b}(\cH_\mu)$, we have
\begin{align*}
	\bE\Big[f(Y^k_{s_1},\dots,Y^k_{s_{p}})g(Y^k_t)\Big]&=\bE\Big[f(Y^k_{s_1},\dots,Y^k_{s_{p}})P^k_{t-s}g(Y^k_s)\Big]\\
	&=\bE\Big[f(Y^k_{s_1},\dots,Y^k_{s_{p}})P_{t-s}g(Y^k_s)\Big]+\bE\Big[f(Y^k_{s_1},\dots,Y^k_{s_{p}})\big(P^k_{t-s}g(Y^k_s)-P_{t-s}g(Y^k_s)\big)\Big].
\end{align*}
On the one hand, by the facts that $\lim_{k\to\infty}Y^k=Y$ weakly on $\Lambda$ and that $P_{t-s}g\in\Cb$, we see that
\begin{equation*}
	\lim_{k\to\infty}\bE\Big[f(Y^k_{s_1},\dots,Y^k_{s_{p}})g(Y^k_t)\Big]=\bE\Big[f(Y_{s_1},\dots,Y_{s_{p}})g(Y_t)\Big]
\end{equation*}
and
\begin{equation*}
	\lim_{k\to\infty}\bE\Big[f(Y^k_{s_1},\dots,Y^k_{s_{p}})P_{t-s}g(Y^k_s)\Big]=\bE\Big[f(Y_{s_1},\dots,Y_{s_{p}})P_{t-s}g(Y_s)\Big].
\end{equation*}
On the other hand, by the tightness of $\{Y^k_s\}_{k\in\bN}$ on $\cH_\mu$, for any $\ep>0$, there exists a compact set $\cK_\ep\subset\cH_\mu$ such that $\bP(Y^k_s\notin \cK_\ep)<\ep$ for any $k\in\bN$, and hence
\begin{align*}
	&\Big|\bE\Big[f(Y^k_{s_1},\dots,Y^k_{s_{p}})\big(P^k_{t-s}g(Y^k_s)-P_{t-s}g(Y^k_s)\big)\Big]\Big|\\
	&\leq\|f\|_\infty\bE\Big[\big|P^k_{t-s}g(Y^k_s)-P_{t-s}g(Y^k_s)\big|\1_{\{Y^k_s\in\cK_\ep\}}\Big]+2\|f\|_\infty\|g\|_\infty\bP\big(Y^k_s\notin\cK_\ep\big)\\
	&\leq\|f\|_\infty\bE\Big[\big|P^k_{t-s}g(Y^k_s)-P_{t-s}g(Y^k_s)\big|\1_{\{Y^k_s\in\cK_\ep\}}\Big]+2\|f\|_\infty\|g\|_\infty\ep.
\end{align*}
Noting that $\lim_{k\to\infty}P^k_{t-s}g(y)=P_{t-s}g(y)$ uniformly in $y\in\cK_\ep$, the dominated convergence theorem yields that the first term of the last line above tends to zero as $k\to\infty$. Since $\ep>0$ is arbitrary, we see that
\begin{equation*}
	\lim_{k\to\infty}\bE\Big[f(Y^k_{s_1},\dots,Y^k_{s_{p}})\big(P^k_{t-s}g(Y^k_s)-P_{t-s}g(Y^k_s)\big)\Big]=0.
\end{equation*}
Consequently, we get
\begin{equation*}
	\bE\Big[f(Y_{s_1},\dots,Y_{s_{p}})g(Y_t)\Big]=\bE\Big[f(Y_{s_1},\dots,Y_{s_{p}})P_{t-s}g(Y_s)\Big].
\end{equation*}
Thus, the process $Y$ is a time-homogeneous Markov process on $\cH_\mu$ with the Markov semigroup $\{P_t\}_{t\geq0}$. The Feller property has been already proved; $P_tf\in\Cb$ for any $f\in\Cb$. This completes the proof.
\end{proof}

The following lemma is useful to show weak existence for the SEE \eqref{intro_eq_SEE-lift}. We borrow the idea from the work of Abi Jaber et al.\ \cite{AbiJaCuLaPu21} on weak existence for SVEs.


\begin{lemm}\label{proof_lemm_weak-existence}
Let \cref{sol_assum_kernel} hold. Suppose that $b:\bR^n\to\bR^n$ and $\sigma:\bR^n\to\bR^{n\times d}$ are continuous. For each $k\in\bN$, let $b_k:\bR^n\to\bR^n$ and $\sigma_k:\bR^n\to\bR^{n\times d}$ be continuous maps, and let $M_k\in[1,\infty)$. Assume that $\lim_{k\to\infty}M_k=\infty$, that $(b_k,\sigma_k)\to(b,\sigma)$ as $k\to\infty$ uniformly on any compact subsets of $\bR^n$, and that $|b_k(x)|+|\sigma_k(x)|\leq c_\LG(1+|x|)$ for any $x\in\bR^n$ and any $k\in\bN$ with a common constant $c_\LG>0$. Fix $y\in\cH_\mu$, and assume that, for each $k\in\bN$, there exists a weak solution $(Y^k,W^k,\Omega^k,\cF^k,\bF^k,\bP^k)$ to the SEE
\begin{equation*}
	\begin{dcases}
	\diff Y^k_t(\theta)=-\theta Y^k_t(\theta)\,\diff t+b_k(\mu_{M_k}[Y^k_t])\,\diff t+\sigma_k(\mu_{M_k}[Y^k_t])\,\diff W^k_t,\ \theta\in\supp,\ t>0,\\
	Y^k_0(\theta)=y(\theta),\ \theta\in\supp.
	\end{dcases}
\end{equation*}
Furthermore, assume that $\{Y^k\}_{k\in\bN}$ is tight on $\Lambda$. Then, weak existence holds for the SEE \eqref{intro_eq_SEE-lift} with the initial condition $y$.
\end{lemm}


\begin{proof}
For each $k\in\bN$, define $K_k(t):=\int_\supp e^{-\theta t}r_{M_k}(\theta)\dmu$, $x_k(t):=\int_\supp e^{-\theta t}y(\theta)r_{M_k}(\theta)\dmu$, $X^k_t:=\mu_{M_k}[Y^k_t]$, and
\begin{equation*}
	Z^k_t:=\int^t_0b_k(X^k_s)\,\diff s+\int^t_0\sigma_k(X^k_s)\,\diff W^k_s.
\end{equation*}
By \cref{sol_prop_equiv} (with $\mu(\diff\theta)$ replaced by $r_{M_k}(\theta)\dmu$), the pair $(X^k,Z^k)$ satisfies the SVE
\begin{equation*}
	X^k_t=x_k(t)+\int^t_0K_k(t-s)\,\diff Z^k_s,\ t>0.
\end{equation*}
Let $x(t):=(\cK y)(t)=\int_\supp e^{-\theta t}y(\theta)\dmu$. By \cref{proof_lemm_mu} (iii) and Tonelli's theorem, we have
\begin{align*}
	\int^\infty_0|x(t)-x_k(t)|^2\,\diff t&=\int^\infty_0\big|\mu[e^{-\cdot t}y]-\mu_{M_k}[e^{-\cdot t}y]\big|^2\,\diff t\\
	&\leq\frac{\ep_{M_k}}{r(M_k)}\int^\infty_0\int_\supp\theta r(\theta)e^{-2\theta t}|y(\theta)|^2\dmu\,\diff t\\
	&=\frac{\ep_{M_k}}{r(M_k)}\int_\supp\theta r(\theta)\frac{1}{2\theta}\1_{(0,\infty)}(\theta)|y(\theta)|^2\dmu\\
	&\leq\frac{\ep_{M_k}}{2r(M_k)}\|y\|^2_{\cH_\mu}.
\end{align*}
Noting \eqref{main_eq_R-ep-lim}, we see that $\lim_{k\to\infty}\int^\infty_0|x(t)-x_k(t)|^2\,\diff t=0$. By the same arguments, we can show that $\lim_{k\to\infty}\int^\infty_0|K(t)-K_k(t)|^2\,\diff t=0$.

Since $\{Y^k\}_{k\in\bN}$ is tight on the Polish space $\Lambda$, by Prokhorov's theorem, we can take a subsequence (again denoted by $\{Y^k\}_{k\in\bN}$) such that $\Law_{\bP^k}(Y^k)$ converges weakly to a probability measure on $\Lambda$. Then, by Skorokhod's representation theorem, there exist $\Lambda$-valued random variables $\tilde{Y}$ and $\tilde{Y}^k$, $k\in\bN$, defined on a common probability space $(\tilde{\Omega},\tilde{\cF},\tilde{\bP})$ such that $\lim_{k\to\infty}\tilde{Y}^k=Y$ in $\Lambda$ $\tilde{\bP}$-a.s.\ and that $\Law_{\tilde{\bP}}(\tilde{Y}^k)=\Law_{\bP^k}(Y^k)$ for any $k\in\bN$. Denote $\tilde{X}:=\mu[\tilde{Y}]$ and $\tilde{X}^k:=\mu_{M_k}[\tilde{Y}^k]$ for each $k\in\bN$. Let $T>0$ be fixed. Note that, by \cref{proof_lemm_mu} (iii), $\mu_{M_k}[\cdot]$ converges to $\mu[\cdot]$ as $k\to\infty$ in the operator norm of bounded linear operators from $\Lambda_T$ to $L^2(0,T;\bR^n)$ for any $T>0$. Thus, we see that $\lim_{k\to\infty}\int^T_0|\tilde{X}^k_t-\tilde{X}_t|^2\,\diff t=0$ $\tilde{\bP}$-a.s.\ for any $T>0$. In particular, noting that $\Law_{\tilde{\bP}}(\tilde{X}^k)=\Law_{\bP^k}(X^k)$ for each $k\in\bN$, we see that $\{X^k\}_{k\in\bN}$ is tight on $L^2_\loc(\bR_+;\bR^n)$.

Next, observe that the increasing process $\int^\cdot_0\big\{|b_k(X^k_t)|^2+|\sigma_k(X^k_t)|^2\big\}\,\diff t$ is strongly majorized by the increasing process $c^2_\LG\int^\cdot_0(1+|X^k_t|)^2\,\diff t$ in the sense that the difference of the two is increasing; see \cite[DefinitionVI.3.34]{JaSh03}. Furthermore, the tightness of $\{X^k\}_{k\in\bN}$ on $L^2_\loc(\bR_+;\bR^n)$ and the continuity of the map $x\mapsto c^2_\LG\int^\cdot_0(1+|x_t|)^2\,\diff t$ from $L^2_\loc(\bR_+;\bR^n)$ to $C(\bR_+;\bR)$ imply that $\{c^2_\LG\int^\cdot_0(1+|X^k_t|)^2\,\diff t\}_{k\in\bN}$ is tight on $C(\bR_+;\bR)$. Therefore, by \cite[Propositions VI.3.35 and VI.3.36 and Theorem VI.4.18]{JaSh03}, the continuous semimartingale $\{Z^k\}_{k\in\bN}$ is tight on $C(\bR_+;\bR^n)$, and hence $\{(X^k,Z^k)\}_{k\in\bN}$ is tight on the Polish space $L^2_\loc(\bR_+;\bR^n)\times C(\bR_+;\bR^n)$. By Prokhorov's theorem, there exists a subsequence (again denoted by $\{(X^k,Z^k)\}_{k\in\bN}$) such that $(X^k,Z^k)$ weakly converges to an $L^2_\loc(\bR_+;\bR^n)\times C(\bR_+;\bR^n)$-valued random variable $(X,Z)$ defined on a probability space $(\Omega,\cF,\bP)$.

Let $\bF$ be the augmentation of the filtration generated by $(\int^\cdot_0X_s\,\diff s,Z)$. Observe that the process $\bar{X}$ defined by
\begin{equation*}
	\bar{X}_t(\omega):=
	\begin{dcases}
	\lim_{h\downarrow0}\frac{1}{h}\int^t_{(t-h)\vee0}X_s(\omega)\,\diff s\ &\text{if the limit exists},\\
	0\ &\text{otherwise},
	\end{dcases}
	\ \ t\geq0,\ \omega\in\Omega,
\end{equation*}
is $\bF$-predictable. Furthermore, Lebesgue's differentiation theorem yields that $X_t(\omega)=\bar{X}_t(\omega)$ for a.e.\ $t>0$ for every $\omega\in\Omega$. Thus, replacing $X$ by $\bar{X}$, if necessary, we may assume that $X$ is $\bF$-predictable. Then, by \cite[Theorem 1.6]{AbiJaCuLaPu21}, it holds that
\begin{equation*}
	X_t=x(t)+\int^t_0K(t-s)\,\diff Z_s
\end{equation*}
$\bP$-a.s.\ for a.e.\ $t>0$. Besides, $Z$ is an $\bR^n$-valued continuous semimartingale with the canonical decomposition given by
\begin{equation*}
	Z_t=\int^t_0b(X_s)\,\diff s+M_t,\ t\geq0,
\end{equation*}
where $M$ is an $\bR^n$-valued continuous local martingale with $M_0=0$ such that
\begin{equation*}
	\langle M^i,M^j\rangle_t=\int^t_0a(X_s)_{i,j}\,\diff s,\ t\geq0,\ i,j\in\{1,\dots,n\},
\end{equation*}
with the notation $a(x):=\sigma(x)\sigma(x)^\top$. Extending the filtered probability space $(\Omega,\cF,\bF,\bP)$, if necessary, we can find a $d$-dimensional $\bF$-Brownian motion $W$ such that $M_t=\int^t_0\sigma(X_s)\,\diff W_s$ for any $t\geq0$ $\bP$-a.s.; see \cite[Theorem 3.4.2 and the arguments in the proof of Theorem 5.4.6]{KaSh91}. Then, we see that the tuple $(X,W,\Omega,\cF,\bF,\bP)$ is a weak solution of the SVE \eqref{intro_eq_SVE} with the forcing term $x=\cK y$. By \cref{sol_cor_equivalence} (i), weak existence holds for the SEE \eqref{intro_eq_SEE-lift} with the initial condition $y\in\cH_\mu$. This completes the proof of \cref{proof_lemm_weak-existence}.
\end{proof}

\subsection{Key estimates for the C-n-R strategy}\label{proof-CnR}

The following proposition provides key estimates for the C-n-R strategy \eqref{proof_eq_CnR}. In the terminology of this proposition, our final goal is to estimate the L\'{e}vy--Prokhorov metric between the laws of $Y$ and $\bar{Y}$, and the process $\hat{Y}$ corresponds to a controlled version of $\bar{Y}$. The first estimate \eqref{proof_eq_estimate-control} below corresponds to the ``control-step'', and the second estimate \eqref{proof_eq_estimate-reimburse} below corresponds to the ``reimbursement-step''. Recall that the total variation distance between two probability measures $\nu_1$ and $\nu_2$ on a Polish space $E$ is defined by
\begin{equation*}
	d_\TV(\nu_1,\nu_2):=\sup_{A\in\cB(E)}|\nu_1(A)-\nu_2(A)|,
\end{equation*}
whose topology is stronger than that of the weak convergence; $d_\LP(\nu_1,\nu_2)\leq d_\TV(\nu_1,\nu_2)$ for any $\nu_1,\nu_2\in\cP(E)$.


\begin{prop}\label{proof_prop_key-estimate}
Suppose that \cref{sol_assum_kernel} holds. Let $\Delta_0,\Delta_1,\Delta_2,\Delta_3\in(0,\infty)$, $m\in[1,\infty)$, $\bar{M}\in[m,\infty)$, $M\in[\bar{M},\infty]$, $\lambda\in(1,\infty)$ and $J\in[1,\infty)$ be given constants, and let $b,\bar{b}:\bR^n\to\bR^n$ and $\sigma,\bar{\sigma}:\bR^n\to\bR^{n\times d}$ be measurable maps. Assume that $\bar{b}$ and $\bar{\sigma}$ are uniformly continuous. Furthermore, assume that
\begin{equation*}
	\sup_{x\in\bR^n}|b(x)-\bar{b}(x)|^2\leq\Delta_0,\ \sup_{x\in\bR^n}|\sigma(x)-\bar{\sigma}(x)|^2\leq\Delta_0,\ \rho_{\bar{b}}\Big(\frac{\ep_{\bar{M}}}{r(\bar{M})}\Big)\leq\Delta_0,\ \rho_{\bar{\sigma}}\Big(\frac{\ep_{\bar{M}}}{r(\bar{M})}\Big)\leq\Delta_0.
\end{equation*}
Let a filtered probability space $(\Omega,\cF,\bF,\bP)$ and a $d$-dimensional $\bF$-Brownian motion $W$ be given. Fix $y\in\cH_\mu$. Suppose that $Y$ and $\bar{Y}$ are mild solutions of the following SEEs:
\begin{equation*}
	\begin{dcases}
	\diff Y_t(\theta)=-\theta Y_t(\theta)\,\diff t+b(\mu_M[Y_t])\,\diff t+\sigma(\mu_M[Y_t])\,\diff W_t,\ \ \theta\in\supp,\ t>0,\\
	Y_0(\theta)=y(\theta),\ \ \theta\in\supp,
	\end{dcases}
\end{equation*}
and
\begin{equation}\label{proof_eq_SEE-barY}
	\begin{dcases}
	\diff\bar{Y}_t(\theta)=-\theta\bar{Y}_t(\theta)\,\diff t+\bar{b}(\mu_{\bar{M}}[\bar{Y}_t])\,\diff t+\bar{\sigma}(\mu_{\bar{M}}[\bar{Y}_t])\,\diff W_t,\ \ \theta\in\supp,\ t>0,\\
	\bar{Y}_0(\theta)=y(\theta),\ \ \theta\in\supp,
	\end{dcases}
\end{equation}
respectively. Also, suppose that $\hat{Y}$ is a mild solution of the following controlled SEE
\begin{equation}\label{proof_eq_SEE-hatY}
	\begin{dcases}
	\diff\hat{Y}_t(\theta)=-\theta\hat{Y}_t(\theta)\,\diff t+\bar{b}(\mu_{\bar{M}}[\hat{Y}_t])\,\diff t+\bar{\sigma}(\mu_{\bar{M}}[\hat{Y}_t])\,\diff W_t+\lambda\mu_m[Y_t-\hat{Y}_t]\1_{[0,\tau]}(t)\,\diff t,\ \ \theta\in\supp,\ t>0,\\
	\hat{Y}_0(\theta)=y(\theta),\ \ \theta\in\supp,
	\end{dcases}
\end{equation}
with the stopping time $\tau$ given by
\begin{equation*}
	\tau=\inf\left\{t\geq0\relmiddle|
	\begin{aligned}
	\int^t_0\big|\mu_m[Y_s-\hat{Y}_s]\big|^2\,\diff s\geq\Delta_1,\ \int^t_0\int_\supp\theta r_m(\theta)|Y_s(\theta)-\hat{Y}_s(\theta)|^2\dmu\,\diff s\geq\Delta_2\\
	\text{or}\ \|Y_t-\hat{Y}_t\|^2_m\geq\Delta_3
	\end{aligned}
	\right\}.
\end{equation*}
Define
\begin{equation*}
	\zeta:=\inf\left\{t\geq0\relmiddle|\int^t_0\|Y_s\|^2_{\cV_\mu}\,\diff s\geq J\right\}\wedge J
\end{equation*}
and
\begin{equation*}
	\hat{\Omega}:=\left\{
	\begin{aligned}
	&\|Y_{t\wedge\tau\wedge\zeta}-\hat{Y}_{t\wedge\tau\wedge\zeta}\|^2_m+2\int^{t\wedge\tau\wedge\zeta}_0\int_\supp\theta r_m(\theta)|Y_s(\theta)-\hat{Y}_s(\theta)|^2\dmu\,\diff s\\
	&\hspace{5cm}+2(\lambda-1)\int^{t\wedge\tau\wedge\zeta}_0\big|\mu_m[Y_s-\hat{Y}_s]\big|^2\,\diff s\\
	&\leq9J\Big\{(1+R_m)\Delta_0+\rho_{\bar{b}}\big(\Delta_1+\ep_m\Delta_2\big)^{1/2}\Delta_1^{1/2}+R_m\rho_{\bar{\sigma}}\big(\Delta_1+\ep_m\Delta_2\big)\Big\}+\frac{1}{3}\Delta_3\\
	&\hspace{10cm}\text{for any $t\geq0$}
	\end{aligned}
	\right\}.
\end{equation*}
Then, it holds that
\begin{equation}\label{proof_eq_estimate-control}
	\bP(\Omega\setminus\hat{\Omega})\leq324JR_m\Delta^{-1}_3\Big\{\Delta_0+\rho_{\bar{\sigma}}\big(\Delta_1+\ep_m\Delta_2\big)\Big\}.
\end{equation}

Assume furthermore that $\bar{\sigma}(x)\bar{\sigma}(x)^\top\geq c_\UE I_{n\times n}$ for any $x\in\bR^n$ for some constant $c_\UE>0$ and that the SEE \eqref{proof_eq_SEE-barY} with the initial condition $y\in\cH_\mu$ satisfies strong existence and pathwise uniqueness. Then, it holds that
\begin{equation}\label{proof_eq_estimate-reimburse}
	d_\TV\big(\Law_\bP(\hat{Y}),\Law_\bP(\bar{Y})\big)\leq\frac{1}{2c^{1/2}_\UE}\lambda\Delta_1^{1/2}.
\end{equation}
\end{prop}


\begin{proof}
First, we prove the estimate \eqref{proof_eq_estimate-control} which corresponds to the ``control-step'' of the C-n-R strategy \eqref{proof_eq_CnR}. Note that, by the definition of the mild solutions, we know that
\begin{equation*}
	\sup_{t\in[0,T]}\Big\{\|Y_t\|_{\cH_\mu}+\|\hat{Y}_t\|_{\cH_\mu}\Big\}+\int^T_0\Big\{\|Y_t\|^2_{\cV_\mu}+\|\hat{Y}_t\|^2_{\cV_\mu}\Big\}<\infty\ \text{a.s.}
\end{equation*}
and
\begin{equation*}
	\int^T_0\Big\{\big|b(\mu_M[Y_t])\big|+\big|\bar{b}(\mu_{\bar{M}}[\hat{Y}_t])\big|+\big|\sigma(\mu_M[Y_t])\big|^2+\big|\bar{\sigma}(\mu_{\bar{M}}[\hat{Y}_t])\big|^2+\big|\mu_m[Y_t-\hat{Y}_t]\big|\Big\}\,\diff t<\infty\ \text{a.s.}
\end{equation*}
for any $T>0$. By using It\^{o}'s formula, for each $\theta\in\supp$, we have
\begin{align*}
	\diff |Y_t(\theta)-\hat{Y}_t(\theta)|^2&=-2\theta|Y_t(\theta)-\hat{Y}_t(\theta)|^2\,\diff t-2\lambda\Bigl\langle Y_t(\theta)-\hat{Y}_t(\theta),\mu_m[Y_t-\hat{Y}_t]\Big\rangle\1_{[0,\tau]}(t)\,\diff t\\
	&\hspace{0.5cm}+2\Big\langle Y_t(\theta)-\hat{Y}_t(\theta),b(\mu_M[Y_t])-\bar{b}(\mu_{\bar{M}}[\hat{Y}_t])\Big\rangle\,\diff t+\big|\sigma(\mu_M[Y_t])-\bar{\sigma}(\mu_{\bar{M}}[\hat{Y}_t])\big|^2\,\diff t\\
	&\hspace{0.5cm}+2\Big\langle Y_t(\theta)-\hat{Y}_t(\theta),\Big(\sigma(\mu_M[Y_t])-\bar{\sigma}(\mu_{\bar{M}}[\hat{Y}_t])\Big)\diff W_t\Big\rangle.
\end{align*}
We integrate both sides with respect to $r_m(\theta)\dmu$. Note that
\begin{align*}
	&\int_\supp\Big(\int^T_0\big|Y_t(\theta)-\hat{Y}_t(\theta)\big|^2\,\big|\sigma(\mu_M[Y_t])-\bar{\sigma}(\mu_{\bar{M}}[\hat{Y}_t])\big|^2\,\diff t\Big)^{1/2}r_m(\theta)\dmu\\
	&\leq\int_\supp\Big(\int^T_0\big|Y_t(\theta)-\hat{Y}_t(\theta)\big|^2\,\big|\sigma(\mu_M[Y_t])-\bar{\sigma}(\mu_{\bar{M}}[\hat{Y}_t])\big|^2\,\diff t\Big)^{1/2}\frac{r(\theta)}{r(m)}\dmu\\
	&\leq\frac{1}{r(m)}\Big(\int_\supp r(\theta)\dmu\Big)^{1/2}\Big(\int^T_0\|Y_t-\hat{Y}_t\|^2_{\cH_\mu}\big|\sigma(\mu_M[Y_t])-\bar{\sigma}(\mu_{\bar{M}}[\hat{Y}_t])\big|^2\,\diff t\Big)^{1/2}\\
	&\leq\frac{1}{r(m)}\Big(\int_\supp r(\theta)\dmu\Big)^{1/2}\Big(\int^T_0\big|\sigma(\mu_M[Y_t])-\bar{\sigma}(\mu_{\bar{M}}[\hat{Y}_t])\big|^2\,\diff t\Big)^{1/2}\sup_{t\in[0,T]}\|Y_t-\hat{Y}_t\|_{\cH_\mu}\\
	&<\infty\ \text{a.s.}
\end{align*}
for any $T>0$. Thus, we can use the stochastic Fubini theorem (cf.\ \cite{Ve12}) and obtain
\begin{align*}
	&\diff\int_\supp r_m(\theta)|Y_t(\theta)-\hat{Y}_t(\theta)|^2\dmu\\
	&=-2\int_\supp\theta r_m(\theta)|Y_t(\theta)-\hat{Y}_t(\theta)|^2\dmu\,\diff t\\
	&\hspace{0.5cm}-2\lambda\Big\langle\int_\supp r_m(\theta)(Y_t(\theta)-\hat{Y}_t(\theta))\dmu,\mu_m[Y_t-\hat{Y}_t]\Big\rangle\1_{[0,\tau]}(t)\,\diff t\\
	&\hspace{0.5cm}+2\Big\langle\int_\supp r_m(\theta)(Y_t(\theta)-\hat{Y}_t(\theta))\dmu,b(\mu_M[Y_t])-\bar{b}(\mu_{\bar{M}}[\hat{Y}_t])\Big\rangle\,\diff t\\
	&\hspace{0.5cm}+\int_\supp r_m(\theta)\dmu\big|\sigma(\mu_M[Y_t])-\bar{\sigma}(\mu_{\bar{M}}[\hat{Y}_t])\big|^2\,\diff t\\
	&\hspace{0.5cm}+2\Big\langle\int_\supp r_m(\theta)(Y_t(\theta)-\hat{Y}_t(\theta))\dmu,\Big(\sigma(\mu_M[Y_t])-\bar{\sigma}(\mu_{\bar{M}}[\hat{Y}_t])\Big)\diff W_t\Big\rangle.
\end{align*}
Noting the notations $R_m:=\int_\supp r_m(\theta)\dmu$, $\|y\|_m:=\big(\int_\supp r_m(\theta)|y(\theta)|^2\dmu\big)^{1/2}$ and $\mu_m[y]:=\int_\supp r_m(\theta)y(\theta)\dmu$ for $y\in\cH_\mu$, we obtain
\begin{equation}\label{proof_eq_b-sigma-W}
\begin{split}
	&\|Y_{t\wedge\tau\wedge\zeta}-\hat{Y}_{t\wedge\tau\wedge\zeta}\|^2_m+2\int^{t\wedge\tau\wedge\zeta}_0\int_\supp\theta r_m(\theta)|Y_s(\theta)-\hat{Y}_s(\theta)|^2\dmu\,\diff s+2\lambda\int^{t\wedge\tau\wedge\zeta}_0\big|\mu_m[Y_s-\hat{Y}_s]\big|^2\,\diff s\\
	&=2I^b_t+R_mI^\sigma_t+2I^W_t
\end{split}
\end{equation}
for any $t\geq0$ a.s., where
\begin{align*}
	&I^b_t:=\int^{t\wedge\tau\wedge\zeta}_0\Big\langle\mu_m[Y_s-\hat{Y}_s],b(\mu_M[Y_s])-\bar{b}(\mu_{\bar{M}}[\hat{Y}_s])\Big\rangle\,\diff s,\\
	&I^\sigma_t:=\int^{t\wedge\tau\wedge\zeta}_0\big|\sigma(\mu_M[Y_s])-\bar{\sigma}(\mu_{\bar{M}}[\hat{Y}_s])\big|^2\,\diff s,\\
	&I^W_t:=\int^{t\wedge\tau\wedge\zeta}_0\Big\langle\mu_m[Y_s-\hat{Y}_s],\Big(\sigma(\mu_M[Y_s])-\bar{\sigma}(\mu_{\bar{M}}[\hat{Y}_s])\Big)\diff W_s\Big\rangle.
\end{align*}
Fix $t\geq0$. First, we estimate $I^b_t$. Observe that
\begin{align*}
	I^b_t&=\int^{t\wedge\tau\wedge\zeta}_0\Big\langle\mu_m[Y_s-\hat{Y}_s],b(\mu_M[Y_s])-\bar{b}(\mu_M[Y_s])\Big\rangle\,\diff s+\int^{t\wedge\tau\wedge\zeta}_0\Big\langle\mu_m[Y_s-\hat{Y}_s],\bar{b}(\mu_M[Y_s])-\bar{b}(\mu_{\bar{M}}[Y_s])\Big\rangle\,\diff s\\
	&\hspace{0.5cm}+\int^{t\wedge\tau\wedge\zeta}_0\Big\langle\mu_m[Y_s-\hat{Y}_s],\bar{b}(\mu_{\bar{M}}[Y_s])-\bar{b}(\mu_{\bar{M}}[\hat{Y}_s])\Big\rangle\,\diff s\\
	&\leq\int^{t\wedge\tau\wedge\zeta}_0\big|\mu_m[Y_s-\hat{Y}_s]\big|^2\,\diff s+\frac{1}{2}\int^J_0\big|b(\mu_M[Y_s])-\bar{b}(\mu_M[Y_s])\big|^2\,\diff s+\frac{1}{2}\int^\zeta_0\big|\bar{b}(\mu_M[Y_s])-\bar{b}(\mu_{\bar{M}}[Y_s])\big|^2\,\diff s\\
	&\hspace{0.5cm}+\int^{\tau\wedge J}_0\big|\bar{b}(\mu_{\bar{M}}[Y_s])-\bar{b}(\mu_{\bar{M}}[\hat{Y}_s])\big|\cdot\big|\mu_m[Y_s-\hat{Y}_s]\big|\,\diff s.
\end{align*}
Since $\sup_{x\in\bR^n}|b(x)-\bar{b}(x)|^2\leq\Delta_0$, we have
\begin{equation*}
	\int^J_0\big|b(\mu_M[Y_s])-\bar{b}(\mu_M[Y_s])\big|^2\,\diff s\leq J\Delta_0.
\end{equation*}
Next, noting that $|\bar{b}(x)-\bar{b}(\bar{x})|^2\leq\rho_{\bar{b}}\big(|x-\bar{x}|^2\big)$ for any $x,\bar{x}\in\bR^n$ and that $\rho_{\bar{b}}:[0,\infty)\to[0,\infty)$ is concave, we have
\begin{equation*}
	\int^\zeta_0\big|\bar{b}(\mu_M[Y_s])-\bar{b}(\mu_{\bar{M}}[Y_s])\big|^2\,\diff s\leq\int^\zeta_0\rho_{\bar{b}}\big(\big|\mu_M[Y_s]-\mu_{\bar{M}}[Y_s]\big|^2\big)\,\diff s\leq\zeta\rho_{\bar{b}}\Big(\frac{1}{\zeta}\int^\zeta_0\big|\mu_M[Y_s]-\mu_{\bar{M}}[Y_s]\big|^2\,\diff s\Big),
\end{equation*}
where we used Jensen's inequality in the second inequality. Since $M\geq\bar{M}$ and $\bar{M}<\infty$, by \cref{proof_lemm_mu} (iii) and the definition of the stopping time $\zeta$, we have
\begin{equation*}
	\int^\zeta_0\big|\mu_M[Y_s]-\mu_{\bar{M}}[Y_s]\big|^2\,\diff s\leq\frac{\ep_{\bar{M}}}{r(\bar{M})}\int^\zeta_0\|Y_s\|^2_{\cV_\mu}\,\diff s\leq\frac{\ep_{\bar{M}}J}{r(\bar{M})}.
\end{equation*}
Then, noting that $\zeta\leq J$ and $\rho_{\bar{b}}\big(\frac{\ep_{\bar{M}}}{r(\bar{M})}\big)\leq\Delta_0$, by the inequality $\rho_{\bar{b}}(st)\leq(1+s)\rho_{\bar{b}}(t)$ for $s,t\geq0$, we get
\begin{equation*}
	\int^\zeta_0\big|\bar{b}(\mu_M[Y_s])-\bar{b}(\mu_{\bar{M}}[Y_s])\big|^2\,\diff s\leq\zeta\rho_{\bar{b}}\Big(\frac{\ep_{\bar{M}}J}{r(\bar{M})\zeta}\Big)\leq\zeta\Big(1+\frac{J}{\zeta}\Big)\rho_{\bar{b}}\Big(\frac{\ep_{\bar{M}}}{r(\bar{M})}\Big)\leq2J\Delta_0.
\end{equation*}
Lastly, observe that
\begin{align*}
	&\int^{\tau\wedge J}_0\big|\bar{b}(\mu_{\bar{M}}[Y_s])-\bar{b}(\mu_{\bar{M}}[\hat{Y}_s])\big|\cdot\big|\mu_m[Y_s-\hat{Y}_s]\big|\,\diff s\\
	&\leq\Big(\int^{\tau\wedge J}_0\big|\bar{b}(\mu_{\bar{M}}[Y_s])-\bar{b}(\mu_{\bar{M}}[\hat{Y}_s])\big|^2\,\diff s\Big)^{1/2}\Big(\int^\tau_0\big|\mu_m[Y_s-\hat{Y}_s]\big|^2\,\diff s\Big)^{1/2}\\
	&\leq\Big(\int^{\tau\wedge J}_0\big|\bar{b}(\mu_{\bar{M}}[Y_s])-\bar{b}(\mu_{\bar{M}}[\hat{Y}_s])\big|^2\,\diff s\Big)^{1/2}\Delta_1^{1/2},
\end{align*}
where the last inequality follows from the definition of the stopping time $\tau$. By using Jensen's inequality and the inequality $\rho_{\bar{b}}(st)\leq(1+s)\rho_{\bar{b}}(t)$ for $s,t\geq0$, we have
\begin{align*}
	\int^{\tau\wedge J}_0\big|\bar{b}(\mu_{\bar{M}}[Y_s])-\bar{b}(\mu_{\bar{M}}[\hat{Y}_s])\big|^2\,\diff s&\leq\int^{\tau\wedge J}_0\rho_{\bar{b}}\big(\big|\mu_{\bar{M}}[Y_s-\hat{Y}_s]\big|^2\big)\,\diff s\\
	&\leq(\tau\wedge J)\rho_{\bar{b}}\Big(\frac{1}{\tau\wedge J}\int^{\tau\wedge J}_0\big|\mu_{\bar{M}}[Y_s-\hat{Y}_s]\big|^2\,\diff s\Big)\\
	&\leq(\tau\wedge J)\Big(1+\frac{2}{\tau\wedge J}\Big)\rho_{\bar{b}}\Big(\frac{1}{2}\int^{\tau\wedge J}_0\big|\mu_{\bar{M}}[Y_s-\hat{Y}_s]\big|^2\,\diff s\Big)\\
	&\leq3J\rho_{\bar{b}}\Big(\frac{1}{2}\int^\tau_0\big|\mu_{\bar{M}}[Y_s-\hat{Y}_s]\big|^2\,\diff s\Big),
\end{align*}
where in the last inequality we used $J\geq1$. Since $m\leq\bar{M}<\infty$, by \cref{proof_lemm_mu} (iii) and the definition of the stopping time $\tau$, we have
\begin{align*}
	\int^\tau_0\big|\mu_{\bar{M}}[Y_s-\hat{Y}_s]\big|^2\,\diff s&\leq2\int^\tau_0\big|\mu_m[Y_s-\hat{Y}_s]\big|^2\,\diff s+2\ep_m\int^\tau_0\int_\supp\theta r_m(\theta)\big|Y_s(\theta)-\hat{Y}_s(\theta)\big|^2\dmu\,\diff s\\
	&\leq2\Delta_1+2\ep_m\Delta_2.
\end{align*}
Hence,
\begin{equation*}
	\int^{\tau\wedge J}_0\big|\bar{b}(\mu_{\bar{M}}[Y_s])-\bar{b}(\mu_{\bar{M}}[\hat{Y}_s])\big|^2\,\diff s\leq3J\rho_{\bar{b}}\big(\Delta_1+\ep_m\Delta_2\big).
\end{equation*}
Summarizing the above estimates, we get
\begin{equation}\label{proof_eq_b}
	I^b_t\leq\int^{t\wedge\tau\wedge\zeta}_0\big|\mu_m[Y_s-\hat{Y}_s]\big|^2\,\diff s+\frac{3}{2}J\Delta_0+(3J)^{1/2}\rho_{\bar{b}}\big(\Delta_1+\ep_m\Delta_2\big)^{1/2}\Delta_1^{1/2}
\end{equation}
for any $t\geq0$ a.s.

Next, we estimate $I^\sigma_t$. Observe that
\begin{align*}
	I^\sigma_t&\leq3\int^J_0\big|\sigma(\mu_M[Y_s])-\bar{\sigma}(\mu_M[Y_s])\big|^2\,\diff s+3\int^\zeta_0\big|\bar{\sigma}(\mu_M[Y_s])-\bar{\sigma}(\mu_{\bar{M}}[Y_s])\big|^2\,\diff s\\
	&\hspace{1cm}+3\int^{\tau\wedge J}_0\big|\bar{\sigma}(\mu_{\bar{M}}[Y_s])-\bar{\sigma}(\mu_{\bar{M}}[\hat{Y}_s])\big|^2\,\diff s.
\end{align*}
By the same arguments as above, we have
\begin{equation*}
	\int^J_0\big|\sigma(\mu_M[Y_s])-\bar{\sigma}(\mu_M[Y_s])\big|^2\,\diff s\leq J\Delta_0,\ \int^\zeta_0\big|\bar{\sigma}(\mu_M[Y_s])-\bar{\sigma}(\mu_{\bar{M}}[Y_s])\big|^2\,\diff s\leq2J\Delta_0,
\end{equation*}
and
\begin{equation*}
	\int^{\tau\wedge J}_0\big|\bar{\sigma}(\mu_{\bar{M}}[Y_s])-\bar{\sigma}(\mu_{\bar{M}}[\hat{Y}_s])\big|^2\,\diff s\leq3J\rho_{\bar{\sigma}}\big(\Delta_1+\ep_m\Delta_2\big).
\end{equation*}
Thus,
\begin{equation}\label{proof_eq_sigma}
	I^\sigma_t\leq9J\Big\{\Delta_0+\rho_{\bar{\sigma}}\big(\Delta_1+\ep_m\Delta_2\big)\Big\}
\end{equation}
for any $t\geq0$ a.s.

By \eqref{proof_eq_b-sigma-W}, \eqref{proof_eq_b} and \eqref{proof_eq_sigma}, together with the assumption $J\geq1$, we obtain
\begin{align*}
	&\|Y_{t\wedge\tau\wedge\zeta}-\hat{Y}_{t\wedge\tau\wedge\zeta}\|^2_m+2\int^{t\wedge\tau\wedge\zeta}_0\int_\supp\theta r_m(\theta)|Y_s(\theta)-\hat{Y}_s(\theta)|^2\dmu\,\diff s+2(\lambda-1)\int^{t\wedge\tau\wedge\zeta}_0\big|\mu_m[Y_s-\hat{Y}_s]\big|^2\,\diff s\\
	&\leq2\Big\{\frac{3}{2}J\Delta_0+(3J)^{1/2}\rho_{\bar{b}}\big(\Delta_1+\ep_m\Delta_2\big)^{1/2}\Delta_1^{1/2}\Big\}+9JR_m\Big\{\Delta_0+\rho_{\bar{\sigma}}\big(\Delta_1+\ep_m\Delta_2\big)\Big\}+2I^W_t\\
	&\leq9J\Big\{(1+R_m)\Delta_0+\rho_{\bar{b}}\big(\Delta_1+\ep_m\Delta_2\big)^{1/2}\Delta_1^{1/2}+R_m\rho_{\bar{\sigma}}\big(\Delta_1+\ep_m\Delta_2\big)\Big\}+2I^W_t
\end{align*}
for any $t\geq0$ a.s. This implies that
\begin{equation*}
	\bP(\Omega\setminus\hat{\Omega})\leq\bP\Big(\sup_{t\geq0}|I^W_t|>\frac{1}{6}\Delta_3\Big)=\bP\Big(\sup_{t\in[0,J]}|I^W_t|>\frac{1}{6}\Delta_3\Big).
\end{equation*}
By using Doob's martingale inequality and the estimate \eqref{proof_eq_sigma}, we have
\begin{align*}
	\bP\Big(\sup_{t\in[0,J]}|I^W_t|>\frac{1}{6}\Delta_3\Big)&\leq36\Delta_3^{-2}\bE\big[|I^W_J|^2\big]\\
	&\leq36\Delta_3^{-2}\bE\Big[\int^{\tau\wedge\zeta}_0\big|\mu_m[Y_s-\hat{Y}_s]\big|^2\cdot\big|\sigma(\mu_M[Y_s])-\bar{\sigma}(\mu_{\bar{M}}[\hat{Y}_s])\big|^2\,\diff s\Big]\\
	&\leq36\Delta_3^{-2}\bE\Big[I^\sigma_J\sup_{s\geq0}\big|\mu_m[Y_{s\wedge\tau}-\hat{Y}_{s\wedge\tau}]\big|^2\Big]\\
	&\leq324J\Delta_3^{-2}\Big\{\Delta_0+\rho_{\bar{\sigma}}\big(\Delta_1+\ep_m\Delta_2\big)\Big\}\bE\Big[\sup_{s\geq0}\big|\mu_m[Y_{s\wedge\tau}-\hat{Y}_{s\wedge\tau}]\big|^2\Big].
\end{align*}
Furthermore, by \cref{proof_lemm_mu} (ii) and the definition of the stopping time $\tau$, we have
\begin{equation*}
	\bE\Big[\sup_{s\geq0}\big|\mu_m[Y_{s\wedge\tau}-\hat{Y}_{s\wedge\tau}]\big|^2\Big]\leq R_m\bE\Big[\sup_{s\geq0}\|Y_{s\wedge\tau}-\hat{Y}_{s\wedge\tau}\|^2_m\Big]\leq R_m\Delta_3.
\end{equation*}
Then, we obtain
\begin{equation*}
	\bP(\Omega\setminus\hat{\Omega})\leq324JR_m\Delta_3^{-1}\Big\{\Delta_0+\rho_{\bar{\sigma}}\big(\Delta_1+\ep_m\Delta_2\big)\Big\},
\end{equation*}
and thus the first estimate \eqref{proof_eq_estimate-control} holds.

Next, we prove the second estimate \eqref{proof_eq_estimate-reimburse} which corresponds to the ``reimbursement-step'' of the C-n-R strategy \eqref{proof_eq_CnR}. Here, we assume that $\bar{\sigma}(x)\bar{\sigma}(x)^\top\geq c_\UE I_{n\times n}$ for any $x\in\bR^n$ for some constant $c_\UE>0$ and that the SEE \eqref{proof_eq_SEE-barY} with the initial condition $y\in\cH_\mu$ satisfies strong existence and pathwise uniqueness. Observe that the mild solution $\hat{Y}$ of the controlled SEE \eqref{proof_eq_SEE-hatY} solves the SEE
\begin{equation*}
	\begin{dcases}
	\diff\hat{Y}_t(\theta)=-\theta\hat{Y}_t(\theta)\,\diff t+\bar{b}(\mu_{\bar{M}}[\hat{Y}_t])\,\diff t+\bar{\sigma}(\mu_{\bar{M}}[\hat{Y}_t])\,\diff \hat{W}_t,\ \ \theta\in\supp,\ t>0,\\
	\hat{Y}_0(\theta)=y(\theta),\ \ \theta\in\supp,
	\end{dcases}
\end{equation*}
where
\begin{equation*}
	\hat{W}_t:=W_t+\int^t_0v_s\,\diff s,\ \ v_t:=\lambda\bar{\sigma}(\mu_{\bar{M}}[\hat{Y}_t])^\dagger\mu_m[Y_t-\hat{Y}_t]\1_{[0,\tau]}(t),\ t\geq0,
\end{equation*}
with the notation that $\bar{\sigma}(x)^\dagger=\bar{\sigma}(x)^\top(\bar{\sigma}(x)\bar{\sigma}(x)^\top)^{-1}$ for $x\in\bR^n$. Observe that, for each $x\in\bR^n$, the matrix $\bar{\sigma}(x)^\dagger\in\bR^{d\times n}$ is the Moore--Penrose inverse of $\bar{\sigma}(x)\in\bR^{n\times d}$ and satisfies $\bar{\sigma}(x)\bar{\sigma}(x)^\dagger=I_{n\times n}$. By the uniform ellipticity condition for $\bar{\sigma}$, we see that the operator norm of the matrix $\bar{\sigma}(x)^\dagger$ satisfies $\|\bar{\sigma}(x)^\dagger\|_\op\leq c^{-1/2}_\UE$ for any $x\in\bR^n$. Then, by the definition of the stopping time $\tau$, we have
\begin{equation*}
	\int^\infty_0|v_t|^2\,\diff t\leq\frac{1}{c_\UE}\lambda^2\int^\tau_0\big|\mu_m[Y_t-\hat{Y}_t]\big|^2\,\diff t\leq\frac{1}{c_\UE}\lambda^2\Delta_1\ \ \text{a.s.}
\end{equation*}
Thus, by means of Novikov's condition and the Girsanov theorem, we see that the process $\hat{W}$ is a $d$-dimensional Brownian motion under a new probability measure $\hat{\bP}$ on $(\Omega,\cF)$ which is equivalent to $\bP$. This means that $(\hat{Y},\hat{W},\Omega,\cF,\bF,\hat{\bP})$ is a weak solution of the SEE \eqref{proof_eq_SEE-barY} with the initial condition $y$. By strong existence and pathwise uniqueness for the SEE \eqref{proof_eq_SEE-barY}, there exists a measurable map $\Psi:C(\bR_+;\bR^d)\to\Lambda$ such that $\hat{Y}=\Psi(\hat{W})$ and $\bar{Y}=\Psi(W)$ in $\Lambda$ $\bP$-a.s. This implies that
\begin{equation*}
	d_\TV\big(\Law_\bP(\hat{Y}),\Law_\bP(\bar{Y})\big)=d_\TV\big(\Law_\bP(\Psi(\hat{W})),\Law_\bP(\Psi(W))\big)\leq d_\TV\big(\Law_\bP(\hat{W}),\Law_\bP(W)\big).
\end{equation*}
Furthermore, by \cite[Lemma A.1 and Theorem A.2]{BuKuSc20}, together with the above estimate, we have
\begin{equation*}
	d_\TV\big(\Law_\bP(\hat{W}),\Law_\bP(W)\big)\leq\frac{1}{2}\bE\Big[\int^\infty_0|v_t|^2\,\diff t\Big]^{1/2}\leq\frac{1}{2c^{1/2}_\UE}\lambda\Delta^{1/2}_1.
\end{equation*}
Thus, the second estimate \eqref{proof_eq_estimate-reimburse} holds. This completes the proof of \cref{proof_prop_key-estimate}.
\end{proof}

\subsection{Proof of \cref{main_theo_main}}\label{proof-main}

We can now complete the proof of \cref{main_theo_main}. We demonstrate the proof of the theorem for the singular case (i). The regular case (ii) can be proved by a similar (and easier) way; see \cref{proof-remark-regular} below. Suppose that $b:\bR^n\to\bR^n$ and $\sigma:\bR^n\to\bR^{n\times d}$ are uniformly continuous and satisfy $\sigma(x)\sigma(x)^\top\geq c_\UE I_{n\times n}$ for any $x\in\bR^n$ for some constant $c_\UE>0$. Without loss of generality, we may assume that $\rho_b(t)+\rho_\sigma(t)>0$ for any $t>0$; otherwise both $b$ and $\sigma$ become constants, and thus there are nothing to prove. Furthermore, assume that the kernel $K$ is singular and satisfies
\begin{equation*}
	\liminf_{m\to\infty}R_m\frac{\rho_\sigma(\ep_m^2)}{\ep_m}=0.
\end{equation*}
Here, by the singularity of the kernel $K$, we have $\ep_m>0$ for any $m\in[1,\infty)$. Recall that $\ep_m\downarrow0$ and $R_m\uparrow\infty$ monotonically as $m\to\infty$. Then, there exists an increasing sequence $\{m_k\}_{k\in\bN}\subset[1,\infty)$ such that $\big\{R_{m_k}\frac{\rho_\sigma(\ep_{m_k}^2)}{\ep_{m_k}}\big\}_{k\in\bN}$ is decreasing, that
\begin{equation*}
	\ep_{m_k}\leq1\ \text{and}\ \rho_b(\ep_{m_k}^2)^{1/2}+R_{m_k}\frac{\rho_\sigma(\ep_{m_k}^2)}{\ep_{m_k}}\leq108^{-4}
\end{equation*}
for any $k\in\bN$, and that
\begin{equation*}
	\lim_{k\to\infty}\Big(\ep_{m_k}+\rho_b(\ep_{m_k}^2)^{1/2}+R_{m_k}\frac{\rho_\sigma(\ep_{m_k}^2)}{\ep_{m_k}}\Big)=0.
\end{equation*}
For each $k\in\bN$, we define $\Delta_{0,k},\Delta_{1,k},\Delta_{2,k},\Delta_{3,k}\in(0,\infty)$, $\lambda_k\in(1,\infty)$ and $J_k\in[1,\infty)$ by
\begin{equation*}
	\Delta_{0,k}=\frac{\rho_b(\ep_{m_k}^2)^{1/2}\ep_{m_k}+R_{m_k}\rho_\sigma(\ep_{m_k}^2)}{1+R_{m_k}},
\end{equation*}
\begin{equation*}
	\Delta_{1,k}=\frac{\ep_{m_k}^2}{2},\ \Delta_{2,k}=\frac{\ep_{m_k}}{2},\ \Delta_{3,k}=\Big(\rho_b(\ep_{m_k}^2)^{1/2}+R_{m_k}\frac{\rho_\sigma(\ep_{m_k}^2)}{\ep_{m_k}}\Big)^{1/2}\ep_{m_k},
\end{equation*}
and
\begin{equation*}
	\lambda_k=1+\Big(\rho_b(\ep_{m_k}^2)^{1/2}+R_{m_k}\frac{\rho_\sigma(\ep_{m_k}^2)}{\ep_{m_k}}\Big)^{1/2}\ep_{m_k}^{-1},\ J_k=\Big(\rho_b(\ep_{m_k}^2)^{1/2}+R_{m_k}\frac{\rho_\sigma(\ep_{m_k}^2)}{\ep_{m_k}}\Big)^{-1/4}.
\end{equation*}
Note that $\{\Delta_{i,k}\}_{k\in\bN}$, $i=0,1,2,3$, are decreasing sequences of positive numbers. Furthermore, noting \eqref{main_eq_R-ep-lim}, we can take an increasing sequence $\{M_k\}_{k\in\bN}$ with $\lim_{k\to\infty}M_k=\infty$ such that $M_k\in[m_k,\infty)$ and
\begin{equation*}
	\rho_b\Big(\frac{\ep_{M_k}}{r(M_k)}\Big)\leq\Delta_{0,k},\ \rho_\sigma\Big(\frac{\ep_{M_k}}{r(M_k)}\Big)\leq\Delta_{0,k},
\end{equation*}
for any $k\in\bN$. Then, for each $k\in\bN$, by taking convolutions of $b$ and $\sigma$ with smooth mollifiers one can construct maps $b_k:\bR^n\to\bR^n$ and $\sigma_k:\bR^n\to\bR^{n\times d}$ such that
\begin{itemize}
\item
$\displaystyle\sup_{x\in\bR^n}|b(x)-b_k(x)|^2\leq\frac{\Delta_{0,k}}{2}$ and $\displaystyle\sup_{x\in\bR^n}|\sigma(x)-\sigma_k(x)|^2\leq\frac{\Delta_{0,k}}{2}$;
\item
$b_k$ and $\sigma_k$ are uniformly continuous and satisfy $\rho_{b_k}(t)\leq\rho_b(t)$ and $\rho_{\sigma_k}(t)\leq\rho_\sigma(t)$ for any $t\geq0$;
\item
$b_k$ and $\sigma_k$ satisfy the local Lipschitz conditions;
\item
$\displaystyle\sigma_k(x)\sigma_k(x)^\top\geq\frac{c_\UE}{2}I_{n\times n}$ for any $x\in\bR^n$.
\end{itemize}
In particular, it holds that
\begin{equation*}
	\rho_{b_k}\Big(\frac{\ep_{M_k}}{r(M_k)}\Big)\leq\Delta_{0,k}\ \text{and}\ \rho_{\sigma_k}\Big(\frac{\ep_{M_k}}{r(M_k)}\Big)\leq\Delta_{0,k}
\end{equation*}
for any $k\in\bN$.

Let us estimate the three important terms appearing in \cref{proof_prop_key-estimate} with $(\bar{b},\bar{\sigma},m,\Delta_0,\Delta_1,\Delta_2,\Delta_3,\lambda,J)$ replaced by $(b_k,\sigma_k,m_k,\Delta_{0,k},\Delta_{1,k},\Delta_{2,k},\Delta_{3,k},\lambda_k,J_k)$. First, observe that
\begin{align*}
	&9J_k\Big\{(1+R_{m_k})\Delta_{0,k}+\rho_{b_k}\big(\Delta_{1,k}+\ep_{m_k}\Delta_{2,k}\big)^{1/2}\Delta_{1,k}^{1/2}+R_{m_k}\rho_{\sigma_k}\big(\Delta_{1,k}+\ep_{m_k}\Delta_{2,k}\big)\Big\}+\frac{1}{3}\Delta_{3,k}\\
	&\leq18J_k\Big(\rho_b(\ep_{m_k}^2)^{1/2}\ep_{m_k}+R_{m_k}\rho_\sigma(\ep_{m_k}^2)\Big)+\frac{1}{3}\Big(\rho_b(\ep_{m_k}^2)^{1/2}+R_{m_k}\frac{\rho_\sigma(\ep_{m_k}^2)}{\ep_{m_k}}\Big)^{1/2}\ep_{m_k}\\
	&=\Big\{18\Big(\rho_b(\ep_{m_k}^2)^{1/2}+R_{m_k}\frac{\rho_\sigma(\ep_{m_k}^2)}{\ep_{m_k}}\Big)^{1/4}+\frac{1}{3}\Big\}\Big(\rho_b(\ep_{m_k}^2)^{1/2}+R_{m_k}\frac{\rho_\sigma(\ep_{m_k}^2)}{\ep_{m_k}}\Big)^{1/2}\ep_{m_k}.
\end{align*}
Since $\rho_b(\ep_{m_k}^2)^{1/2}+R_{m_k}\frac{\rho_\sigma(\ep_{m_k}^2)}{\ep_{m_k}}\leq108^{-4}$, it holds that
\begin{equation}\label{proof_eq_estimate-A}
\begin{split}
	&9J_k\Big\{(1+R_{m_k})\Delta_{0,k}+\rho_{b_k}\big(\Delta_{1,k}+\ep_{m_k}\Delta_{2,k}\big)^{1/2}\Delta_{1,k}^{1/2}+R_{m_k}\rho_{\sigma_k}\big(\Delta_{1,k}+\ep_{m_k}\Delta_{2,k}\big)\Big\}+\frac{1}{3}\Delta_{3,k}\\
	&\leq\frac{1}{2}\Big(\rho_b(\ep_{m_k}^2)^{1/2}+R_{m_k}\frac{\rho_\sigma(\ep_{m_k}^2)}{\ep_{m_k}}\Big)^{1/2}\ep_{m_k}.
\end{split}
\end{equation}
Second, we have
\begin{equation}\label{proof_eq_estimate-B}
\begin{split}
	J_kR_{m_k}\Delta^{-1}_{3,k}\Big\{\Delta_{0,k}+\rho_{\sigma_k}\big(\Delta_{1,k}+\ep_{m_k}\Delta_{2,k}\big)\Big\}&\leq2J_k\Delta_{3,k}^{-1}\Big\{\rho_b(\ep_{m_k}^2)^{1/2}\ep_{m_k}+R_{m_k}\rho_\sigma(\ep_{m_k}^2)\Big\}\\
	&=2\Big(\rho_b(\ep_{m_k}^2)^{1/2}+R_{m_k}\frac{\rho_\sigma(\ep_{m_k}^2)}{\ep_{m_k}}\Big)^{1/4}.
\end{split}
\end{equation}
Third, we have
\begin{equation}\label{proof_eq_estimate-C}
	\lambda_k\Delta_{1,k}^{1/2}=\frac{\ep_{m_k}+\Big(\rho_b(\ep_{m_k}^2)^{1/2}+R_{m_k}\frac{\rho_\sigma(\ep_{m_k}^2)}{\ep_{m_k}}\Big)^{1/2}}{\sqrt{2}}.
\end{equation}


\begin{rem}\label{proof_rem_LG}
The maps $b,b_k:\bR^n\to\bR^n$ and $\sigma,\sigma_k:\bR^n\to\bR^{n\times d}$ satisfy the linear growth condition \eqref{pre_eq_LG} with the same constant $c_\LG>0$. Indeed, for $\varphi=b,b_k,\sigma,\sigma_k$, we have $|\varphi(x)-\varphi(0)|\leq\sqrt{\rho_\varphi(1)}$ for any $x\in\bR^n$ with $|x|\leq1$; also, for any $x\in\bR^n$ with $|x|\in(p,p+1]$ and $p\in\bN$,
\begin{align*}
	|\varphi(x)-\varphi(0)|\leq\sum^{p+1}_{q=1}\Big|\varphi\Big(\frac{q}{p+1}x\Big)-\varphi\Big(\frac{q-1}{p+1}x\Big)\Big|\leq(p+1)\sqrt{\rho_\varphi(1)}\leq(|x|+1)\sqrt{\rho_\varphi(1)}.
\end{align*}
Hence,
\begin{equation*}
	|\varphi(x)|\leq|\varphi(0)|+(1+|x|)\sqrt{\rho_\varphi(1)}\leq\Big\{|\varphi(0)|+\sqrt{\rho_\varphi(1)}\Big\}(1+|x|)
\end{equation*}
for any $x\in\bR^n$. Noting that $\rho_{b_k}\leq\rho_b$, $\rho_{\sigma_k}\leq\rho_\sigma$, and
\begin{equation*}
	|b_k(0)|\leq|b(0)|+\sqrt{\frac{\Delta_{0,k}}{2}}\leq|b(0)|+1,\ |\sigma_k(0)|\leq|\sigma(0)|+\sqrt{\frac{\Delta_{0,k}}{2}}\leq|\sigma(0)|+1,
\end{equation*}
we see that the linear growth condition \eqref{pre_eq_LG} holds for $\varphi=b,b_k,\sigma,\sigma_k$ with the same constant
\begin{equation*}
	c_\LG=\Big\{|b(0)|+1+\sqrt{\rho_b(1)}\Big\}\vee\Big\{|\sigma(0)|+1+\sqrt{\rho_\sigma(1)}\Big\}<\infty.
\end{equation*}
\end{rem}

For each $k\in\bN$, consider the following SEE:
\begin{equation}\label{proof_eq_SEE-approximation}
	\begin{dcases}
	\diff Y^k_t(\theta)=-\theta Y^k_t(\theta)\,\diff t+b_k(\mu_{M_k}[Y^k_t])\,\diff t+\sigma_k(\mu_{M_k}[Y^k_t])\,\diff W_t,\ \ \theta\in\supp,\ t>0,\\
	Y^k_0(\theta)=y(\theta),\ \ \theta\in\supp.
	\end{dcases}
\end{equation}
Recall that the coefficients $b_k:\bR^n\to\bR^n$ and $\sigma_k:\bR^n\to\bR^{n\times d}$ satisfy the linear growth condition and the local Lipschitz condition. Thus, we can apply \cref{pre_lemm_SEE-M} to the SEE \eqref{proof_eq_SEE-approximation}. In particular, for the SEE \eqref{proof_eq_SEE-approximation}, strong existence and pathwise uniqueness hold, the solution is weakly continuous with respect to the initial condition $y\in\cH_\mu$, and the mild solution is a time-homogeneous Markov process. By means of \cref{proof_prop_key-estimate}, we will show the following two claims:


\begin{claim}\label{proof_claim_1}
For any $y\in\cH_\mu$, the sequence of the laws of the mild solutions $Y^{k,y}$ of the SEEs \eqref{proof_eq_SEE-approximation} with the initial condition $y$ (which does not depend on the choice of the probability space or the Brownian motion by uniqueness in law) is a Cauchy sequence in $(\cP(\Lambda),d_\LP)$.
\end{claim}


\begin{claim}\label{proof_claim_2}
If $(Y^y,W^y,\Omega^y,\cF^y,\bF^y,\bP^y)$ is a weak solution of the original SEE \eqref{intro_eq_SEE-lift} with a given initial condition $y\in\cH_\mu$, then the law of $Y^y$ is uniquely characterized as the weak limit of the laws of $\{Y^{k,y}\}_{k\in\bN}$, and the weak convergence is uniform with respect to the initial condition $y$ on each bounded subset of $\cH_\mu$.
\end{claim}

\cref{proof_claim_1} and \cref{proof_lemm_weak-existence} imply that, for any $y\in\cH_\mu$, weak existence holds for the original SEE \eqref{intro_eq_SEE-lift} with the initial condition $y$. Then, \cref{proof_claim_2} implies that uniqueness in law holds for the SEE \eqref{intro_eq_SEE-lift}, that the map $y\mapsto\Law_{\bP^y}(Y^y)$ is continuous, and that $Y^y$ is a time-homogeneous Markov process on $\cH_\mu$ which satisfies the Feller property.

\subsubsection{Proof of \cref{proof_claim_1}: Weak existence}

First, let us prove \cref{proof_claim_1}. Let an arbitrary filtered probability space $(\Omega,\cF,\bF,\bP)$ and a $d$-dimensional $\bF$-Brownian motion $W$ be given, and fix an arbitrary initial condition $y\in\cH_\mu$. For each $k\in\bN$, we denote by $Y^k$ the mild solution of the SEE \eqref{proof_eq_SEE-approximation} with the initial condition $y$. To show \cref{proof_claim_1}, take arbitrary $k,\ell\in\bN$ with $k<\ell$. Noting that $\displaystyle\sup_{x\in\bR^n}|b_\ell(x)-b_k(x)|^2\leq\Delta_{0,k}$ and $\displaystyle\sup_{x\in\bR^n}|\sigma_\ell(x)-\sigma_k(x)|^2\leq\Delta_{0,k}$, we will apply \cref{proof_prop_key-estimate} to
\begin{equation*}
	b=b_\ell,\ \sigma=\sigma_\ell,\ \bar{b}=b_k,\ \bar{\sigma}=\sigma_k,
\end{equation*}
\begin{equation*}
	m=m_k,\ M=M_\ell,\ \bar{M}=M_k,\ \lambda=\lambda_k,\ J=J_k,
\end{equation*}
and
\begin{equation*}
	\Delta_0=\Delta_{0,k},\ \Delta_1=\Delta_{1,k},\ \Delta_2=\Delta_{2,k},\ \Delta_3=\Delta_{3,k}.
\end{equation*}
By \cref{pre_prop_SEE-stop}, there exists a unique mild solution $\hat{Y}^{k,\ell}$ of the controlled SEE
\begin{equation*}
	\begin{dcases}
	\diff\hat{Y}^{k,\ell}_t(\theta)=-\theta\hat{Y}^{k,\ell}_t(\theta)\,\diff t+b_k(\mu_{M_k}[\hat{Y}^{k,\ell}_t])\,\diff t+\sigma_k(\mu_{M_k}[\hat{Y}^{k,\ell}_t])\,\diff W_t+\lambda_k\mu_{m_k}[Y^\ell_t-\hat{Y}^{k,\ell}_t]\1_{[0,\tau_{k,\ell}]}(t)\,\diff t,\\\
	\hspace{7cm}\theta\in\supp,\ t>0,\\
	\hat{Y}^{k,\ell}_0(\theta)=y(\theta),\ \ \theta\in\supp,
	\end{dcases}
\end{equation*}
with the stopping time $\tau_{k,\ell}$ defined by
\begin{equation*}
	\tau_{k,\ell}:=\inf\left\{t\geq0\relmiddle|
	\begin{aligned}
	\int^t_0\big|\mu_{m_k}[Y^\ell_s-\hat{Y}^{k,\ell}_s]\big|^2\,\diff s\geq\Delta_{1,k},\ \int^t_0\int_\supp\theta r_{m_k}(\theta)|Y^\ell_s(\theta)-\hat{Y}^{k,\ell}_s(\theta)|^2\dmu\,\diff s\geq\Delta_{2,k}\\
	\text{or}\ \|Y^\ell_t-\hat{Y}^{k,\ell}_t\|^2_{m_k}\geq\Delta_{3,k}
	\end{aligned}
	\right\}.
\end{equation*}
Define
\begin{equation*}
	\zeta_{k,\ell}:=\inf\left\{t\geq0\relmiddle|\int^t_0\|Y^\ell_s\|^2_{\cV_\mu}\,\diff s\geq J_k\right\}\wedge J_k
\end{equation*}
and
\begin{equation*}
	\hat{\Omega}_{k,\ell}=\left\{
	\begin{aligned}
	&\|Y^\ell_{t\wedge\tau_{k,\ell}\wedge\zeta_{k,\ell}}-\hat{Y}^{k,\ell}_{t\wedge\tau_{k,\ell}\wedge\zeta_{k,\ell}}\|^2_{m_k}+2\int^{t\wedge\tau_{k,\ell}\wedge\zeta_{k,\ell}}_0\int_\supp\theta r_{m_k}(\theta)|Y^\ell_s(\theta)-\hat{Y}^{k,\ell}_s(\theta)|^2\dmu\,\diff s\\
	&\hspace{5cm}+2(\lambda_k-1)\int^{t\wedge\tau_{k,\ell}\wedge\zeta_{k,\ell}}_0\big|\mu_{m_k}[Y^\ell_s-\hat{Y}^{k,\ell}_s]\big|^2\,\diff s\\
	&\leq\frac{1}{2}\Big(\rho_b(\ep_{m_k}^2)^{1/2}+R_{m_k}\frac{\rho_\sigma(\ep_{m_k}^2)}{\ep_{m_k}}\Big)^{1/2}\ep_{m_k}\hspace{1cm}\text{for any $t\geq0$}
	\end{aligned}
	\right\}.
\end{equation*}
Then, by \cref{proof_prop_key-estimate} and the estimates \eqref{proof_eq_estimate-A}, \eqref{proof_eq_estimate-B} and \eqref{proof_eq_estimate-C}, we obtain
\begin{equation}\label{proof_eq_estimate-prob}
	\bP(\Omega\setminus\hat{\Omega}_{k,\ell})\leq648\Big(\rho_b(\ep_{m_k}^2)^{1/2}+R_{m_k}\frac{\rho_\sigma(\ep_{m_k}^2)}{\ep_{m_k}}\Big)^{1/4}
\end{equation}
and
\begin{equation}\label{proof_eq_estimate-TV}
	d_\TV\big(\Law_\bP(\hat{Y}^{k,\ell}),\Law_\bP(Y^k)\big)\leq\frac{1}{2c_\UE^{1/2}}\Big\{\ep_{m_k}+\Big(\rho_b(\ep_{m_k}^2)^{1/2}+R_{m_k}\frac{\rho_\sigma(\ep_{m_k}^2)}{\ep_{m_k}}\Big)^{1/2}\Big\}.
\end{equation}
Observe that
\begin{equation*}
	\frac{1}{2}\Big(\rho_b(\ep_{m_k}^2)^{1/2}+R_{m_k}\frac{\rho_\sigma(\ep_{m_k}^2)}{\ep_{m_k}}\Big)^{1/2}\ep_{m_k}=\frac{1}{2}\Delta_{3,k}=(\lambda_k-1)\Delta_{1,k}\leq108^{-2}\Delta_{2,k}.
\end{equation*}
Thus, noting $\lambda_k>1$, by the definition of the stopping time $\tau_{k,\ell}$, we have
\begin{equation}\label{proof_eq_large-tau-1}
	\hat{\Omega}_{k,\ell}\subset\{\zeta_{k,\ell}<\tau_{k,\ell}\}.
\end{equation}	
Now we fix an arbitrary $T>0$. Recall the definition \eqref{sol_eq_norm-Lambda} of the norm $\|\cdot\|_{\Lambda_T}$. Noting that $r(\theta)\leq r_{m_k}(\theta)$, on the event $\hat{\Omega}_{k,\ell}\cap\{\zeta_{k,\ell}\geq T\}$, we have
\begin{align*}
	\|Y^\ell-\hat{Y}^{k,\ell}\|^2_{\Lambda_T}&=\sup_{t\in[0,T]}\|Y^\ell_t-\hat{Y}^{k,\ell}_t\|^2_{\cH_\mu}+\int^T_0\|Y^\ell_t-\hat{Y}^{k,\ell}_t\|^2_{\cV_\mu}\,\diff t\\
	&\leq(1+T)\sup_{t\in[0,T]}\|Y^\ell_t-\hat{Y}^{k,\ell}_t\|^2_{\cH_\mu}+\int^T_0\int_\supp\theta r(\theta)|Y^\ell_t(\theta)-\hat{Y}^{k,\ell}_t(\theta)|^2\dmu\,\diff t\\
	&\leq(1+T)\sup_{t\in[0,T]}\|Y^\ell_t-\hat{Y}^{k,\ell}_t\|^2_{m_k}+\int^T_0\int_\supp\theta r_{m_k}(\theta)|Y^\ell_t(\theta)-\hat{Y}^{k,\ell}_t(\theta)|^2\dmu\,\diff t\\
	&\leq\Big(1+T+\frac{1}{2}\Big)\frac{1}{2}\Big(\rho_b(\ep_{m_k}^2)^{1/2}+R_{m_k}\frac{\rho_\sigma(\ep_{m_k}^2)}{\ep_{m_k}}\Big)^{1/2}\ep_{m_k}\\
	&\leq(1+T)\Big(\rho_b(\ep_{m_k}^2)^{1/2}+R_{m_k}\frac{\rho_\sigma(\ep_{m_k}^2)}{\ep_{m_k}}\Big)^{1/2}\ep_{m_k}.
\end{align*}
Therefore, by \eqref{proof_eq_estimate-prob} and the definition of the stopping time $\zeta_{k,\ell}$,
\begin{align*}
	&\bP\Big(\|Y^\ell-\hat{Y}^{k,\ell}\|_{\Lambda_T}>(1+T)^{1/2}\Big(\rho_b(\ep_{m_k}^2)^{1/2}+R_{m_k}\frac{\rho_\sigma(\ep_{m_k}^2)}{\ep_{m_k}}\Big)^{1/4}\ep^{1/2}_{m_k}\Big)\\
	&\leq\bP(\Omega\setminus\hat{\Omega}_{k,\ell})+\bP(\zeta_{k,\ell}<T)\\
	&\leq648\Big(\rho_b(\ep_{m_k}^2)^{1/2}+R_{m_k}\frac{\rho_\sigma(\ep_{m_k}^2)}{\ep_{m_k}}\Big)^{1/4}+\bP\Big(\int^T_0\|Y^\ell_t\|^2_{\cV_\mu}\,\diff t\geq J_k\Big)+\1_{\{J_k<T\}}\\
	&\leq648\Big(\rho_b(\ep_{m_k}^2)^{1/2}+R_{m_k}\frac{\rho_\sigma(\ep_{m_k}^2)}{\ep_{m_k}}\Big)^{1/4}+J_k^{-1}\bE\Big[\int^T_0\|Y^\ell_t\|^2_{\cV_\mu}\,\diff t\Big]+\1_{\{J_k<T\}}.
\end{align*}
Noting \cref{proof_rem_LG}, by \cref{poof_lemm_estimate}, we have
\begin{equation*}
	\bE\Big[\int^T_0\|Y^\ell_t\|^2_{\cV_\mu}\,\diff t\Big]\leq C_0e^{C_0T}(1+\|y\|^2_{\cH_\mu}),
\end{equation*}
where $C_0=C_0(\mu,|b(0)|,|\sigma(0)|,\rho_b(1),\rho_\sigma(1))\geq1$ is a constant which does not depend on $k$, $\ell$, $T$, $y$ or the choice of the probability space. Hence,
\begin{align*}
	&\bP\Big(\|Y^\ell-\hat{Y}^{k,\ell}\|_{\Lambda_T}>(1+T)^{1/2}\Big(\rho_b(\ep_{m_k}^2)^{1/2}+R_{m_k}\frac{\rho_\sigma(\ep_{m_k}^2)}{\ep_{m_k}}\Big)^{1/4}\ep^{1/2}_{m_k}\Big)\\
	&\leq\Big\{648+C_0e^{C_0T}(1+\|y\|^2_{\cH_\mu})\Big\}\Big(\rho_b(\ep_{m_k}^2)^{1/2}+R_{m_k}\frac{\rho_\sigma(\ep_{m_k}^2)}{\ep_{m_k}}\Big)^{1/4}+\1_{\{J_k<T\}}.
\end{align*}
Recall the well-known fact that the L\'{e}vy--Prokhorov metric between two probability measures $\nu_1$ and $\nu_2$ on a Polish space $(E,d)$ is represented in terms of couplings between $\nu_1$ and $\nu_2$:
\begin{equation*}
	d_\LP(\nu_1,\nu_2)=\inf\left\{\ep>0\relmiddle|\inf_{\pi\in\sC(\nu_1,\nu_2)}\int_{E\times E}\1_{\{d(x_1,x_2)>\ep\}}\,\diff\pi(x_1,x_2)<\ep\right\},
\end{equation*}
where $\sC(\nu_1,\nu_2)$ denotes the set of probability measures on $E\times E$ with $\nu_1$ and $\nu_2$ as respective first and second marginals (cf.\ \cite{Du02}). By the above observations, we get
\begin{align}
	\nonumber&d_\LP\big(\Law_\bP(Y^\ell|_{[0,T]}),\Law_\bP(\hat{Y}^{k,\ell}|_{[0,T]})\big)\\
	\nonumber&\leq\inf\left\{\ep>0\relmiddle|\bP\Big(\|Y^\ell-\hat{Y}^{k,\ell}\|_{\Lambda_T}>\ep\Big)<\ep\right\}\\
	&\leq\Big\{648+C_0e^{C_0T}(1+\|y\|^2_{\cH_\mu})\Big\}\Big(\rho_b(\ep_{m_k}^2)^{1/2}+R_{m_k}\frac{\rho_\sigma(\ep_{m_k}^2)}{\ep_{m_k}}\Big)^{1/4}+\1_{\{J_k<T\}}.
	\label{proof_eq_estimate-LP}
\end{align}
Observe that
\begin{align*}
	&d_\LP\big(\Law_\bP(Y^\ell|_{[0,T]}),\Law_\bP(Y^k|_{[0,T]})\big)\\
	&\leq d_\LP\big(\Law_\bP(Y^\ell|_{[0,T]}),\Law_\bP(\hat{Y}^{k,\ell}|_{[0,T]})\big)+d_\LP(\Law_\bP(\hat{Y}^{k,\ell}|_{[0,T]}),\Law_\bP(Y^k|_{[0,T]})\big)\\
	&\leq d_\LP\big(\Law_\bP(Y^\ell|_{[0,T]}),\Law_\bP(\hat{Y}^{k,\ell}|_{[0,T]})\big)+d_\TV(\Law_\bP(\hat{Y}^{k,\ell}|_{[0,T]}),\Law_\bP(Y^k|_{[0,T]})\big)\\
	&\leq d_\LP\big(\Law_\bP(Y^\ell|_{[0,T]}),\Law_\bP(\hat{Y}^{k,\ell}|_{[0,T]})\big)+d_\TV(\Law_\bP(\hat{Y}^{k,\ell}),\Law_\bP(Y^k)\big),
\end{align*}
where the second inequality follows from the well-known inequality $d_\LP(\nu_1,\nu_2)\leq d_\TV(\nu_1,\nu_2)$ for any probability measures $\nu_1$ and $\nu_2$ on a Polish space. Therefore, by \eqref{proof_eq_estimate-LP} and \eqref{proof_eq_estimate-TV}, we obtain
\begin{align*}
	&d_\LP\big(\Law_\bP(Y^\ell|_{[0,T]}),\Law_\bP(Y^k|_{[0,T]})\big)\\
	&\leq\Big\{648+C_0e^{C_0T}(1+\|y\|^2_{\cH_\mu})\Big\}\Big(\rho_b(\ep_{m_k}^2)^{1/2}+R_{m_k}\frac{\rho_\sigma(\ep_{m_k}^2)}{\ep_{m_k}}\Big)^{1/4}+\1_{\{J_k<T\}}\\
	&\hspace{2cm}+\frac{1}{2c_\UE^{1/2}}\Big\{\ep_{m_k}+\Big(\rho_b(\ep_{m_k}^2)^{1/2}+R_{m_k}\frac{\rho_\sigma(\ep_{m_k}^2)}{\ep_{m_k}}\Big)^{1/2}\Big\}.
\end{align*}
Note that the right-hand side tends to zero as $k\to\infty$. Consequently, $\{\Law_\bP(Y^k|_{[0,T]})\}_{k\in\bN}$ is a Cauchy sequence in $(\cP(\Lambda_T),d_\LP)$ for any $T>0$. In particular, $\{\Law_\bP(Y^k)\}_{k\in\bN}$ is relatively weakly compact in $\cP(\Lambda)$. By Prokhorov's theorem, we see that $\{Y^k\}_{k\in\bN}$ is tight on the Polish space $\Lambda$. Then, by \cref{proof_lemm_weak-existence}, weak existence holds for the SEE \eqref{intro_eq_SEE-lift}.

\subsubsection{Proof of \cref{proof_claim_2}: Uniqueness in law, weak continuity and the Markov property}\label{proof_subsubsection_proof-of-ClaimB}

Next, we prove \cref{proof_claim_2}. This can be proved by a similar manner as in the proof of \cref{proof_claim_1}. Let $y\in\cH_\mu$ be given, and suppose that $(Y,W,\Omega,\cF,\bF,\bP)$ is a weak solution of the original SEE \eqref{intro_eq_SEE-lift} with the initial condition $y$. For each $k\in\bN$, again we denote by $Y^k$ the mild solution of the SEE \eqref{proof_eq_SEE-approximation} with the initial condition $y$, which is now defined for the Brownian motion $W$ on the filtered probability space $(\Omega,\cF,\bF,\bP)$ specified above. Fix $k\in\bN$. Here, we apply \cref{proof_prop_key-estimate} to the given coefficients $b,\sigma$ and
\begin{equation*}
	\bar{b}=b_k,\ \bar{\sigma}=\sigma_k,\ m=m_k,\ M=\infty,\ \bar{M}=M_k,\ \lambda=\lambda_k,\ J=J_k,
\end{equation*}
\begin{equation*}
	\Delta_0=\Delta_{0,k},\ \Delta_1=\Delta_{1,k},\ \Delta_2=\Delta_{2,k},\ \Delta_3=\Delta_{3,k}.
\end{equation*}
By \cref{pre_prop_SEE-stop}, there exists a unique mild solution $\hat{Y}^k$ of the controlled SEE
\begin{equation*}
	\begin{dcases}
	\diff\hat{Y}^k_t(\theta)=-\theta\hat{Y}^k_t(\theta)\,\diff t+b_k(\mu_{M_k}[\hat{Y}^k_t])\,\diff t+\sigma_k(\mu_{M_k}[\hat{Y}^k_t])\,\diff W_t+\lambda_k\mu_{m_k}[Y_t-\hat{Y}^k_t]\1_{[0,\tau_k]}(t)\,\diff t,\\
	\hspace{7cm}\theta\in\supp,\ t>0,\\
	\hat{Y}^k_0(\theta)=y(\theta),\ \ \theta\in\supp,
	\end{dcases}
\end{equation*}
with the stopping time $\tau_k$ defined by
\begin{equation*}
	\tau_k:=\inf\left\{t\geq0\relmiddle|
	\begin{aligned}
	\int^t_0\big|\mu_{m_k}[Y_s-\hat{Y}^k_s]\big|^2\,\diff s\geq\Delta_{1,k},\ \int^t_0\int_\supp\theta r_{m_k}(\theta)|Y_s(\theta)-\hat{Y}^k_s(\theta)|^2\dmu\,\diff s\geq\Delta_{2,k}\\
	\text{or}\ \|Y_t-\hat{Y}^k_t\|^2_{m_k}\geq\Delta_{3,k}
	\end{aligned}
	\right\}.
\end{equation*}
Define
\begin{equation*}
	\zeta_k:=\inf\left\{t\geq0\relmiddle|\int^t_0\|Y_s\|^2_{\cV_\mu}\,\diff s\geq J_k\right\}\wedge J_k
\end{equation*}
and
\begin{equation*}
	\hat{\Omega}_k=\left\{
	\begin{aligned}
	&\|Y_{t\wedge\tau_k\wedge\zeta_k}-\hat{Y}^k_{t\wedge\tau_k\wedge\zeta_k}\|^2_{m_k}+2\int^{t\wedge\tau_k\wedge\zeta_k}_0\int_\supp\theta r_{m_k}(\theta)|Y_s(\theta)-\hat{Y}^k_s(\theta)|^2\dmu\,\diff s\\
	&\hspace{5cm}+2(\lambda_k-1)\int^{t\wedge\tau_k\wedge\zeta_k}_0\big|\mu_{m_k}[Y_s-\hat{Y}^k_s]\big|^2\,\diff s\\
	&\leq\frac{1}{2}\Big(\rho_b(\ep_{m_k}^2)^{1/2}+R_{m_k}\frac{\rho_\sigma(\ep_{m_k}^2)}{\ep_{m_k}}\Big)^{1/2}\ep_{m_k}\hspace{1cm}\text{for any $t\geq0$}
	\end{aligned}
	\right\}.
\end{equation*}
Then, by \cref{proof_prop_key-estimate} and the estimates \eqref{proof_eq_estimate-A}, \eqref{proof_eq_estimate-B} and \eqref{proof_eq_estimate-C}, we obtain
\begin{equation*}
	\bP(\Omega\setminus\hat{\Omega}_k)\leq648\Big(\rho_b(\ep_{m_k}^2)^{1/2}+R_{m_k}\frac{\rho_\sigma(\ep_{m_k}^2)}{\ep_{m_k}}\Big)^{1/4}
\end{equation*}
and
\begin{equation*}
	d_\TV\big(\Law_\bP(\hat{Y}^k),\Law_\bP(Y^k)\big)\leq\frac{1}{2c_\UE^{1/2}}\Big\{\ep_{m_k}+\Big(\rho_b(\ep_{m_k}^2)^{1/2}+R_{m_k}\frac{\rho_\sigma(\ep_{m_k}^2)}{\ep_{m_k}}\Big)^{1/2}\Big\}.
\end{equation*}
Furthermore, by the same arguments as in the proof of \cref{proof_claim_1}, we get
\begin{equation}\label{proof_eq_large-tau-2}
	\hat{\Omega}_k\subset\{\zeta_k<\tau_k\}.
\end{equation}
Also, by \cref{poof_lemm_estimate}, we have
\begin{equation*}
	\bE\Big[\int^T_0\|Y_t\|^2_{\cV_\mu}\,\diff t\Big]\leq C_0e^{C_0T}(1+\|y\|^2_{\cH_\mu})
\end{equation*}
for any $T>0$, where $C_0=C_0(\mu,|b(0)|,|\sigma(0)|,\rho_b(1),\rho_\sigma(1))\geq1$ is a constant which does not depend on $T$, $y$ or the choice of the weak solution $(Y,W,\Omega,\cF,\bF,\bP)$. Then, by the same arguments as in the proof of \cref{proof_claim_1}, we can show that
\begin{align*}
	&d_\LP\big(\Law_\bP(Y|_{[0,T]}),\Law_\bP(Y^k|_{[0,T]})\big)\\
	&\leq\Big\{648+C_0e^{C_0T}(1+\|y\|^2_{\cH_\mu})\Big\}\Big(\rho_b(\ep_{m_k}^2)^{1/2}+R_{m_k}\frac{\rho_\sigma(\ep_{m_k}^2)}{\ep_{m_k}}\Big)^{1/4}+\1_{\{J_k<T\}}\\
	&\hspace{2cm}+\frac{1}{2c_\UE^{1/2}}\Big\{\ep_{m_k}+\Big(\rho_b(\ep_{m_k}^2)^{1/2}+R_{m_k}\frac{\rho_\sigma(\ep_{m_k}^2)}{\ep_{m_k}}\Big)^{1/2}\Big\}
\end{align*}
for any $k\in\bN$, $T>0$ and $y\in\cH_\mu$. Therefore, for any $T>0$, we have $\lim_{k\to\infty}Y^k|_{[0,T]}=Y|_{[0,T]}$ weakly on $\Lambda_T$, and the convergence is uniform with respect to the initial condition $y$ on each bounded subset of $\cH_\mu$. This implies that $\lim_{k\to\infty}Y^k=Y$ weakly on $\Lambda$ uniformly in $y$ on each bounded subset of $\cH_\mu$. This means that the law of $Y$ is specified as the weak limit of the sequence $\{Y^k\}_{k\in\bN}$ of the mild solutions of the approximating SEEs \eqref{proof_eq_SEE-approximation} which satisfy strong existence and pathwise uniqueness (and in particular uniqueness in law). Consequently, uniqueness in law holds for the original SEE \eqref{intro_eq_SEE-lift}. Furthermore, by the weak continuity of the map $y\mapsto Y^k$, together with the local uniformity in $y\in\cH_\mu$ of the weak convergence $\lim_{k\to\infty}Y^k=Y$, we see that $y\mapsto\Law_\bP(Y)$ is continuous as a map from $(\cH_\mu,\|\cdot\|_{\cH_\mu})$ to $(\cP(\Lambda),d_\LP)$. The Markov property of $Y$ follows from that of $Y^k$ and a standard approximation argument; see the last part of the proof of \cref{pre_lemm_SEE-M}. The Feller property follows from the weak continuity of $y\mapsto\Law_\bP(Y)$. This completes the proof of \cref{main_theo_main} in the ``singular case'' (i).

\subsection{Some remarks on the proof of the main theorem}\label{proof-remark}

Let us make some remarks on the proof of \cref{main_theo_main}. We will give a sketch of the proof of the ``regular case'' and discuss about the ``balance condition''.

\subsubsection{The regular case}\label{proof-remark-regular}

Now we consider the ``regular case'' (ii) of \cref{main_theo_main}. Namely, assume that the kernel $K$ is regular (i.e., $\lim_{t\downarrow0}K(t)<\infty$, or equivalently $\mu([0,\infty))<\infty$) and that the following holds:
\begin{equation*}
	\liminf_{\delta\downarrow0}\frac{\rho_\sigma(\delta^2)}{\delta}=0.
\end{equation*}
By the regularity of $K$, we can take $r\equiv1$ in \cref{sol_assum_kernel}. In this case, for any $m\in[1,\infty)$, we have $r_m(\theta)=1$ for any $\theta\in[0,\infty)$, and thus
\begin{align*}
	&R_m=\int_\supp r_m(\theta)\dmu=\mu([0,\infty))<\infty,\\
	&\ep_m=\int_\supp(1-r_m(\theta))^2\theta^{-1}r_m(\theta)^{-1}\dmu=0.
\end{align*}
Take a decreasing sequence $\{\delta_k\}_{k\in\bN}\subset(0,1]$ such that $\lim_{k\to\infty}\delta_k=0$ and $\lim_{k\to\infty}\frac{\rho_\sigma(\delta^2_k)}{\delta_k}=0$. For each $k\in\bN$, let $m_k=M_k=1$, and define $\Delta_{0,k},\Delta_{1,k},\Delta_{2,k},\Delta_{3,k}\in(0,\infty)$, $\lambda_k\in(1,\infty)$ and $J_k\in[1,\infty)$ by
\begin{equation*}
	\Delta_{0,k}=\rho_b(\delta^2_k)^{1/2}\delta_k+\rho_\sigma(\delta^2_k),\ \Delta_{1,k}=\delta^2_k,\ \Delta_{2,k}=1,\ \Delta_{3,k}=\Big(\rho_b(\delta^2_k)^{1/2}+\frac{\rho_\sigma(\delta^2_k)}{\delta_k}\Big)^{1/2}\delta_k,
\end{equation*}
and
\begin{equation*}
	\lambda_k=1+\Big(\rho_b(\delta^2_k)^{1/2}+\frac{\rho_\sigma(\delta^2_k)}{\delta_k}\Big)^{1/2}\delta^{-1}_k,\ J_k=\Big(\rho_b(\delta^2_k)^{1/2}+\frac{\rho_\sigma(\delta^2_k)}{\delta_k}\Big)^{-1/4}.
\end{equation*}
Construct suitable maps $b_k:\bR^n\to\bR^n$ and $\sigma_k:\bR^n\to\bR^{n\times d}$ by taking convolutions of $b$ and $\sigma$ with smooth mollifiers. Then, by the same arguments as in \cref{proof-main}, we can show that the conclusions of \cref{main_theo_main} hold in the regular case (ii).

\subsubsection{The balance condition}\label{proof-remark-balance}

Let us remark on the ``balance condition''
\begin{equation}\label{proof_eq_condition-balance}
	\liminf_{m\to\infty}R_m\frac{\rho_\sigma(\ep^2_m)}{\ep_m}=0
\end{equation}
appearing in the singular case (i) of \cref{main_theo_main}. Note that, in the singular case, $\ep_m>0$ for any $m\in[1,\infty)$, $\lim_{m\to\infty}R_m=\infty$ and $\lim_{m\to\infty}R_m\ep_m=0$. By means of \cref{proof_prop_key-estimate}, the key step in \cref{proof-main} is to construct sequences $\{m_k\}_{k\in\bN}\subset[1,\infty)$, $\{\lambda_k\}_{k\in\bN}\subset(1,\infty)$ and $\{\Delta_{i,k}\}_{k\in\bN}\subset(0,\infty)$, $i=1,2,3$, such that $\lim_{k\to\infty}m_k=\infty$ and that
\begin{equation}\label{proof_eq_condition-Delta}
	\begin{dcases}
	\Delta_{3,k}\leq C(\lambda_k-1)\Delta_{1,k}\ \text{for any $k\in\bN$ for some constant $C\in(0,\infty)$},\\
	\Delta_{3,k}\leq C\Delta_{2,k}\ \text{for any $k\in\bN$ for some constant $C\in(0,\infty)$},\\
	\lim_{k\to\infty}\Delta_{3,k}=0,\\
	\lim_{k\to\infty}R_{m_k}\Delta^{-1}_{3,k}\rho_\sigma(\Delta_{1,k}+\ep_{m_k}\Delta_{2,k})=0,\\
	\lim_{k\to\infty}\lambda_k\Delta^{1/2}_{1,k}=0.
	\end{dcases}
\end{equation}
Here, in view of the C-n-R strategy \eqref{proof_eq_CnR}, the first four lines of \eqref{proof_eq_condition-Delta} correspond to the ``control-step'' \eqref{proof_eq_estimate-control}, and the last line corresponds to the ``reimbursement-step'' \eqref{proof_eq_estimate-reimburse}. The first two lines are required in order to ensure that the stopping time $\tau$ is sufficiently large on the event $\hat{\Omega}$; see \eqref{proof_eq_large-tau-1} and \eqref{proof_eq_large-tau-2}.

We have shown in \cref{proof-main} that the balance condition \eqref{proof_eq_condition-balance} makes it possible to construct sequences satisfying all the requirements in \eqref{proof_eq_condition-Delta}. Conversely, if the diffusion coefficient $\sigma$ is of ``H\"{o}lder-type'', i.e., $0<\liminf_{t\downarrow0}\frac{\rho_\sigma(t)}{t^\gamma}\leq\limsup_{t\downarrow0}\frac{\rho_\sigma(t)}{t^\gamma}<\infty$ for some constant $\gamma\in(0,1]$, then the balance condition \eqref{proof_eq_condition-balance} is necessary for the existence of such sequences. Indeed, by easy calculations, we can show that \eqref{proof_eq_condition-Delta} implies that
\begin{equation*}
	\lim_{k\to\infty}\Delta_{1,k}=0,\ \lim_{k\to\infty}\ep_{m_k}\Delta_{2,k}=0,\ \lim_{k\to\infty}\Delta_{3,k}\Delta^{-1/2}_{1,k}=0,
\end{equation*}
and that
\begin{equation}\label{proof_eq_condition-gamma}
	\lim_{k\to\infty}R_{m_k}\Delta^{\gamma-1/2}_{1,k}=0,\ \lim_{k\to\infty}R_{m_k}\ep^\gamma_{m_k}\Delta^{\gamma-1}_{2,k}=0,\ \lim_{k\to\infty}R_{m_k}\Delta^\gamma_{1,k}\Delta^{-1}_{2,k}=0,\ \lim_{k\to\infty}R_{m_k}\ep^\gamma_{m_k}\Delta^\gamma_{2,k}\Delta^{-1/2}_{1,k}=0.
\end{equation}
Since $\lim_{k\to\infty}R_{m_k}=\infty$ and $\lim_{k\to\infty}\Delta_{1,k}=0$, the first equality in \eqref{proof_eq_condition-gamma} implies that $\gamma>\frac{1}{2}$. Now assume that the balance condition \eqref{proof_eq_condition-balance} does not hold; $\liminf_{m\to\infty}R_m\ep^{2\gamma-1}_m>0$. Then, noting that $\gamma\in(\frac{1}{2},1)$, the first, second and fourth equalities in \eqref{proof_eq_condition-gamma} imply that
\begin{equation*}
	\lim_{k\to\infty}\frac{\Delta_{1,k}}{\ep^2_{m_k}}=0,\ \lim_{k\to\infty}\frac{\ep_{m_k}}{\Delta_{2,k}}=0\ \text{and}\ \lim_{k\to\infty}\Big(\frac{\Delta_{2,k}}{\ep_{m_k}}\Big)^\gamma\Big(\frac{\ep^2_{m_k}}{\Delta_{1,k}}\Big)^{1/2}=0,
\end{equation*}
respectively, and we get a contradiction.

The above discussions indicate that the balance condition \eqref{proof_eq_condition-balance} is essential for the C-n-R strategy to prove \cref{main_theo_main} based on the estimates appearing in \cref{proof_prop_key-estimate}.


\appendix
\setcounter{theo}{0}
\setcounter{equation}{0}

\section{Appendix}\label{appendix}

In this appendix, we prove existence and uniqueness of the mild solutions of SEEs of the form \eqref{proof_eq_SEE-hatY} involving stopping times which depend on the solutions.


\begin{prop}\label{pre_prop_SEE-stop}
Suppose that \cref{sol_assum_kernel} holds. Let a filtered probability space $(\Omega,\cF,\bF,\bP)$ and a $d$-dimensional $\bF$-Brownian motion $W$ be fixed. Let $b_1,b_2:\Omega\times[0,\infty)\times\cH_\mu\to\bR^n$ and $\sigma_1,\sigma_2:\Omega\times[0,\infty)\times\cH_\mu\to\bR^{n\times d}$ be measurable maps satisfying the following conditions:
\begin{itemize}
\item
For $\varphi=b_1,b_2,\sigma_1,\sigma_2$, the process $(\varphi(t,y))_{t\geq0}$ is predictable for any $y\in\cH_\mu$.
\item
There exist a constant $c_\LG>0$ and a nonnegative predictable process $\xi$ satisfying $\bE[\int^T_0\xi^2_t\,\diff t]<\infty$ for any $T>0$ such that, for $\varphi=b_1,b_2,\sigma_1,\sigma_2$,
\begin{equation*}
	|\varphi(t,y)|\leq c_\LG(\xi_t+\|y\|_{\cH_\mu})
\end{equation*}
for any $t\geq0$ and any $y\in\cH_\mu$ a.s.
\item
There exist $L_k>0$, $k\in\bN$, such that, for $\varphi=b_1,b_2,\sigma_1,\sigma_2$,
\begin{equation*}
	|\varphi(t,y)-\varphi(t,\bar{y})|\leq L_k\|y-\bar{y}\|_{\cH_\mu}
\end{equation*}
for any $t\geq0$ and any $y,\bar{y}\in\cH_\mu$ satisfying $\|y\|_{\cH_\mu}\vee\|\bar{y}\|_{\cH_\mu}\leq k$ a.s., for any $k\in\bN$.
\end{itemize}
Also, let $f:\Omega\times[0,\infty)\times\cV_\mu\to\bR^{d_f}$ and $g:\Omega\times[0,\infty)\times\cH_\mu\to\bR^{d_g}$ with $d_f,d_g\in\bN$ be measurable maps satisfying the following conditions:
\begin{itemize}
\item
The process $(f(t,y))_{t\geq0}$ is predictable for any $y\in\cV_\mu$, and the process $(g(t,y))_{t\geq0}$ is predictable for any $y\in\cH_\mu$.
\item
There exist a constant $c_f>0$ and a nonnegative predictable process $\eta$ satisfying $\int^T_0\eta_t\,\diff t<\infty$ a.s.\ for any $T>0$ such that
\begin{equation*}
	|f(t,y)|\leq c_f(\eta_t+\|y\|^2_{\cV_\mu})
\end{equation*}
for any $t\geq0$ and any $y\in\cV_\mu$ a.s.
\item
For $\bP$-almost every $\omega\in\Omega$, the map $[0,\infty)\times\cH_\mu\ni(t,y)\mapsto g(\omega,t,y)\in\bR^{d_g}$ is continuous.
\end{itemize}
Let $O_f\subset\bR^{d_f}$ and $O_g\subset\bR^{d_g}$ be open sets. Then, for any initial condition $y\in\cH_\mu$, there exists a unique mild solution $Y$ to the following SEE:
\begin{equation}\label{pre_eq_SEE-stop}
	\begin{dcases}
	\diff Y_t(\theta)=-\theta Y_t(\theta)\,\diff t+\big\{b_1(t,Y_t)+b_2(t,Y_t)\1_{[0,\tau]}(t)\big\}\,\diff t+\big\{\sigma_1(t,Y_t)+\sigma_2(t,Y_t)\1_{[0,\tau]}(t)\big\}\,\diff W_t,\\
	\hspace{7cm}\theta\in\supp,\ t>0,\\
	Y_0(\theta)=y(\theta),\ \ \theta\in\supp.
	\end{dcases}
\end{equation}
with the stopping time $\tau=\tau(Y)$ given by
\begin{equation*}
	\tau=\inf\left\{t\geq0\relmiddle|\int^t_0f(s,Y_s)\,\diff s\notin O_f\ \text{or}\ g(t,Y_t)\notin O_g\right\}.
\end{equation*}
\end{prop}


\begin{proof}
By \cite[Theorem 2.18]{Ha23}, there exists a unique mild solution $\hat{Y}$ to the SEE
\begin{equation*}
	\begin{dcases}
	\diff\hat{Y}_t(\theta)=-\theta\hat{Y}_t(\theta)\,\diff t+\big\{b_1(t,\hat{Y}_t)+b_2(t,\hat{Y}_t)\big\}\,\diff t+\big\{\sigma_1(t,\hat{Y}_t)+\sigma_2(t,\hat{Y}_t)\big\}\,\diff W_t,\ \ \theta\in\supp,\ t>0,\\
	\hat{Y}_0(\theta)=y(\theta),\ \ \theta\in\supp.
	\end{dcases}
\end{equation*}
Define a stopping time $\hat{\tau}=\hat{\tau}(\hat{Y})$ by
\begin{equation*}
	\hat{\tau}:=\inf\left\{t\geq0\relmiddle|\int^t_0f(s,\hat{Y}_s)\,\diff s\notin O_f\ \text{or}\ g(t,\hat{Y}_t)\notin O_g\right\}.
\end{equation*}
Again by \cite[Theorem 2.18]{Ha23}, there exists a unique mild solution $Y$ to the SEE
\begin{equation*}
	\begin{dcases}
	\diff Y_t(\theta)=-\theta Y_t(\theta)\,\diff t+\big\{b_1(t,Y_t)+b_2(t,Y_t)\1_{[0,\hat{\tau}]}(t)\big\}\,\diff t+\big\{\sigma_1(t,Y_t)+\sigma_2(t,Y_t)\1_{[0,\hat{\tau}]}(t)\big\}\,\diff W_t,\\
	\hspace{7cm}\theta\in\supp,\ t>0,\\
	Y_0(\theta)=y(\theta),\ \ \theta\in\supp.
	\end{dcases}
\end{equation*}
Also, for any $k\in\bN$, there exists a unique mild solution $Y^k$ to the SEE
\begin{equation*}
	\begin{dcases}
	\diff Y^k_t(\theta)=-\theta Y^k_t(\theta)\,\diff t+\big\{b^k_1(t,Y^k_t)+b^k_2(t,Y^k_t)\big\}\,\diff t+\big\{\sigma^k_1(t,Y^k_t)+\sigma^k_2(t,Y_t)\big\}\,\diff W_t,\ \ \theta\in\supp,\ t>0,\\
	Y^k_0(\theta)=y(\theta),\ \ \theta\in\supp,
	\end{dcases}
\end{equation*}
where $b^k_1,b^k_2:\Omega\times[0,\infty)\times\cH_\mu\to\bR^n$ and $\sigma^k_1,\sigma^k_2:\Omega\times[0,\infty)\times\cH_\mu\to\bR^{n\times d}$ are defined by
\begin{equation*}
	\varphi^k(t,y):=
	\begin{dcases}
	\varphi(t,y)\ &\text{if $\|y\|_{\cH_\mu}<k$},\\
	\varphi\Big(t,\frac{k}{\|y\|_{\cH_\mu}}y\Big)\ &\text{if $\|y\|_{\cH_\mu}\geq k$},
	\end{dcases}
\end{equation*}
for $\varphi=b_1,b_2,\sigma_1,\sigma_2$. Note that, for $\varphi=b_1,b_2,\sigma_1,\sigma_2$,
\begin{equation*}
	|\varphi^k(t,y)-\varphi^k(t,\bar{y})|\leq L_k\|y-\bar{y}\|_{\cH_\mu}
\end{equation*}
for any $t\geq0$ and any $y,\bar{y}\in\cH_\mu$ a.s. Thus, by \cite[Theorem 2.17]{Ha23}, there exists a constant $C_k>0$ such that, for any $T>0$,
\begin{align*}
	&\bE\Big[\sup_{t\in[0,T]}\|Y^k_{t\wedge\hat{\tau}\wedge\zeta_k}-Y_{t\wedge\hat{\tau}\wedge\zeta_k}\|^2_{\cH_\mu}+\int^{T\wedge\hat{\tau}\wedge\zeta_k}_0\|Y^k_t-Y_t\|^2_{\cV_\mu}\,\diff t\Big]\\
	&\leq C_ke^{C_kT}\bE\Big[\int^{T\wedge\hat{\tau}\wedge\zeta_k}_0\Big\{|b^k_1(t,Y_t)+b^k_2(t,Y_t)-b_1(t,Y_t)-b_2(t,Y_t)|^2\\
	&\hspace{5cm}+|\sigma^k_1(t,Y_t)+\sigma^k_2(t,Y_t)-\sigma_1(t,Y_t)-\sigma_2(t,Y_t)|^2\Big\}\,\diff t\Big]
\end{align*}
and
\begin{align*}
	&\bE\Big[\sup_{t\in[0,T]}\|Y^k_{t\wedge\hat{\zeta}_k}-\hat{Y}_{t\wedge\hat{\zeta}_k}\|^2_{\cH_\mu}+\int^{T\wedge\hat{\zeta}_k}_0\|Y^k_t-\hat{Y}_t\|^2_{\cV_\mu}\,\diff t\Big]\\
	&\leq C_ke^{C_kT}\bE\Big[\int^{T\wedge\hat{\zeta}_k}_0\Big\{|b^k_1(t,\hat{Y}_t)+b^k_2(t,\hat{Y}_t)-b_1(t,\hat{Y}_t)-b_2(t,\hat{Y}_t)|^2\\
	&\hspace{5cm}+|\sigma^k_1(t,\hat{Y}_t)+\sigma^k_2(t,\hat{Y}_t)-\sigma(t,\hat{Y}_t)-\sigma_2(t,\hat{Y}_t)|^2\Big\}\,\diff t\Big],
\end{align*}
where $\zeta_k$ and $\hat{\zeta}_k$ are stopping times defined by
\begin{equation*}
	\zeta_k:=\inf\{t\geq0\,|\,\|Y_t\|_{\cH_\mu}\geq k\}\ \text{and}\ \hat{\zeta}_k:=\inf\{t\geq0\,|\,\|\hat{Y}_t\|_{\cH_\mu}\geq k\}.
\end{equation*}
By the definitions of $b^k_1,b^k_2$, $\sigma^k_1,\sigma^k_2$, $\zeta_k$ and $\hat{\zeta}_k$, we see that, for any $T>0$,
\begin{equation*}
	Y^k_t=Y_t\ \ \text{for any $t\in[0,T\wedge\hat{\tau}\wedge\zeta_k]$ a.s.}
\end{equation*}
and
\begin{equation*}
	Y^k_t=\hat{Y}_t\ \ \text{for any $t\in[0,T\wedge\hat{\zeta}_k]$ a.s.}
\end{equation*}
This implies that
\begin{equation*}
	\bP\big(Y_t=\hat{Y}_t\ \text{for any $t\in[0,T\wedge\hat{\tau}]$}\big)\geq\bP\big(\zeta_k>T\ \text{and}\ \hat{\zeta}_k>T\big)
\end{equation*}
for any $T>0$ and any $k\in\bN$. By the $\cH_\mu$-continuity of the processes $Y$ and $\hat{Y}$, the right-hand side tends to $1$ as $k\to\infty$. Thus, we see that $Y_t=\hat{Y}_t$ for any $t\in[0,T\wedge\hat{\tau}]$ a.s.\ for any $T>0$, which implies that $Y_t=\hat{Y}_t$ for any $t\in[0,\hat{\tau}]$ a.s. Now we define
\begin{equation*}
	\tau:=\inf\left\{t\geq0\relmiddle|\int^t_0f(s,Y_s)\,\diff s\notin O_f\ \text{or}\ g(t,Y_t)\notin O_g\right\}.
\end{equation*}
Noting that the processes $\int^\cdot_0 f(s,Y_s)\,\diff s$, $g(\cdot,Y_\cdot)$, $\int^\cdot_0f(s,\hat{Y}_s)\,\diff s$ and $g(\cdot,\hat{Y}_\cdot)$ are continuous and that $O_f\subset\bR^{d_f}$ and $O_g\subset\bR^{d_g}$ are open sets, we have
\begin{align*}
	\bP\big(\hat{\tau}<\tau\big)&=\bP\Big(\int^{\hat{\tau}}_0f(s,Y_s)\,\diff s\in O_f,\ g(\hat{\tau},Y_{\hat{\tau}})\in O_g\ \text{and}\ \hat{\tau}<\tau\Big)\\
	&\leq\bP\Big(\int^{\hat{\tau}}_0f(s,\hat{Y}_s)\,\diff s\in O_f,\ g(\hat{\tau},\hat{Y}_{\hat{\tau}})\in O_g\Big)=0
\end{align*}
and
\begin{align*}
	\bP\big(\tau<\hat{\tau}\big)&=\bP\Big(\int^\tau_0f(s,\hat{Y}_s)\,\diff s\in O_f,\ g(\tau,\hat{Y}_{\tau})\in O_g\ \text{and}\ \tau<\hat{\tau}\Big)\\
	&\leq\bP\Big(\int^\tau_0f(s,Y_s)\,\diff s\in O_f,\ g(\tau,Y_\tau)\in O_g\Big)=0.
\end{align*}
Thus, we see that $\tau=\hat{\tau}$ a.s. By the construction, the process $Y$ is a mild solution of the SEE \eqref{pre_eq_SEE-stop}.

Now we prove the uniqueness. Assume that there is another mild solution $\bar{Y}$ to the SEE \eqref{pre_eq_SEE-stop} with the stopping time $\tau$ replaced by
\begin{equation*}
	\bar{\tau}=\inf\left\{t\geq0\relmiddle|\int^t_0f(s,\bar{Y}_s)\,\diff s\notin O_f\ \text{or}\ g(t,\bar{Y}_t)\notin O_g\right\}.
\end{equation*}
By the same arguments as above, we see that $Y_t=\bar{Y}_t$ for any $t\in[0,\tau\wedge\bar{\tau}]$ a.s. This implies that $\tau=\bar{\tau}$ a.s. Therefore, $Y$ and $\bar{Y}$ solve the same SEE. By the uniqueness of the mild solution, we get $Y_t=\bar{Y}_t$ for any $t\geq0$ a.s. This completes the proof.
\end{proof}





\begin{thebibliography}{99}

\bibitem{AbiJa21}
E.\ Abi Jaber,
Weak existence and uniqueness for affine stochastic Volterra equations with $L^1$-kernels,
{\it Bernoulli},
27(3), 1583--1615,
2021.

\bibitem{AbiJaCuLaPu21}
E.\ Abi Jaber, C.\ Cuchiero, M.\ Larsson, and S.\ Pulido,
A weak solution theory for stochastic Volterra equations of convolution type,
{\it Ann.\ Appl.\ Probab.},
31(6), 2924--2952,
2021.

\bibitem{AbiJaEu19}
E.\ Abi Jaber and O.\ El Euch,
Multifactor approximation of rough volatility models,
{\it SIAM J.\ Financ.\ Math.},
10(2), 309--349,
2019.

\bibitem{AbiJaLaPu19}
E.\ Abi Jaber, M.\ Larsson, S.\ Pulido,
Affine Volterra processes,
{\it Ann.\ Appl.\ Probab.},
29(5), 3155--3200,
2019.

\bibitem{BaBeVe11}
O.\ Barndorff-Nielsen, F.\ Benth, and A.\ Veraart,
Ambit processes and stochastic partial differential equations,
In {\it Advanced Mathematical Methods for Finance, eds G.\ Di Nunno and B.\ {\O}ksendal},
35--74,
Springer, Berlin, Heidelberg,
2011.

\bibitem{BeMi80a}
M.A.\ Berger and V.J.\ Mizel,
Volterra equations with It\^o integrals---I,
{\it J. Integral Equations},
2(3), 187--245,
1980.

\bibitem{BeMi80b}
M.A.\ Berger and V.J.\ Mizel,
Volterra equations with It\^o integrals---II,
{\it J. Integral Equations},
2(4), 319--337,
1980.

\bibitem{BuKuSc20}
O.\ Butkovsky, A.\ Kulik, and M.\ Scheutzow,
Generalized couplings and ergodic rates for SPDEs and other Markov models,
{\it Ann. Appl. Probab.},
30, 1--39,
2020.

\bibitem{Ch03}
A.S.\ Cherny,
On the uniqueness in law and the pathwise uniqueness for stochastic differential equations,
{\it Theory Probab.\ Appl.},
46(3), 406--419,
2003.

\bibitem{CuTe19}
C.\ Cuchiero and J.\ Teichmann,
Markovian lifts of positive semidefinite affine Volterra-type processes,
{\it Decis.\ Econ.\ Finance},
42, 407--448,
2019.

\bibitem{CuTe20}
C.\ Cuchiero and J.\ Teichmann,
Generalized Feller processes and Markovian lifts of stochastic Volterra processes: the affine case,
{\it J.\ Evol.\ Equ.},
20, 1301--1348,
2020.

\bibitem{Du02}
R.M.\ Dudley,
{\it Real Analysis and Probability},
Cambridge University Press,
2002.

\bibitem{EuFuRo18}
O.\ El Euch, M.\ Fukasawa, and M.\ Rosenbaum,
The microstructural foundations of leverage effect and rough volatility,
{\it Finance Stoch.},
22, 241--280,
2018.

\bibitem{Fl11}
F.\ Flandoli,
{\it Random Perturbation of PDEs and Fluid Dynamic Models},
Springer, Berlin,
2011.

\bibitem{GrLoSt90}
G.\ Gripenberg, S.O.\ Londen, and O.\ Staffans,
{\it Volterra Integral and Functional Equations},
Cambridge University Press,
1990.

\bibitem{Ha23}
Y.\ Hamaguchi,
Markovian lifting and asymptotic log-Harnack inequality for stochastic Volterra integral equations,
{\it preprint},
arXiv:2304.06683,
2023.

\bibitem{JaSh03}
J.\ Jacod and A.N.\ Shiryaev,
{\it Limit Theorems for Stochastic Processes}, 2nd edition,
Springer, Berlin,
2003.

\bibitem{JaRo16}
T.\ Jaisson and M.\ Rosenbaum,
Rough fractional diffusions as scaling limits of nearly unstable heavy tailed Hawkes processes,
{\it Ann. Appl. Probab.},
26(5), 2860--2882,
2016.

\bibitem{KaSh91}
I.\ Karatzas and S.E.\ Shreve,
{\it Brownian Motion and Stochastic Calculus}, 2nd edition,
Springer, New York,
1991.

\bibitem{Kl18}
A.\ Kulik,
{\it Ergodic Behavior of Markov Processes: With Applications to Limit Theorems},
De Gruyter,
2018.

\bibitem{KlSc20}
A.\ Kulik and M.\ Scheutzow,
Well-posedness, stability and sensitivities for stochastic delay equations: a generalized coupling approach,
{\it Ann.\ Probab.},
48(6), 3041--3076,
2020.

\bibitem{Ku14}
T.\ Kurtz,
Weak and strong solutions of general stochastic models,
{\it Electron.\ Commun.\ Probab.},
19(58), 1--16,
2014.

\bibitem{MySa15}
L.\ Mytnik and T.S.\ Salisbury,
Uniqueness of Volterra-type stochastic integral equations,
{\it preprint},
arXiv:1502.05513,
2015.

\bibitem{PrSc22a}
D.J.\ Pr\"{o}mel and D.\ Scheffels,
On the existence of weak solutions to stochastic Volterra equations,
{\it preprint},
arXiv:2207.01367,
2022.

\bibitem{PrSc22b}
D.J.\ Pr\"{o}mel and D.\ Scheffels,
Pathwise uniqueness for singular stochastic Volterra equations with H\"{o}lder coefficients,
{\it preprint},
arXiv:2212.08029,
2022.

\bibitem{PrSc23}
D.J.\ Pr\"{o}mel and D.\ Scheffels,
Stochastic Volterra equations with H\"{o}lder diffusion coefficients,
{\it Stochastic Process.\ Appl.},
161, 291--315,
2023.

\bibitem{SaKiMa87}
S.G.\ Samko, A.A.\ Kilbas, and O.I.\ Marichev,
{\it Fractional Integrals and Derivatives, Theory and Applications},
Gordon and Breach Science Publishers, Yverdon, Switzerland,
1987.

\bibitem{ScSoVo12}
R.L.\ Schilling, R.\ Song, and Z.\ Vondra\v{c}ek,
{\it Bernstein Functions: Theory and Applications},
Walter de Gruyter,
2009.

\bibitem{StVa79}
D.W.\ Stroock and S.R.S.\ Varadhan,
{\it Multidimensional Diffusion Processes},
Springer, Berlin,
1979.

\bibitem{Ve12}
M.\ Veraar,
The stochastic Fubini theorem revisited,
{\it Stochastics},
84(4), 543--551,
2012.

\bibitem{Wa08}
Z.\ Wang,
Existence and uniqueness of solutions to stochastic Volterra equations with singular kernels and non-Lipschitz coefficients,
{\it Statist. Probab. Lett.},
78, 1062--1071,
2008.

\bibitem{YaWa71}
T.\ Yamada and S.\ Watanabe,
On the uniqueness of solutions of stochastic differential equations,
{\it J.\ Math.\ Kyoto Univ.},
11, 155--167,
1971.

\bibitem{Zh10}
X.\ Zhang,
Stochastic Volterra equations in Banach spaces and stochastic partial differential equation,
{\it J.\ Funct.\ Anal.},
258, 1361--1425,
2010.

\end{thebibliography}
\end{document}